\documentclass[10pt,a4paper,reqno]{article}

\usepackage[latin,english]{babel}
\usepackage[utf8]{inputenc}
\usepackage[T1]{fontenc}
\usepackage{eufrak, color, mathrsfs,dsfont}
\usepackage{times}
\usepackage{amsmath}
\usepackage{amssymb}
\usepackage{amsfonts}
\usepackage{amsthm}
\usepackage{enumitem}
\usepackage{stmaryrd}
\usepackage{tikz}
\usepackage{mathtools}
\usepackage{mathabx}
\usepackage{cases} 
\usepackage{braket}
\usepackage[toc,page]{appendix}
\usepackage{pdfsync}
\usepackage{hyperref}
\usepackage{psfrag}
\mathtoolsset{showonlyrefs}
\usepackage{math}
\usepackage[english]{babel}
\usepackage{amsmath}
\usepackage{amsthm}
\usepackage{amsfonts,mathrsfs,relsize,bigints,stmaryrd}
\usepackage{amssymb}
\usepackage{anysize}
\usepackage{hyperref}
\usepackage{appendix}
\usepackage{mathtools}
\usepackage{mathabx}
\usepackage{graphicx}
\usepackage{xcolor}
\usepackage{ulem}
\usepackage{dsfont} 
\usepackage{braket}
\usepackage[left=2cm,right=2cm,top=2cm,bottom=2.5cm]{geometry}

\numberwithin{equation}{section}

\newcommand\restr[2]{{
		\left.\kern-\nulldelimiterspace 
		#1 
		\vphantom{\big|} 
		\right|_{#2} 
}}

\allowdisplaybreaks

\theoremstyle{plain}
\newtheorem{lem}{Lemma}
\newtheorem{teo}[lem]{Theorem}
\newtheorem{prop}[lem]{Proposition}

\theoremstyle{definition}

\newtheorem{rmk}[lem]{Remark}

\renewcommand{\epsilon}{\varepsilon}

\renewcommand{\bar}{\overline}
\newcommand{\h}{\hbar}

\newcommand{\pa}{\partial}

\numberwithin{equation}{section}
\counterwithin{lem}{section}
\title{\bf Long-wave instability of periodic shear flows with constant magnetic field for the 2D resistive MHD equations}
\author{
Roberto Feola\footnote{
	Dipartimento di Matematica e Fisica, Universit\'a degli Studi Roma Tre, Largo San Leonardo Murialdo 1, 00146 Roma, Italy. 	\textit{Email:} \texttt{roberto.feola@uniroma3.it}; 
	}, \ \ Luca Franzoi\footnote{
	Dipartimento di Matematica, Universit\'a degli Studi di Roma Tor Vergata, Via della Ricerca Scientifica 1, 00133 Roma, Italy. \textit{Email:} \texttt{franzoi@mat.uniroma2.it}; 
	}, \ \ Riccardo Montalto\footnote{Dipartimento di Matematica ``Federigo Enriques'', Universit\'a degli Studi di Milano, Via Cesare Saldini 50, 20133 Milano, Italy. \textit{Email:} \texttt{riccardo.montalto@unimi.it}}, \ \ 
     Claudia Pe\~na\footnote{Basque Center for Applied Mathematics (BCAM), Alameda de Mazarredo 14, 48009, Bilbao, Spain – Universidad del País Vasco / Euskal Herriko Unibertsitatea (UPV/EHU), Barrio Sarriena s/n, 48940 Leioa, Spain. \textit{Email:} \texttt{cpena@bcamath.org}.
     }
}

\date{}
\begin{document}

\maketitle


\noindent
{\bf Abstract.} 
We investigate the long-wave linear stability and instability of the two-dimensional viscous, resistive Magnetohydrodynamic (MHD) equations, in vorticity-current formulation, on the periodic domain $\T_\alpha\times \T = \Big( \R/(\frac{2 \pi}{\alpha} \Z) \times \R/(2 \pi \Z) \Big)$, around  a periodic shear flow $(U(y),0)$ coupled with a constant background magnetic field $\bb=({\rm b}_1,{\rm b}_2)$. It is a non-trivial extension of a recent paper for the Navier-Stokes equations by Colombo, Dolce, Montalto \& Ventura to the MHD setting in the spirit of the classical works of Kolmogorov, Meshalkin, Sinai and Yudovich.
We establish explicit conditions on the shear flow profile $U(y)$ involving the viscosity $\nu$, the resistivity $\eta$ and the components of the background magnetic field $\bb$ to obtain linear long-wave stability and instability in the regime $\alpha\ll 1$. The proof combines a non-perturbative normal form transformation decoupling the zero Fourier mode from the non-zero modes with sharp asymptotic expansions of the  eigenvalues bifurcating from the zero unperturbed eigenvalue with respect to the parameter $\alpha$. As a dynamical consequence, we obtain a splitting of the phase space into unstable and stable subspaces, on which solutions grow or decay exponentially in Sobolev norm.

\medskip 



\smallskip 

\tableofcontents

\section{Introduction}

In this paper we consider the two dimensional resistive Magnetohydrodynamic (MHD) equations in vorticity-current formulation on a rectangular torus
\begin{equation}\label{eq:MHD_vort-curr}
	\begin{cases}
			\d_tw + u\cdot\nabla w - b\cdot\nabla\jmath = \nu\Delta w + f  \,,  & t\in \R \\
		\d_t\jmath + u\cdot\nabla \jmath - b\cdot\nabla w = \eta\Delta \jmath + 2H(u,b) \,,  & (\wtx,y) \in \T_{\alpha} \times \T\,, \\
		u=\nabla^\perp\phi\,, \quad   w=\Delta\phi \,, &  \\
		b=\nabla^\perp\psi \,,  \quad \jmath=\Delta\psi \,, & 
	\end{cases}
\end{equation}
where $u=u(t,\wtx,y)$ is the velocity field of the fluid,  $b=b(t,\wtx,y)$ is the magnetic field,   $w = w(t,\wtx,y) = \nabla^\perp \cdot u$ is the scalar vorticity, $\jmath=\jmath(t,\wtx,y)= \nabla^\perp \cdot b$ is the scalar current,
$\nu>0$ is the kinematic viscosity, $\eta > 0$ is the resistivity, $f=f(t,\wtx,y)$ is an external force acting on the fluid,  the nonlinearity is
\begin{equation}
	H(u,b) := \pa_{\wtx}b\cdot\nabla u_2 - \pa_{y}b\cdot\nabla u_1 \,,
\end{equation}
and $\T_\alpha=\R/(\frac{2\pi}{\alpha}\Z) \simeq [0,\frac{2\pi}{\alpha})$, with $\alpha<1$ being the inverse aspect ratio. In these coordinates, we have $\nabla=(\pa_{\wtx},\pa_{y})$, $\nabla^\perp=(-\pa_{y},\pa_{\wtx})$ and $\Delta=\pa_{\wtx}^2 + \pa_y^2$. This system describes the behaviour of a viscous, electrically conducting, incompressible fluid, with non-negligible resistivity, and it is commonly used to model strongly collisional plasmas in astrophysics. 
The first equation in \eqref{eq:MHD_vort-curr} is called the {\it momentum equation} and describes the Navier-Stokes-type evolution of the vorticity $w$, whereas the second one is sometimes called the {\it induction equation}, which is derived from Faraday's law and Ohm's law for a conducting fluid. We refer to \cite{Dav16} for a derivation of the model.

The goal of this work is to understand the stability properties of the equations \eqref{eq:MHD_vort-curr} 
near the steady vector field $(u(y), b)$, where $u(y)$ is the periodic shear flow
\begin{equation}\label{eq:def_uS}
	u(y):=(U(y),0) \,,  \quad \text{with scalar vorticity} \quad  \omega(y):=-U'(y) \,,
\end{equation}
(with $U:\T\to\R$ being sufficiently regular and $\int_{\T} U(y) \wrt y=0$) and $b$ is the constant background magnetic field
\begin{equation}\label{eq:def_bS}
	b:={\bb} := ({\rm b}_{1},{\rm b}_{2}) \in \R^2\,, \quad \text{with trivial scalar current} \quad  \overline{\jmath}=0 \,.
\end{equation}
We remark that $(u(y),\bb)$ is a stationary solution of the system \eqref{eq:MHD_vort-curr} only for a choice of the external force depending on the shear flow and when ${\rm b}_{2}=0$. Indeed, when \eqref{eq:def_uS}, \eqref{eq:def_bS} are inserted in \eqref{eq:MHD_vort-curr}, the first equation is solved by  fixing $f=f(y):= \nu U'''(y) $, whereas in the second equation an error term of the form ${\rm b}_{2} U''(y)$ is produced. For further details, see Remark \ref{rem.forces} later.

\subsection{Overview of the literature}

\paragraph{Well-posedness results.}
The mathematical aspects of the MHD equations have been extensively studied in recent decades, especially with regard to the Cauchy problem. We focus here on the 2D model. In  the resistive and viscous case, the global well-posedness  has been proven in  \cite{DL72} for finite-energy weak solutions, and in \cite{ST83} when the initial datum and the external force are smooth. For the non-resistive case, the local existence of smooth solutions in 2D has been proved in \cite{JN06} when the initial datum and the body force belong to the space $H^s$, with $s\geq 3$. Moreover, 2D global weak solutions   have been constructed in \cite{Koz89} for the system without viscosity, but with resistivity. In the ideal case, i.e. no diffusion in both equations, the local existence of strong solutions for initial data in $H^s(\R^d)$ with $s> d/2+1$, has been established in \cite{Schmidt88,Secchi93}.
Assuming that the initial datum $b_0$ is a small perturbation of a constant steady state, the existence of global-in-time smooth {\it small} solutions (namely, the velocity field is small and the magnetic field is close to a constant steady state) of \eqref{eq:MHD_vort-curr} has been established in \cite{BSS88,CL18,LXZ15,RWXZ14}. However, the existence of global-in-time smooth solutions for \eqref{eq:MHD_vort-curr} with generic periodic initial data (no perturbation of a constant steady state, big velocity) is still open even in 2D. For more references, including the 3D model, we refer to \cite{CMT24}.

\paragraph{Stability and instability of shear flows: hydrodynamics.} 
The study of the dynamical properties of perturbations of laminar solutions, such as shear flows, is a fundamental problem in hydrodynamic stability \cite{DR04}. Classical contributions were made by Kelvin, Rayleigh, Orr, Heisenberg, C.C. Lin, Tollmien, among many others, between the late nineteenth and early twentieth centuries. For the Navier–Stokes equations, fundamental stability issues are closely tied to the boundary conditions (for instance, no-slip boundary conditions for a physical boundary), leading to instability regimes, see e.g. \cite{BG23, GGN16-1,GGN16-2}. On the other hand, fluids with periodic boundary conditions and external forcing are frequently considered in physical applications, such as two-dimensional turbulence theories, geophysical models, and molecular dynamics, see e.g \cite{BE12}. Regarding the stability of shear flows in the long-wave regime, Kolmogorov \cite{AM60} proposed to investigate the behaviour of perturbations of the periodic shear flow $(\sin(y),0)$ varying the parameters $\alpha$ and $\nu$. The stability properties of this \textit{Kolmogorov flow} were first investigated by Meshalkin \& Sinai in \cite{MS61}. Later, Yudovich \cite{Yudo66} formally studied a long-wave instability mechanism for general shear flows, which only recently has been rigorously confirmed by Colombo, Dolce, Montalto \& Ventura in \cite{CDMV25}.

\noindent
In absence of the long-wave regime, other instability mechanisms have been studies for the 2D Euler equations, close to shear flows \cite{Lin03, LV18, LZ22} and to potential steady states \cite{Lin04,LLZ26}.

\noindent
We also mention (in the context of Water Waves equations) the stability problem for the Stokes waves, that is exact nonlinear traveling solutions of the water waves system: at the linear level, such solutions have rigorously been proved  in 2D to be modulational unstable \cite{BCMV24,BMV22,BMV23,BMV24,NS23}, and in 3D to be transversal unstable \cite{CNS25,CNS26} and longitudinal-transversal unstable \cite{BMR26}.

\medskip

\paragraph{Stability and instability of shear flows: magneto-hydrodynamics.}
The problem of the instability of hydromagnetic shear flows, in which the additional influence of a magnetic field is taken into account, exhibits an even richer phenomenology and has also received considerable attention owing, at least in part, to its importance in a number of astrophysical contexts: for instance, see \cite{DR04,Cha13} and the references in \cite{HT01}. 
Since the early work of Rayleigh on shear-flow-driven instabilities in fluids, it has been known that shear flows are usually unstable only in the presence of an inflection point in the velocity profile. However, Tatsuno \& Dorland \cite{TD06} showed with numerical simulations that in ideal magnetohydrodynamics, in particular in tokamak experiments, there are important instabilities for which no inflection point of the shear is required and that the instability mechanism is driven by reversed magnetic shear.
There are also other counterexamples to classical hydrodynamic results for magnetic fluids, such as to Miles-Howard theorem \cite{LZTH10} (a stratified, inviscid parallel shear flow is linearly stable to small perturbations as long as the local gradient Richardson number ${\rm Ri}$ is everywhere greater than $1/4$) and to  Howard's semicircle theorem \cite{HT01} (a geometric limitation for the eigenvalues of the linearized problem). 
In the non-ideal setting, with finite but non-zero viscosity and resistivity, there are also results that manifest the variety of shear flow instabilities: Fraser, Cresswell \& Garaud \cite{FCG22} 
studied the formation of instabilities for the 3D MHD equations around the shear flow $(0,0,\sin(x))$ and the parallel, constant magnetic field $(0,0,1)$, extending the hydrodynamic Kelvin-Helmholtz instability with two extra unstable modes.

\medskip

The stability of the Couette flow is of great relevance in the context of the MHD equations as well. Liss \cite{Lis20} proved the first nonlinear stability threshold result for the viscous and resistive 3D MHD system with small initial data of size $\varepsilon\ll \nu$ in Sobolev spaces, in the regime $\eta=\nu \in (0,1]$. In the 2D setting, Dolce \cite{Dol23} proved an analogous result with small initial data of size $\varepsilon\ll \nu^{\frac23}$ and in the regime $\eta^2 \lesssim {\rm Pr}_{\rm m} \leq 1$, where ${\rm Pr}_{\rm m} := \nu/\eta$ is the magnetic Prandtl number.
The Sobolev stability of 2D MHD equations near a combination of the Couette flow and a parallel, constant magnetic field has been proved by Knobel \& Zillinger \cite{KZ25} (with large magnetic field, full viscosity and only horizontal magnetic dissipation), by Knobel \cite{Kno25} (assuming a viscosity dominant regime $0<\eta\leq\nu$), by Chen \& Zi \cite{CZ25} (with large magnetic field and in the regime $\nu=\eta\in (0,1]$) and by Dolce, Knobel \& Zillinger \cite{DKZ26} (without resistivity). In particular, the linear norm inflation of the linear solutions is shown in \cite{Kno25,KZ25}, whereas in \cite{DKZ26} the authors prove the nonlinear norm inflation of the current over a long time of existence.  Finally, we conclude by mentioning some recent results on the linear kinematic dynamo equation \cite{CzSV25,CzSV26}, where 3D  divergence-free velocity fields (time-independent in $\R^3$, time-periodic in $\T^3$) are constructed in order to generate an unstable dynamics for the magnetic field.

\subsection{Main results}

To state the main result, we first consider perturbations $(\Omega(t,x,y), J(t,x,y))$ of  $(\omega(y),0)$ of the form
\begin{equation}
	w(t,\wtx,y) := \omega(y) + \Omega(t,\alpha \wtx,y) \,, \quad \jmath(t,\wtx,y) := J(t,\alpha \wtx,y) \,,
\end{equation}
so that $(\Omega(t,x,y), J(t,x,y))$ are $2\pi$-periodic with respect to $x\in \T:= \R/\Z$. 
Choosing as the external force
\begin{equation}\label{fix.force}
    f=f(y) = \nu U'''(y) \,,
\end{equation}
the equations satisfied by the variables $(\Omega(t,x,y), J(t,x,y))$ are
\begin{equation}
	\begin{cases}
		\d_t\Omega + \alpha U(y)\pa_x\Omega - \alpha U''(y)\pa_x\Delta_{\alpha}^{-1}\Omega- {\bb}\cdot\nabla_\alpha J + \cN_\Omega(\Omega,J) = \nu\Delta_\alpha\Omega \,, \\
		\d_tJ + \alpha U(y)\pa_xJ - {\bb}\cdot\nabla_\alpha\Omega + \alpha U''(y)\pa_x\Delta_{\alpha}^{-1}J + \cN_J(\Omega,J) = \eta \Delta_\alpha J  - {\rm b}_2U''(y) \,, \\
		u_\Omega=\nabla_{\alpha}^\perp\phi_\Omega\,, \quad   \Omega=\Delta_\alpha\phi_\Omega\,, \\
		b_J=\nabla_{\alpha}^\perp\psi_J \,,  \quad J=\Delta_\alpha\psi_J \,,
	\end{cases}
\end{equation}
for $t\in\R$ and $(x,y)\in\T^2$, where we have denoted by
\begin{align}
 \cN_\Omega(\Omega,J):=u_\Omega\cdot\nabla_{\alpha}\Omega-b_J\cdot\nabla_{\alpha}J   \,, \quad \cN_J(\Omega,J):=u_\Omega\cdot\nabla_{\alpha}J-b_J\cdot\nabla_{\alpha}\Omega-2H_\alpha(u_\Omega,b_\Omega) \,, 
\end{align}
with $\nabla_{\alpha}:= (\alpha \pa_x , \pa_y)$ and $\Delta_{\alpha}:= \alpha^2 \pa_{x}^2 + \pa_{y}^2$ for $(x,y)\in\T^2$ after the change of variables. The system can be also written in matrix form as 
$\pa_t(\Omega,J)=F(\Omega,J)$, where
\begin{equation}\label{eq:def_F}
	F(\Omega,J):=\begin{pmatrix}
	-\alpha U\pa_x\Omega + \alpha U''\pa_x\Delta_\alpha^{-1}\Omega + {\bb}\cdot\nabla_{\alpha}J + \nu\Delta_{\alpha}\Omega  \\
	-\alpha U\pa_xJ + {\bb}\cdot\nabla_\alpha\Omega - \alpha U''\pa_x\Delta_{\alpha}^{-1}J  + \eta\Delta_\alpha J 
	\end{pmatrix} - \begin{pmatrix}
	    \cN_\Omega(\Omega,J) \\ \cN_J(\Omega,J)
	\end{pmatrix}  + \begin{pmatrix}
	    0 \\  - {\rm b}_2U''(y)
	\end{pmatrix}.
\end{equation}
We consider the linearized operator corresponding to the linearization of \eqref{eq:MHD_vort-curr} around the stationary state $\big( (U(y),0), \bb \big)$  in \eqref{eq:def_uS}-\eqref{eq:def_bS} (or equivalently, by neglecting the nonlinear terms and  the error term in \eqref{eq:def_F}), which is given by 
\begin{equation}\label{cL.numu}
	\cL_{\nu,\eta} := \begin{pmatrix}
		\nu \Delta_{\alpha} - \alpha\Big( U(y) + U''(y) (-\Delta_{\alpha})^{-1}  \Big)\pa_{x}   & {\bb}\cdot \nabla_{\alpha} \\
		{\bb}\cdot \nabla_{\alpha} & \eta\Delta_{\alpha} - \alpha\Big( U(y) - \big(U''(y) +2U'(y)\pa_{y} \big)(-\Delta_{\alpha})^{-1}   \Big)\pa_{x}  
	\end{pmatrix}
\end{equation}
We observe that the frequencies with respect to $x$ are decoupled in $\cL_{\nu,\eta}$. That is, on the subspace of the plane waves $\{ f(y)e^{\im kx} \}$, with $k\in\Z\setminus\{0\}$, the action of $\cL_{\nu,\eta}$ in \eqref{cL.numu} becomes
\begin{equation}
    \cL_{\nu,\eta}[f(y)e^{\im k x}] = \big( \cL_{\nu,\eta,\varepsilon}[f(y)] \big) e^{\im kx} \,,
\end{equation}
where the linear operator $\cL_{\nu,\eta,\varepsilon}$ is given by
\begin{equation}\label{cL.vare}
\scalebox{0.9}{$
    \cL_{\nu,\eta,\varepsilon} := \begin{pmatrix}
		\nu (\pa_{y}^2-\varepsilon^2)- \im\varepsilon\big( U(y) + U''(y) (\varepsilon^2-\pa_{y}^2)^{-1}  \big)   & \im\varepsilon{\rm b}_1 + {\rm b}_2\pa_{y} \\ 
		\im\varepsilon{\rm b}_1+ {\rm b}_2\pa_{y} & \eta(\pa_{y}^2 -\varepsilon^2)- \im\varepsilon \big(U(y) - (U''(y) + 2U'(y)\pa_{y}) (\varepsilon^2-\pa_{y}^2)^{-1} \big)
	\end{pmatrix},
$}
\end{equation}
where we have denoted
\begin{equation}\label{vare.def}
    \varepsilon := \alpha k  \,.
\end{equation}
Without loss of generality, we may assume $k>0$. Indeed, it is not hard to check that for any
eigenvalue $\lambda$ of $\cL_{\nu,\eta,\alpha k}$ with $k>0$, then $\overline{\lambda}$ is an eigenvalue of $\cL_{\nu,\eta,-\alpha k}$.
\\
For $s\geq 0$, we define the Sobolev space
\begin{equation}\label{spaziFunz.intro}
\begin{aligned}
   H^s :=  H^s(\T) &:= \Big\{ u(y)=\sum_{j\in\Z} \whu(j) e^{\im j y} \in L^2(\T) \, : \, \| u \|_{s}^2 := \sum_{j\in\Z}\la j\ra^{2s} |\whu (j)|^2 < \infty \Big\} \,,  \\
   H_0^s:= H_0^s(\T) &:= \Big\{ u\in H^s(\T)  \, : \,\int_{\T} u(y)\wrt y = 0 \Big\}\,, 
\end{aligned}
\end{equation}
and, for $u = (u_1, u_2) \in H^s \times H^s$, we write $\| u \|_s := \| u_1 \|_s + \| u_2 \|_s$. For any $u\in H_0^s(\T) $ and for any $k \in \N$, we also define $\pa_{y}^{-k} u(y) := \sum_{j\in \Z\setminus\{0\}} (\im j)^{-k} \whu(j) e^{\im j y}$. Let also $ {\rm Id}_{\perp} := {\rm Id}_{H_0^s}$.
\\[1mm]
\noindent
The first main result of this paper shows the presences of unstable eigenvalues for the operator $\cL_{\nu,\eta,\varepsilon}$ in \eqref{cL.vare}
\begin{teo}\label{theo.princ}
{\bf (Spectral long-wave instability/stability).}
    Let $\nu,\eta>0$ and consider $u(y)=\big( \begin{smallmatrix}
        U(y) \\ 0
    \end{smallmatrix} \big)$ and $b=\big(\begin{smallmatrix}
        {\rm b}_1\\ {\rm b}_2
    \end{smallmatrix}\big)$, assuming that $U \in C^{S+2}(\T,\R)$ for 
    some integer $S\geq 1$. Denote by $\omega(y)=-U'(y)$ and $\overline{\jmath}=0$ 
    the corresponding vorticity and current fields.
    Let $\cL_{\nu,\eta}$ in \eqref{cL.numu} be the linearized operator  at $(\omega,\overline{\jmath})$ and let $\cL_{\nu,\eta,\varepsilon}$ in \eqref{cL.vare},  $\cL_{\nu,\eta,\varepsilon}:H^{s + 2}(\T) \times H^{s + 2}(\T)\to H^s(\T) \times H^s(\T)$, $0 \leq s \leq S$ be its restriction to the subspace of the plane waves $\{ f(y)e^{\im kx} \}$ for some fixed $k\in\Z\setminus\{0\}$. Let $\varepsilon$ as in \eqref{vare.def}.
    There exists $\varepsilon_0 \equiv \varepsilon_{0}(S,\nu,\eta,\bb) \in (0,1)$ small enough such that, for any $\varepsilon \in (0,\varepsilon_0)$, the spectrum $\sigma(\cL_{\nu,\eta,\varepsilon})$ of $\cL_{\nu,\eta,\varepsilon}$ is discrete and the following hold:
    \begin{itemize}
        \item[(i)] When ${\rm b}_1\neq0$, if 
        \begin{equation}\label{cond shear instab 1}
            \eta > \nu \quad \textit{ and } \quad \la \big( {\rm Id}_\perp - \tfrac{{\rm b}_2^2}{\nu\eta} \pa_{y}^{-2} \big)^{-1}\pa_y^{-1}U,\pa_y^{-1}U\ra_{L^2}>\frac{(\nu+\eta)\nu\eta}{\eta-\nu} \,,
        \end{equation}
         there exist two simple eigenvalues $\mu_{-},\mu_{+} \in \C\setminus\{0\}$ such that ${\rm Re}(\mu_{\pm})>0$, whereas the remaining eigenvalues in $\sigma(\cL_{\nu,\eta,\varepsilon})$ have negative real part;
    \item [(ii)] When ${\rm b}_1=0$, if 
    \begin{equation}\label{cond shear instab 2}
        \la \big( {\rm Id}_\perp - \tfrac{{\rm b}_2^2}{\nu\eta} \pa_{y}^{-2} \big)^{-1}\pa_y^{-1}U,\pa_y^{-1}U\ra_{L^2}>\nu^2 \,,
    \end{equation}
    there exists exactly one simple eigenvalue $\mu_{+} \in \C\setminus\{0\}$  such that ${\rm Re}(\mu_{+})>0$,  whereas the remaining eigenvalues in $\sigma(\cL_{\nu,\eta,\varepsilon})$ have negative real part.
     \item[(iii)] When ${\rm b}_1\neq0$, if $\nu \geq \eta$ all the eigenvalues in $\sigma(\cL_{\nu,\eta,\varepsilon})$ have negative real part. 
    \end{itemize}
\end{teo}
\begin{rmk}
We make some remarks on the conditions \eqref{cond shear instab 1}, \eqref{cond shear instab 2}. First of all, if the two inequalities hold with opposite sign, then in the cases $(i)$ and $(ii)$ of the latter theorem one has that the all the eigenvalues have negative real parts. We omit the analysis of these cases since it is similar to the one of item $(iii)$. 

\noindent
We also note that the condition \eqref{cond shear instab 1} is satisfied if
$$
\| \partial_y^{- 1} U \|_{L^2}^2 > C_0(\eta, \nu, {\rm b}_2)
$$
for some constant $C_0(\eta, \nu, {\rm b}_2) > 0$. Indeed, by setting 
$$
\gamma (\eta, \nu, {\rm b}_2) := \inf_{j \in \Z \setminus \{ 0 \}} \dfrac{1}{1 + \frac{{\rm b}_2^2}{\nu \eta j^2}} > 0 \,,
$$
one has 
$$
\la \big( {\rm Id}_\perp - \tfrac{{\rm b}_2^2}{\nu\eta} \pa_{y}^{-2} \big)^{-1}\pa_y^{-1}U,\pa_y^{-1}U\ra_{L^2} \geq \gamma (\eta, \nu, {\rm b}_2) \| \partial_y^{- 1} U \|_{L^2}^2\,.
$$
Hence, \eqref{cond shear instab 1} is satisfied provided that 
$$
\| \partial_y^{- 1} U \|_{L^2}^2 > \frac{(\nu+\eta)\nu\eta}{\gamma (\eta, \nu, {\rm b}_2)(\eta-\nu)} =: C_0(\eta, \nu, {\rm b}_2)\,.
$$
Similar arguments hold for the condition \eqref{cond shear instab 2}. 

     \end{rmk}
In the second main result, we show the existence of initial data such that their evolution under the flow generated by $\cL_{\nu,\eta,\varepsilon}$ exhibits for positive times either unbounded growth or dissipation in Sobolev topology.
\begin{teo}
    {\bf (Dynamical behaviour).}\label{thm:dinamico}
    Under the assumptions of Theorem \ref{theo.princ}, we consider the Cauchy problem
    \begin{equation}\label{sistemaOrig}
        \begin{cases}
            \pa_{t} u(t) = \cL_{\nu,\eta,\varepsilon} u(t) \,, \\
            u(0) = \vf \in H^s \times H^s \,,
        \end{cases}
    \end{equation}
    with $0\leq s\leq S$. The following hold:
    \begin{itemize}
\item[(i)] Under the assumptions of the items (i) of Theorem \ref{theo.princ}, we have the splitting of the phase space
\begin{equation}\label{stable.unstable.dec}
    H^s(\T) \times H^s(\T) = \cU_{\varepsilon} \oplus \cS_{\varepsilon} \,
\end{equation}
where ${\mathcal S}_\e \subset H^s(\T)$ is an infinite dimensional subspace and ${\mathcal U}_\e \subset H^s(\T)$ is a two-dimensional subspace such that:
\begin{itemize}
    \item for any $\varphi\in \cU_\varepsilon$ the corresponding solution $u(t)$ of \eqref{sistemaOrig} satisfies the estimate
\begin{equation}\label{unstable.est.thm}
    \|u(t)\|_{s} \gtrsim e^{\mu \e^2 t} \| \vf \|_s  \,, 
\quad \forall \, t \geq 0 \,,
\end{equation}
for some constant $\mu >0$;
\item for any $\varphi\in \cS_\varepsilon$ the corresponding solution $u(t)$ of \eqref{sistemaOrig} satisfies the estimate
\begin{equation}\label{stable.est.thm}
    \|u(t)\|_{s}\lesssim e^{-\mu_0 t}\|\varphi\|_{s}\,,
\quad \forall\, t\geq0\,,
\end{equation}
for some constant $\mu_0 >0$; 
\end{itemize}
\item[(ii)] Under the assumptions of the item (ii) of Theorem \ref{theo.princ}, we have the splitting of the phase space
\begin{equation}\label{stable.unstable.deccaso2}
    H^s(\T) \times H^s(\T) = \cU_{\varepsilon} \oplus \cS_{\varepsilon} \,
\end{equation}
where  ${\mathcal S}_\e \subset H^s(\T)$ is an infinite dimensional subspace and ${\mathcal U}_\e \subset H^s(\T)$ is a one dimensional subspace such that: 
\begin{itemize}
    \item for any $\varphi\in \cU_\varepsilon$ the corresponding solution $u(t)$ of \eqref{sistemaOrig} satisfies the estimate
\begin{equation}\label{unstable.est.thmcaso2}
    \|u(t)\|_{s} \gtrsim e^{\mu_1 \e^2 t} \| \vf \|_s \,, 
\quad \forall \, t > 0 \,,
\end{equation}
for some constant $\mu_1  >0$;
\item for any $\varphi\in \cS_\varepsilon$ the corresponding solution $u(t)$ of \eqref{sistemaOrig} satisfies the estimate
\begin{equation}\label{stable.est.thmcaso2}
    \|u(t)\|_{s}\lesssim \e^{-1} e^{- \mu_2 \e^2 t}\|\varphi\|_{s}\,,
\quad \forall\, t \geq 0\,,
\end{equation}
for some constant  $\mu_2 >0$; 
\end{itemize}
\item[(iii)] Under the assumptions of the item (iii) of Theorem \ref{theo.princ}, one has that there exists a constant $\mu_3 > 0$ such that for any $\vf \in H^s \times H^s $, 
\begin{equation}\label{stima b1 neq 0 caso stabile}
\| u(t) \|_{s} \lesssim \e^{-1} e^{- \mu_3 \e^2 t} \| \vf \|_{s}, \quad \forall \, t \geq 0\,,
\end{equation}
for some constant $\mu_3>0$.
\end{itemize}
\end{teo}

\begin{rmk} \label{rem.comparison}
{\bf (Comparison of the long-wave instability between Navier-Stokes and MHD).}
    The linear long wave instability for the two dimensional Navier-Stokes equation around the periodic shear flow $(U(y),0)$ in \eqref{eq:def_uS} has been studied in \cite{CDMV25}, extending in a more general case the stability of the periodic shear flow $(\sin(y),0)$ as proposed by Kolmogorov in \cite{AM60} and first investigated by Meshalkin \& Sinai \cite{MS61}.  
     When ${\rm b}_1 = {\rm b}_2 = 0$,  we recover the Navier-Stokes instability result from \cite{CDMV25} under the condition on the shear flow $\| \pa_{y}^{-1}U \|_{L^2} >  \nu$, since  the components of the linear operator $\cL_{\nu,\eta,\varepsilon}$ decouples into the Orr-Sommerfeld operator from the linearization of the momentum equation at the stationary state and an analogous one associated to the resistive induction equation. On the other hand, when ${\rm b}_{1}\neq 0$, we detect small unstable eigenvalues only in the resistivity-dominant regime $\eta >\nu$, whereas in the case $\nu \leq \eta$ all the eigenvalues of $\cL_{\nu,\eta,\varepsilon}$ are stable for $\varepsilon\ll 1$ small enough, as stated in items $(i)$ and $(iii)$ of Theorem \ref{theo.princ}. The precise expansion of the unstable eigenvalues will be obtained in Proposition \ref{prop:positive.zeromode}.
     \end{rmk}

  \begin{rmk} \label{rem.forces}
  {\bf (Physically relevant external forces).}
    In the Navier-Stokes framework, it is physically meaningful to study the stability of the shear flow $(U(y),0)$, since it 
    is a steady solution of 
 \begin{equation}\label{eq:NS_vort-curr}
 	\begin{cases}
 		\pa_t w + u\cdot\nabla w = \nu\Delta w + f \,, \\
 		u=\nabla^\perp\phi\,, \quad w=\Delta\phi \,,
 	\end{cases}  \textnormal{with external force} \quad f(\wtx,y)=\nu U'''(y) \,. 
 \end{equation}
For the MHD equations \eqref{eq:MHD_vort-curr}, it is not always the case that the stationary state $\big( (U(y),0), \bb \big)$  in \eqref{eq:def_uS}-\eqref{eq:def_bS} is also a steady solution. Indeed, as explained before, the momentum equation for $w$ is solved by $u= (U(y),0)$ and $b=\bb$ choosing the same forcing term (see \eqref{fix.force}), but, when solving the induction equation for $\jmath$, the error term
\begin{equation}\label{error.induction}
	{\rm b}_2U''(y) 
\end{equation}
is produces, which vanishes only if we either assume $U''=0$, which trivialize to $U(y)\equiv 0$ by the zero average condition and it is therefore not interesting, or ${\rm b}_2=0$. The problem that we face with the MHD system is that we cannot consider the option of introducing a forcing term in  the induction equation for $\jmath$ to compensate the error term in \eqref{error.induction}. The reason for this is purely physical. In a nutshell, the momentum equation for $w$ comes from Newton's second law, so an external body force (gravity, stirring, etc.) is a legitimate physical input. By contrast, the induction equation is not a force balance, as it is derived directly from Faraday's law combined with Ohm's law for a (perfectly or resistively) conducting fluid. Any real modification of the magnetic field $b$ must come from an electromotive force (e.g. an external current) and enters in the form of $\nabla\times g$ for some vector field $g$, obtaining the equation
\begin{equation}
    \pa_t b + u \cdot \nabla b - b\cdot \nabla u = \eta \Delta b + \nabla \times g \,, \quad \nabla\cdot b = 0 \,.
\end{equation}
In dimension two, we then have that $\nabla \times g$ is orthogonal to $b$, as well as $\nabla\times (\nabla \times g)$ to $\jmath = \nabla\times b$, which prevent the presence of forcing terms in the second equation of \eqref{eq:MHD_vort-curr}.

Summing up, the physically relevant case is only when ${\rm b}_2 = 0$. Nevertheless, we decided to include also the case ${\rm b}_2 \neq 0$ in our result for a pure mathematical interest, since, as will become clear later, its presence captures nontrivial contributions to the leading term, interacting with the viscosity $\nu>0$ and the resistivity $\eta>0$, see for instance \eqref{Dperp.intro}.
\end{rmk}

\subsection{Strategy of the proof}

We now describe the main ideas that lead to the proof of Theorem \ref{theo.princ}. Such proof is mostly based on the normal form approach developed in \cite{CDMV25}: we aim to construct a bounded and invertible transformation (on
Sobolev spaces) that molds the linearized operator into a much simpler one with a prescribed structure. Specifically for our case, the goal is to decouple the action of $\cL_{\nu,\eta,\varepsilon}$ in \eqref{cL.vare} on the mode $j=0$ from the non-zero modes in order to:
\begin{itemize}
    \item [(i)] Compute sharp asymptotics of the two eigenvalues relative to the mode $j=0$ and deduce conditions to have their real parts positive;
    \item [(ii)] Show that the remaining part of the spectrum of $\cL_{\nu,\eta,\varepsilon}$ is pure point and that such eigenvalues have negative real part.
\end{itemize}
Before giving more details, we first rewrite the operator $\cL_{\nu,\eta,\varepsilon}$ in \eqref{cL.vare}.
Using the notation (see also \eqref{eq:def_pi0})
\begin{align}
    \pi_{0}u := \frac{1}{2\pi} \int_{\T} u(y) \wrt y \,, \quad \pi_{0}^\perp := {\rm Id} - \pi_{0} \,,
 \end{align}
and the fact that $(\epsilon^2-\pa_y^2)^{-1}\pi_0=\epsilon^{-2}\pi_0$, we split $\cL_{\nu,\eta,\varepsilon}$ as
\begin{equation}\label{splitting.notpert.intro}
	\cL_{\nu,\eta,\epsilon} = \cD_{\nu,\eta} + \frac{\im}{\epsilon}\cS_0 - \im\epsilon\cR_{\nu,\eta,\epsilon},
\end{equation}
where
\begin{equation}\label{diag.intro}
	\cD_{\nu,\eta,\varepsilon}:= \begin{pmatrix}
		\nu\pa_{y}^2 & {\rm b}_2\pa_{y} \\
		{\rm b}_2\pa_{y} &\eta\pa_y^2
	\end{pmatrix},\, \quad 
	\cS_0:= \begin{pmatrix}
		-U''(y)\pi_0 & 0 \\
		0 & U''(y)\pi_0
	\end{pmatrix}\,
\end{equation}
and
\begin{equation} \label{R.intro}
	\cR_{\nu,\eta,\varepsilon} := \begin{pmatrix}
		-\im\varepsilon\nu + U(y) + U''(y) (\varepsilon^2-\pa_{y}^2)^{-1}\pi_0^\perp & -{\rm b}_1 \\
		-{\rm b}_1 & -\im\varepsilon\eta + U(y) - \big(U''(y) + 2U'(y)\pa_{y}\big) (\varepsilon^2-\pa_{y}^2)^{-1}\pi_0^\perp
	\end{pmatrix} \,.
\end{equation}
As in \cite{CDMV25}, rewriting $\cL_{\nu,\eta,\varepsilon}$ as in \eqref{splitting.notpert.intro} brings to light the non perturbative nature of the problem as $\varepsilon\rightarrow0^+$. On the other hand, the presence of several parameters $\nu,\eta, {\rm b}_1, {\rm b}_2$ and their interplay makes the analysis of the linear operator $\cL_{\nu,\eta,\varepsilon}$ different and more delicate compared to \cite{CDMV25}.

The proof of Theorem \ref{theo.princ}  consists of two main steps:
\\[1mm]
\noindent {\bf 1)} {\it Decoupling of the mode 0.} By the natural decomposition of the Sobolev space
\begin{align}
    H^s(\T) \times H^s(\T) & = \big( \C \oplus H_0^s(\T) \big) \times \big( \C \oplus H_0^s(\T) \big) = \C^2 \oplus \big( H_0^s (\T) \times H_0^s(\T) \big) \,,
\end{align}
we represent the action of the operator $\cL_{\nu,\eta,\varepsilon}$ by the matrix
\begin{equation}\label{matrix.tL.intro}
    \tL_{\nu,\eta,\varepsilon} = \begin{bmatrix}
        \Pi_{0} \cL_{\nu,\eta,\varepsilon} \Pi_{0} &  \Pi_{0} \cL_{\nu,\eta,\varepsilon} \Pi_{0}^\perp  \\
        \Pi_{0}^\perp \cL_{\nu,\eta,\varepsilon} \Pi_{0}  & \Pi_{0}^\perp \cL_{\nu,\eta,\varepsilon} \Pi_{0}^\perp 
    \end{bmatrix} \,,
\end{equation}
where we denoted
\begin{equation}
    \Pi_{0} \begin{pmatrix}
        h_1 \\ h_2
    \end{pmatrix} := \begin{pmatrix}
       \pi_0 h_1 \\ \pi_0 h_2
    \end{pmatrix} \in \C^2 \,, \quad  \Pi_{0}^\perp \begin{pmatrix}
        h_1 \\ h_2
    \end{pmatrix} := \begin{pmatrix}
       \pi^\perp_0 h_1 \\ \pi_0^\perp h_2
    \end{pmatrix}  \in H_0^s(\T) \times H_0^s(\T) \,.
\end{equation}
We aim to simplify the structure of $\tL_{\nu,\eta,\varepsilon}$ by conjugating it with an invertible transformation $\tT$ in such a way that
\begin{equation}\label{decompo.intro}
    \tT \tL_{\nu,\eta,\varepsilon} \tT^{-1} = \begin{bmatrix}
        M_{\nu,\eta}^{(0)} & 0 \\ 0 & \cL_{\nu,\eta,\varepsilon}^{(0)}
    \end{bmatrix},
\end{equation}
where the $2\times 2$ matrix $ M_{\nu,\eta}^{(0)}$ encodes the two eigenvalues relative to the mode $j=0$, and the operator $\cL_{\nu,\eta,\varepsilon}^{(0)}$ is a perturbed modification of the operator $\cL_{\nu,\eta,\varepsilon}$ with the action restricted only on $H_0^s(\T) \times H_0^s(\T)$, which will be analyzed later. To achieve such normal form, we search for the map $\tT$ of the form
\begin{equation}
    \tT = \begin{bmatrix}
        {\rm Id}_0 & \tT_{2} \\ \tT_{3} & {\rm Id}_\perp
    \end{bmatrix} \,,
\end{equation}
for some operators $\tT_2 \in \cB(H_0^s(\T) \times H_0^s(\T) , \C^2)$ and $\tT_3 \in \cB(\C^2, H_0^s(\T) \times H_0^s(\T) )$, with ${\rm Id}_0$ and ${\rm Id}_\perp$ being the identity maps on $\C^2$ and $H_0^s(\T) \times H_0^s(\T)$, respectively. A map $\tT$ of this form is invertible if we ensure that
\begin{equation}
    \| \tT_2 \tT_3 \|_{\cB(\C^2)}, \ \| \tT_3 \tT_2 \|_{\cB(H_0^s \times H_0^s )} \ll 1 \,.
\end{equation}
To obtain the desired structure after the conjugation, we determine the two operators $\tT_2$ and $\tT_3$ by solving two separate quadratic equations involving the entries of $\tL_{\nu,\eta,\varepsilon}$ in \eqref{matrix.tL.intro}. We achieve all these properties in Proposition \ref{prop:fixed_point}, obtaining that
\begin{equation}
    \tT_2 = O(\varepsilon^3) \quad \textnormal{and} \quad \tT_3 = O(\varepsilon^{-1}) \,,
\end{equation}
which implies that the transformation $\tT = {\rm Id} + O(\varepsilon^{-1})$ is ``far from the identity'', reflecting once more the non-perturbative nature of the problem. Finally, in Proposition \ref{prop:positive.zeromode} we explicitly compute the eigenvalues of the matrix $M_{\nu,\eta}^{(0)}$ with asymptotics expansions up to order $\varepsilon^2$, having previously computed expansions of the operator $\tT_3$ in powers of $\varepsilon$. The asymptotic expansions of these eigenvalues are used to deduce conditions for the instability of the mode $j=0$.
\\[1mm]
\noindent {\bf 2)} {\it Stability analysis of the non-zero modes.} Having the decomposition in \eqref{decompo.intro}, we need to prove now that $\cL_{\nu,\eta,\varepsilon}^{(0)}$ acting on $H_0^s(\T)\times H_0^s(\T)$ is spectrally stable, meaning that all its eigenvalues have negative real part. This is shown by a perturbative argument. We rewrite $\cL_{\nu,\eta,\varepsilon}^{(0)}$ as
\begin{equation}
    \cL_{\nu,\eta,\varepsilon}^{(0)} = \tD_{\perp\perp} + \cQ_{\varepsilon}^{(0)}
\end{equation}
where $\tD_{\perp\perp}$ is the unperturbed Fourier multiplier
\begin{equation}\label{Dperp.intro}
    \tD_{\perp\perp} := \begin{bmatrix}
        \nu\pa_{y}^2 & {\rm b}_2 \pa_y \\ 
        {\rm b}_{2} \pa_y & \eta \pa_y^2
    \end{bmatrix},
\end{equation}
and $\cQ_{\varepsilon}^{(0)}$ is a bounded operator on $H_0^s(\T)\times H_0^s(\T)$ of size $O(\varepsilon)$ in operator norm. In Lemma \ref{lem.eigen.D.perp.resist-iii}, we compute the unperturbed eigenvalues of $\tD_{\perp\perp}$, which, for any $j\in \Z\setminus\{0\}$, are given by
\begin{equation}
    \lambda_{\pm}(j) := -\frac{(\nu+\eta)j^2}{2} \pm \frac12 \sqrt{(\nu-\eta)^2 j^4 - 4 {\rm b}_2^2 j^2} \,, \quad \forall \, j \in \Z\setminus\{0\} \,,
\end{equation}
and satisfy ${\rm Re}(\lambda_{\pm}(j))\leq -\sigma |j|^2$ for some constant $\sigma\equiv \sigma(\nu,\eta,{\rm b}_2)>0$. After the diagonalization of $\tD_{\perp\perp}$, we will show in Lemma \ref{spettro mathcal L epsilon (1)} that the bounded perturbation $\cQ_{\varepsilon}^{(0)}$ does not destroy the spectral properties of $\tD_{\perp\perp}$. 

The combination of Proposition \ref{prop:positive.zeromode} and Lemma \ref{spettro mathcal L epsilon (1)}, together with the invertibility of the involved maps, will lead to the proof of Theorem \ref{theo.princ}, as presented in Section \ref{sect.conclusion}. Furthermore, in this conclusive section we will prove also Theorem \ref{thm:dinamico}, which will be deduced by the dynamics generated  by the $2 \times 2$ matrix $M_{\nu,\eta}^{(0)}(\varepsilon)$ in Lemma \ref{dinamica sistemino due per due} and by the dynamical stability of the operator $\cL_{\nu,\eta,\varepsilon}^{(0)}$ in Lemma \ref{stability dynamics}.

\section{Functional setting}
In this section, we collect some general definitions and properties concerning norms and matrix representation of operators which are used in the whole paper.

\smallskip

\noindent \textbf{Notations}. 
We set $\N:= \{1,2,3,...\}$ and $\N_{0}:= \N \cup \{0\}$. \\
In the whole paper, the notation $A \lesssim B$ means that there exists a constant $C=C(\nu,\eta,{\rm b}_2,S)>0$ depending on all the fixed parameters of the problem, namely $\nu,\eta>0$ and the Sobolev regularity $S>0$ such that $A\leq C(\nu,\eta,{\rm b}_2,S)B$.

\subsection{Function spaces}
Let $\K=\R,\C$ and $m=1,2$. 
 Given a function $u\in L^2(\T,\K^m)$, its Fourier expansion is given by
\begin{equation}
	u(y)=\sum_{j\in\Z}\whu (j)e^{\im jy}\,, \quad \whu(j):=\frac{1}{2\pi}\int_\T e^{-\im jy}u(y) \wrt y \in \K^m \,. 
\end{equation}
We shall work in the context of Sobolev spaces at scale $s\geq0$
\begin{equation}\label{spaziFunz}
\begin{aligned}
    H^s(\T,\K^m) &:= \Big\{ u(y) \in L^2(\T,\K^m) \, : \, \| u \|_{s}^2 := \sum_{j\in\Z}\la j\ra^{2s} |\whu (j)|^2 < \infty \Big\} \,,  \\
    H_0^s(\T,\K^m) &:= \Big\{ u\in H^s(\T,\K^m)  \, : \,\int_{\T} u(y)\wrt y = 0 \Big\}\,,
\end{aligned}
\end{equation}
with $ \braket{j}:= {\rm max}\{1, |j|\}$.
As a notation, we may also write $H^s\equiv H^s(\T)\equiv H^s(\T,\K^m)$ and the same for $H_0^s$. For $s > \frac12$ one has the inclusion $H^{s}(\T) \subset \cC^0(\T)$ on the set of continuous functions on $\T$. In addition, $H^s$ is an algebra. For any integer $s \geq 0$, we define $C^s(\T)$ as the space of $C^s$ functions from $\T \to \K^m$ with the standard norm
$$
\| u \|_{C^s} := \max_{0 \leq k \leq s} \sup_{y \in \T} |\partial_y^k u(y)|\,.
$$
If $s \geq 0$ is not an integer, we denote by $C^s(\T)$ the H\"older-Zygmund space $C^{k, \alpha}(\T)$, where $k$ is the integer part of $s$ and $\alpha := s - k$, which is the space of $C^k$ functions such that $\partial_y^k u$ is $\alpha$-H\"older. This space is equipped by the norm 
$$
\| u \|_{C^s} \equiv \| u \|_{C^{k, \alpha}} := \| u \|_{C^k} + \sup_{y_1 \neq y_2} \dfrac{|\partial_y^k u(y_1) - \partial_y^k u(y_2)|}{|y_1 - y_2|^\alpha}\,.
$$
We recall for any $s \geq 0$, the standard estimate (see \cite[Chapters 2,3]{BCD})
$$
\| u v \|_s \leq C(s) \| u \|_{C^s} \| v \|_s, \quad \forall \, u \in C^s(\T), \quad \forall \, v \in H^s(\T) \,,
$$
for some constant $C(s) > 0$. 

\noindent
Along the paper we shall also work on product spaces.
To this aim, recalling \eqref{spaziFunz} (with $m=1$), 
we define 
\begin{equation}
    \bH^s := H^s \times H^s \quad \textnormal{and} \quad \bH_0^s := H_0^s \times H_0^s \,.
\end{equation}
With a slight abuse of notation, we always 
denote by $\|\cdot\|_{s}$ the standard product norm 
on the spaces ${\bf H}^{s}$, ${\bf H}^{s}_0$.
We also denote $\bL^2 := L^2(\T) \times L^2(\T)$, with scalar product
\begin{equation}
    \braket{\big( \begin{smallmatrix}
        f_1 \\ f_2
    \end{smallmatrix} \big),\big( \begin{smallmatrix}
        g_1 \\ g_2
    \end{smallmatrix} \big)}_{\bL^2} := \braket{f_1,g_1}_{L^2}+\braket{f_2,g_2}_{L^2} \,, \quad \braket{f,g}_{L^2} := \int_{\T} f(x) \overline{g(x)} \wrt x \,.
\end{equation}
Given two Banach spaces $X,Y$, we denote by $\cB(X,Y)$ the space of bounded, linear operators $\cL : X \rightarrow Y$ equipped with the standard operator norm:
\begin{equation}
	\norm{\cL}_{\cB(X,Y)} := 
    \sup_{\norm{x}_X=1}\norm{\cL x}_Y.
\end{equation}
We also use sometimes the notation $\| \cL \|_{X \to Y}$ for the operator norm. For bounded linear operators  
$\cL : X \rightarrow X$ we shall simply denote by
$\cB(X)$ the set $\cB(X,X)$. As usual we identify the space of $2 \times 2$ matrices with ${\mathcal B}(\C^2)$. 

Finally, we also introduce the projections
\begin{equation}\label{eq:def_pi0}
	\pi_0 u:=\frac{1}{2\pi}\int_\T u(y) \wrt y\,,  \quad \pi_0^\perp:={\rm Id}-\pi_0 \,
\end{equation}
in such a way that any $u\in L^2(\T,\K^m)$ decomposes as $u =\pi_0 u+ \pi_0^\bot u$.

\subsection{Matrix block-decomposition of linear operators}\label{sect.matrix.repr}

The action of a linear operator $\cR$ over $L^2(\T,\C^2)$ is represented by
\begin{equation}\label{matrix.op.1234}
    \cR=\begin{pmatrix}
			\cR_1 & \cR_2 \\
			\cR_3 & \cR_4
		\end{pmatrix} \,,
\end{equation}
where each $\cR_i$ is a linear operator acting on scalar components in $L^2(\T,\R)$. 
In the following sections, we need to expand each matrix linear operator as in \eqref{matrix.op.1234} with matrix representations along the zero and non-zero modes of the Fourier basis. For every $s_1,s_2 \geq 0$, the following spaces are isomorphic:
\begin{equation}\label{splitting.iso}
	H^{s_1}\times H^{s_2} \simeq  \C^2 \oplus ( H_0^{s_1} \times H_0^{s_2}) \,, \quad 
    \text{with} \quad \vec{h} \mapsto \Pi_0 \vec{h} +  \Pi_{0}^\perp \vec{h} \,,
\end{equation}
where we introduced the following notation for the projections over vectors
\begin{align}\label{eq:def_Pi0}
	\Pi_0 & : H^{s_1}\times H^{s_2} \rightarrow \C^2, \quad \vec{h}=\begin{pmatrix}
	    h_1\\h_2
	\end{pmatrix} \mapsto \Pi_0[\vec{h}]:= \begin{pmatrix}
		\pi_0h_1 \\
		\pi_0h_2
	\end{pmatrix}, \\
	\Pi_0^\perp & : H^{s_1}\times H^{s_2} \rightarrow H_0^{s_1}\times H_0^{s_2}, \quad \vec{h}=\begin{pmatrix}
	    h_1\\h_2
	\end{pmatrix} \mapsto \Pi_0^\perp[\vec{h}]:=\begin{pmatrix}
		\pi_0^\perp h_1\\
		\pi_0^\perp h_2
	\end{pmatrix} \,,
\end{align}
with $\pi_0$ and $\pi_0^\perp$ as in \eqref{eq:def_pi0}. For $\cR$ as in \eqref{matrix.op.1234}, we also define
\begin{equation}
    \Pi_0\cR := \begin{pmatrix}
        \pi_0\cR_1 & \pi_0\cR_2 \\
        \pi_0\cR_3 & \pi_0\cR_4
    \end{pmatrix}\,, \quad \Pi_0^\perp \cR := \begin{pmatrix}
        \pi_0^\perp \cR_1 & \pi_0^\perp \cR_2 \\
        \pi_0^\perp \cR_3 & \pi_0^\perp \cR_4
    \end{pmatrix}\,, 
\end{equation}
so that
\begin{equation}
	\forall \, \vec{h}=\begin{pmatrix}
	    h_1\\h_2
	\end{pmatrix}\in H^{s_1}\times H^{s_2}\,, \quad \Pi_0\cR[\vec{h}] := \begin{pmatrix}
		\pi_0\cR_1[h_1] + \pi_0\cR_2[h_2] \\
		\pi_0\cR_3[h_1] + \pi_0\cR_4[h_2]
	\end{pmatrix},
\end{equation}
with analogous action for  $\Pi_0^\perp \cR$.
In addition, given any functions $f\in H_0^s(\T)$ and $g\in L_0^2(\T)$, we introduce the rank-1 operator of order 0 acting as
\begin{equation}\label{eq:traspose_operator_L2}
	L^2(\T)\ni h \mapsto f(y)g(y)^T[h] := \braket{h,g}_{L^2}f(y) \in H_0^s(\T).
\end{equation}
We may regard each bounded operator $\cL : H^{s_1}\times H^{s_2} \rightarrow H^{s_1'}\times H^{s_2'}$  as a matrix $\tL := \C^2 \oplus  (H_0^{s_1}\times H_0 ^{s_2}) \to \C^2 \oplus  (H_0^{s_1'}\times H_0 ^{s_2'})$ of the form
\begin{equation}\label{eq:matrix_repr}
	\tL := \begin{bmatrix}
		\Pi_0\cL[\vec{1}] & (\Pi_0^\perp\cL^*[\vec{1}])^T \\
		\Pi_0^\perp\cL[\vec{1}] & \cL^\perp
	\end{bmatrix},
\end{equation}
where $\vec{1}= \big(\begin{smallmatrix}
    1 \\ 1
\end{smallmatrix} \big)$,   $\cL^\perp:=\Pi_0^\perp\cL\Pi_0^\perp : H_0^{s_1}\times H_0^{s_2} \rightarrow H_0^{s_1'}\times H_0^{s_2'}$ 
and,
	making an abuse of notation (for notational convenience),  we have written, for any $\cL=\Big(\begin{smallmatrix}
			\cL_1 & \cL_2 \\
			\cL_3 & \cL_4
		\end{smallmatrix}\Big)$,
	\begin{equation}\label{pi0.matrix}
		\Pi_0\cL[\vec{1}] := \begin{pmatrix}
			\pi_0\cL_1[1] & \pi_0\cL_2[1] \\
			\pi_0\cL_3[1] & \pi_0\cL_4[1]
		\end{pmatrix}
	\end{equation}
	(analogous for $\Pi_0^\perp$) and
	\begin{equation}\label{rank.matrix}
		(\Pi_0^\perp\cL^*[\vec{1}])^T := \begin{pmatrix}
			(\pi_0^\perp\cL_1^*[1])^T & (\pi_0^\perp\cL_2^*[1])^T \\
			(\pi_0^\perp\cL_3^*[1])^T & (\pi_0^\perp\cL_4^*[1])^T
		\end{pmatrix} .
	\end{equation}
	Observe that the projections inside these matrices act over the space $H^s \simeq \C\oplus H_0^s$ and not on its product.
The representation \eqref{eq:matrix_repr} and the action of the operator $\cL$ determine each other.
\begin{lem}
	The following relation holds:
	\begin{equation}
		\tL \begin{bmatrix}
			\Pi_0[\vec{h}] \\
			\Pi_0^\perp[\vec{h}]
		\end{bmatrix} = \begin{bmatrix}
		\Pi_0\cL[\vec{h}] \\
		\Pi_0^\perp\cL[\vec{h}]
		\end{bmatrix} = \begin{bmatrix}
		\Pi_0\cL\Pi_0 & \Pi_0\cL\Pi_0^\perp\\
		\Pi_0^\perp\cL\Pi_0 & \Pi_0^\perp\cL\Pi_0^\perp
		\end{bmatrix}[\vec{h}] \,.
	\end{equation}
	In particular, the action of the matrix $\tL$ in \eqref{eq:matrix_repr} is linked to  the associated operator $\cL$.
\end{lem}
\begin{proof} We compute the action of $\Pi_0 \cL$.
	According to the splitting of the space in \eqref{splitting.iso}, we have
	\begin{equation}
		\Pi_0\cL[\vec{h}] = \Pi_0\cL\Pi_0[\vec{h}] + \Pi_0\cL\Pi_0[\vec{h}].
	\end{equation}
	Using the linearity of $\cL$ and the fact that $\Pi_0[\vec{h}]=\Big( \begin{smallmatrix}
	  \pi_{0}[h_1] \\ \pi_{0}[h_2]  
	\end{smallmatrix} \Big) \in\C^2$, one has that $\cL_i\pi_0[h_j] = \cL_i[1]\pi_0[h_j]$, for each $i=1,\ldots,4$ and $j=1,2$. Then,
	\begin{equation}
		\Pi_0\cL\Pi_0[\vec{h}] = \Pi_0 \begin{pmatrix}
			\cL_1 & \cL_2 \\
			\cL_3 & \cL_4
		\end{pmatrix} \begin{pmatrix}
		\pi_0h_1 \\
		\pi_0h_2
		\end{pmatrix} = \begin{pmatrix}
		\Pi_0\cL_1[1]\pi_0h_1 + \Pi_0\cL_2[1]\pi_0h_2 \\
		\Pi_0\cL_3[1]\pi_0h_1 + \Pi_0\cL_4[1]\pi_0h_2
		\end{pmatrix} = \Pi_0\cL[1][\Pi_0[\vec{h}]].
	\end{equation}
	On the other hand, since $\Pi_0f=\int_\T f(y)\,{\rm d}y$ by definition and both $\Pi_0$, $\Pi_0^\perp$ are self-adjoint with respect to the $L^2$-scalar product, one has that
	\begin{align}
		\Pi_0\cL\Pi_0^\perp[\vec{h}] & = \begin{pmatrix}
			\int_\T \cL_1[\pi_0^\perp[h_1]]\wrt y + \int_\T \cL_2[\pi_0^\perp[h_2]] \wrt y \\
			\int_\T \cL_3[\pi_0^\perp[h_1]] \wrt y  + \int_\T \cL_4[\pi_0^\perp[h_2]] \wrt y
		\end{pmatrix} =  \begin{pmatrix}
		\la 1,\cL_1\pi_0^\perp h_1 \ra_{L^2} + \la 1, \cL_2\pi_0^\perp h_2 \ra_{L^2} \\
		\la 1, \cL_3\pi_0^\perp h_1 \ra_{L^2} + \la 1, \cL_4\pi_0^\perp h_2 \ra_{L^2}
		\end{pmatrix} \\
		& = \begin{pmatrix}
			\la \pi_0^\perp \cL_1^*[1],h_1 \ra_{L^2} + \la \pi_0^\perp \cL_2^*[1],h_2 \ra_{L^2} \\
			\la \pi_0^\perp \cL_3^*[1],h_1 \ra_{L^2} + \la \pi_0^\perp \cL_4^*[1],h_2 \ra_{L^2}
		\end{pmatrix} = (\Pi_0^\perp\cL^*[1])^T[\Pi_0^\perp[\vec{h}]].
	\end{align}
	recalling the notation  in \eqref{rank.matrix}. The computations for $\Pi_0^\perp\cL[\vec{h}]$ are analogous, and therefore omitted.
\end{proof}


\section{Properties of the linearized operator
}\label{sect.properties.L}

In this section, we compute the matrix representation of the operator $\cL_{\nu,\eta,\varepsilon}$ in \eqref{cL.vare} according to the definitions given in Section \ref{sect.matrix.repr}. In particular, in Section \ref{sect.expand.first} we compute its entries and their expansions in powers of $\varepsilon$, which highlights the non-perturbative nature of the problem. In doing so, we identify the unperturbed leading order operator $\tD_{\perp\perp}$: we study its invertibility and its diagonalization in Sections \ref{sect.invert.D} and \ref{sect.eigen.D}, respectively.
 
\subsection{Expansion of the linearized operator}\label{sect.expand.first}
In the next lemma, we rewrite the operator $\cL_{\nu,\eta,\varepsilon}$ in \eqref{cL.vare}  
with its block representation according to the notation of \eqref{eq:matrix_repr}.

\begin{lem}\label{lemma.adjoint}
The operator $\cL_{\nu,\eta,\varepsilon}$ in \eqref{cL.vare} is represented  by a block-matrix $\tL_{\nu,\eta,\varepsilon}$ of the form
\begin{equation}\label{tL.vare}
	\tL_{\nu,\eta,\varepsilon} = \left[ \begin{matrix}
		\tL_{00} & \tL_{0\perp} \\
		\tL_{\perp0} & \tL_{\perp\perp}
	\end{matrix} \right] := \left[\begin{matrix}
	\Pi_0\cL_{\nu,\eta,\epsilon}\Pi_0 & \Pi_0\cL_{\nu,\eta,\epsilon}\Pi_0^\perp\\
	\Pi_0^\perp\cL_{\nu,\eta,\epsilon}\Pi_0 & \Pi_0^\perp\cL_{\nu,\eta,\epsilon}\Pi_0^\perp
	\end{matrix}\right] \overset{\eqref{eq:matrix_repr}}{=} \left[\begin{matrix}
	\Pi_0\cL_{\nu,\eta,\epsilon}[\vec{1}] & (\Pi_0^\perp\cL_{\nu,\eta,\epsilon}^*[\vec{1}])^T \\
	\Pi_0^\perp\cL_{\nu,\eta,\epsilon}[\vec{1}] & \cL_{\nu,\eta,\epsilon}^\perp
	\end{matrix}\right],
\end{equation}
where 
the adjoint operator of $\cL_{\nu,\eta,\epsilon}$ 
is given by
	\begin{equation}\label{aggiunto.L}
		\cL_{\nu,\eta,\epsilon}^* = \begin{bmatrix}
			\cL_{\varepsilon,1}^* & \cL_{\varepsilon,2}^*\\
			\cL_{\varepsilon,3}^* & \cL_{\varepsilon,4}^*
		\end{bmatrix},
	\end{equation}
	with
    \begin{equation}\label{funzioniabcd}
	\begin{aligned}
		\cL_{\varepsilon,1}^* &:= \nu(\pa_y^2-\varepsilon^2) + \frac{\im}{\epsilon}\pi_0\circ(U''(y){\rm Id}) + \im\epsilon\big( U(y) + \pi_0^\perp(\epsilon^2-\pa_y^2)^{-1}\circ (U''(y){\rm Id}) \big)\,,\\
		\cL_{\varepsilon,2}^*=\cL_{\varepsilon,3}^* &:= - {\rm b}_2\pa_y -\im\varepsilon {\rm b}_1 \,, \\
		\cL_{\varepsilon,4}^* &:= \eta(\pa_{y}^2-\varepsilon^2) - \frac{\im}{\epsilon}\pi_0\circ(U''(y){\rm Id}) + \im\varepsilon \big( U(y) + \pi_0^\perp(\epsilon^2-\pa_y^2)^{-1}\circ (U''(y){\rm Id}+2U'(y)\pa_y) \big).
	\end{aligned}
    \end{equation}
\end{lem}
\begin{proof}
We just have to explicitly compute the entries \eqref{funzioniabcd} of \eqref{aggiunto.L}.
    Recalling the splitting of $\cL_{\nu,\eta,\varepsilon}$ in \eqref{splitting.notpert.intro}, we denote its entries by
    \begin{align}
        \cL_{\varepsilon,1} & := \nu(\pa_y^2-\varepsilon^2) - \frac{\im}{\varepsilon}U''(y)\pi_0 -\im\varepsilon\big( U(y) + U''(y)(\epsilon^2-\pa_y^2)\pi_0^\perp \big) \,, \\
        \cL_{\varepsilon,2} = \cL_{\varepsilon,3} & := {\rm b}_2\pa_y + \im\varepsilon{\rm b}_1 \,, \\
    \cL_{\varepsilon,4} & := \eta(\pa_y^2-\varepsilon^2) + \frac{\im}{\varepsilon}U''(y)\pi_0 -\im\varepsilon\big( U(y) -(U''(y)+2U'(y)\pa_y)(\epsilon^2-\pa_y^2)\pi_0^\perp \big) \,.
    \end{align}
    First, we compute $\cL_{\varepsilon,2}^{*}$. For any $f,g\in H^s(\T)$ we have
        \begin{align}
        \langle f, \cL_{\varepsilon,2}^*[g] \rangle_{L^2} & =  \langle  \cL_{\varepsilon,2}[f],g\rangle_{L^2} = \langle ({\rm b}_2\pa_y + \im\varepsilon{\rm b}_1)f,g \rangle_{L^2} = \langle f,( - {\rm b}_2\pa_y -\im\varepsilon{\rm b}_1)g\rangle_{L^2} \,.
    \end{align}
    We now compute  $\cL_{\varepsilon,4}^*$. The computations for $\cL_{\varepsilon,1}^*$ are analogous and therefore omitted. 
    For any $f,g\in H^s(\T)$, we have
    \begin{align}
         \langle f, \cL_{\varepsilon,2}^*[g] \rangle_{L^2} & =  \langle  \cL_{\varepsilon,2}[f],g\rangle_{L^2} \label{caldo 0}\\
         & =  \Big\langle [\eta(\pa_y^2-\varepsilon^2) + \frac{\im}{\varepsilon}U''(y)\pi_0 - \im\varepsilon U(y)]f,g \Big\rangle_{L^2}  +  \langle \im\varepsilon ( U''(y)+2 U'(y) \pa_y) (\varepsilon^2-\pa_y^2)^{-1}  \pi_0^\perp f,g \rangle_{L^2} \,.
    \end{align}
    We look at each term separately. On the one hand, we have
    \begin{equation}
        \Big\langle \Big[\eta(\pa_y^2-\varepsilon^2) + \frac{\im}{\varepsilon}U''(y)\pi_0 - \im\varepsilon U(y)\Big]f,g \Big\rangle_{L^2} = \Big\langle f,\Big[\eta(\pa_y^2-\varepsilon^2) - \frac{\im}{\varepsilon}\pi_0 \circ (U''(y){\rm Id}) +\im\epsilon U(y)\Big]g \Big\rangle_{L^2} \,, \label{caldo 1}
    \end{equation}
    where we have used the self-adjointness of $\pa_y^2-\varepsilon^2$ and that $U(y) \in \R$. We now compute
    \begin{align}
        \langle \im\varepsilon  ( U''(y)+2 U'(y) \pa_y) (\varepsilon^2-\pa_y^2)^{-1}  \pi_0^\perp f,g \rangle_{L^2} &  = \langle  f, -\im\varepsilon \pi_0^\perp \circ (\varepsilon^2-\pa_{y}^2)^{-1} ( U''(y) - 2 \pa_{y} \circ U'(y) ) g \rangle_{L^2} \label{caldo 2} \\
         &  = \langle  f, \im\varepsilon \pi_0^\perp \circ (\varepsilon^2-\pa_{y}^2)^{-1} ( U''(y)  + 2   U'(y)\pa_{y} ) g \rangle_{L^2}
    \end{align}
 where we have used the self-adjointness of $(\pa_y^2-\varepsilon^2)^{-1}$, the reality of $U(y) \in \R$ and that $\pa_y \circ U'(y) = U''(y) + U'(y) \pa_y$. Collecting all the contributions \eqref{caldo 1}, \eqref{caldo 2} back in \eqref{caldo 0}, we deduce that
 \begin{equation}
      \langle f,\cL_{\varepsilon,4}^*[g] \rangle_{L^2} = \Big\langle f, \Big[\eta(\pa_y^2-\varepsilon^2) - \frac{\im}{\epsilon}\pi_0\circ(U''(y){\rm Id}) + \im\varepsilon \big( U(y) + \pi_0^\perp \circ (\varepsilon^2-\pa_y^2)^{-1} \circ (U''(y){\rm Id}+2U'(y)\pa_y) \big)\Big]g \Big\rangle_{L^2} \,.
 \end{equation}
    This concludes the proof of the claimed expressions in \eqref{funzioniabcd}.
\end{proof}

We now provide expansions in powers of $\varepsilon$ of the operator in 
\eqref{tL.vare}. We first need a technical lemma.
\begin{lem}\label{lem:expansion_Laplacian_resist}
	In $H_0^s(\T)$, the following expansion holds
	\begin{equation}\label{lapla.expand}
		(\varepsilon^2-\pa_y^2)^{-1}\Pi_0^\perp = -\pa_y^{-2}\Pi_0^\perp + \varepsilon^{2} \cQ^{(2)}(\varepsilon)\,, \quad \cQ^{(2)}(\varepsilon) := 
        (\varepsilon^2-\pa_y^2)^{-1}\pa_y^{-2} \,,
	\end{equation}
    where $\| \cQ^{(2)}(\varepsilon) \|_{\cB(H_0^s,H_0^{s+2})}\leq 1$.
\end{lem}
\begin{proof}
	Let $h\in H_0^s(\T)$. Then,
	\begin{equation}
		(\varepsilon^2-\pa_y^2)^{-1}\Pi_0^\perp[h] = (\varepsilon^2-\pa_y^2)^{-1}(\pa_y^2 -\varepsilon^2+\varepsilon^2)\pa_y^{-2}h =- \pa_y^{-2}h +\varepsilon^2(\varepsilon^2-\pa_y^2)^{-1}\pa_y^{-2}h 
	\end{equation}
	which proves the expansion in \eqref{lapla.expand}. Moreover,
    for $h\in H_0^{s}$,
    we have
    \begin{align}
        \|\cQ^{(2)}(\varepsilon)[h]\|_{s+2}^2 
        = \sum_{j\neq0} \la j\ra^{2(s+2)}
        \Big|\tfrac{1}{j^{2}(j^2+\epsilon^2)}\Big|^2 
        |\whh(j)|^2 
        \leq \big(\tfrac{1}{1+\epsilon^2}\big)^2
        \| h \|_{s}^2 \leq 
        \| h \|_{s}^2 \,,
    \end{align}
    from which we deduce the desired estimate.
\end{proof}

We are now in position to prove the following.

\begin{lem}\label{espandi.e.non.mollo}
  Let $U \in C^{S + 2}(\T)$, $S \geq 0$.  The components of the operator $\tL_{\nu,\eta,\varepsilon}$ in \eqref{tL.vare} read as
    \begin{align}
    \tL_{00}:\C^2\to \C^2 \,, \quad \tL_{00} & = \begin{bmatrix}
		-\nu\varepsilon^2 & \im\varepsilon{\rm b}_1 \\
		\im\varepsilon{\rm b}_1 & -\eta\varepsilon^2
	\end{bmatrix}\,,  \label{L00}  \\
     \tL_{0\perp}:\bH_0^s \to \C^2 \,, \quad \tL_{0\perp} & = \begin{bmatrix}
		[\im\varepsilon^3(\varepsilon^2-\pa_y^2)^{-1}U(y)]^T & 0^T \\
		0^T & [\im\varepsilon^3(\varepsilon^2-\pa_y^2)^{-1}U(y)]^T
	\end{bmatrix} \,, \label{L0p} \\
     \tL_{\perp0}:\C^2\to \bH_{0}^s \,, \quad \tL_{\perp0} & = \begin{bmatrix}
		-\im\varepsilon U(y)-\frac{\im}{\varepsilon}U''(y) & 0 \\
		0 & -\im\varepsilon U(y) + \frac{\im}{\varepsilon}U''(y)
	\end{bmatrix}\,, \label{Lp0} \\
     \tL_{\perp\perp}:\bH_0^s \to \bH_{0}^s \,, \quad  \tL_{\perp\perp} &= 
        \tD_{\perp\perp} + \tR_{\perp\perp}\,, \label{Lpp}
\end{align}
where
\begin{align}
    \tD_{\perp\perp} & := \begin{bmatrix}
			\nu\pa_y^2 & {\rm b}_2\pa_y \\
			{\rm b}_2\pa_y  & \eta\pa_y^2
		\end{bmatrix} \,,  \label{D.perperp} \\
    \tR_{\perp\perp} & := \begin{bmatrix}
		-\nu\varepsilon^2 - \im\varepsilon\pi_0^\perp\big( U(y) + U''(y) (\varepsilon^2-\pa_{y}^2)^{-1}  \big) & \im\varepsilon{\rm b}_1 \\
		\im\varepsilon{\rm b}_1 & -\eta\varepsilon^2-\im\varepsilon\pi_0^\perp\big(U(y) - (U''(y) + 2U'(y)\pa_{y}) (\varepsilon^2-\pa_{y}^2)^{-1} \big) 
		\end{bmatrix} \,.
\end{align}
In particular we have the expansions
 \begin{align}
		\tL_{00}&= \varepsilon\begin{bmatrix}
			0 & \im{\rm b}_1 \\
			\im{\rm b}_1 & 0
		\end{bmatrix} + \varepsilon^2\begin{bmatrix}
			-\nu & 0 \\
			0 & -\eta
		\end{bmatrix} =: \varepsilon\tL_{00}^{(1)} + \varepsilon^2\tL_{00}^{(2)} \,, \label{L00.esp} \\
		\tL_{0\perp}&= \varepsilon^3\begin{bmatrix}
			-(\im\pa_y^{-2}U(y))^T & 0^T \\
			0^T & -(\im\pa_y^{-2}U(y))^T
		\end{bmatrix} +  \varepsilon^5 \tT_{0\perp}^{(5)}(\varepsilon) =: \varepsilon^3 \tL_{0\perp}^{(3)} + \varepsilon^5 \tT_{0\perp}^{(5)}(\varepsilon) \,, \label{Lop.eps}  \\
		\tL_{\perp 0} &= \varepsilon^{-1}\begin{bmatrix}
			-\im U''(y) & 0 \\
			0 & \im U''(y)
		\end{bmatrix}+ \varepsilon\begin{bmatrix}
			-\im U(y) & 0 \\
			0 & -\im U(y)
		\end{bmatrix} =: \varepsilon^{-1}\tL_{\perp 0}^{(-1)} + \varepsilon\tL_{\perp 0}^{(1)} \,. \label{Lp0.eps}
	\end{align}
    with the following estimates for any $0 \leq s \leq S$: 
    \begin{align}
          \| \tL_{00} \|_{\cB(\C^2)} & \lesssim \varepsilon \,, \quad   \| \tL_{0\perp}\|_{\cB(\bH_0^s,\C^2)}\lesssim \varepsilon^3 \,, \quad  \| \tL_{\perp 0}\|_{\cB(\C^2,\bH_0^s)}\lesssim \varepsilon^{-1}\,,\label{esti.easy} \\
           \| \tT_{0\perp}^{(5)}(\varepsilon) \|_{\cB(\bH_0^s,\C^2)}& \lesssim 1\,,\label{stimaT5nonbanale} \\
             \|  \tR_{\perp\perp}\|_{\cB(\bH_0^s)} &\lesssim \varepsilon \,.
             \label{Lpp.eps}
    \end{align}
\end{lem}

\begin{proof}
Recalling 
Lemma \ref{lemma.adjoint} and the notation in \eqref{pi0.matrix},
one gets formul\ae\, \eqref{L00}, \eqref{Lp0}, \eqref{Lpp} and \eqref{D.perperp}.
The expansions \eqref{L00.esp}, \eqref{Lp0.eps} follow trivially.
Therefore, we focus on the term 
$\tL_{0\perp}=
(\Pi_{0}^{\perp}\mathcal{L}^{*}_{\nu,\eta,\varepsilon
}[\vec{1}])^{T}$ (see \eqref{tL.vare}) in order to prove \eqref{Lp0}, with expansion \eqref{Lop.eps}. 
we have to compute, recalling the notation in \eqref{rank.matrix},
\begin{equation}
    \tL_{0\perp} =
(\Pi_{0}^{\perp}\mathcal{L}^{*}_{\nu,\eta,\varepsilon
}[\vec{1}])^{T} = \begin{bmatrix}
    (\pi_0^\perp \cL_{\varepsilon,1}^*[1])^T &    (\pi_0^\perp \cL_{\varepsilon,2}^*[1])^T \\  (\pi_0^\perp \cL_{\varepsilon,3}^*[1])^T &  (\pi_0^\perp \cL_{\varepsilon,4}^*[1])^T
\end{bmatrix} \,.
\end{equation}
where the operators $\cL_{\varepsilon,m}^*$, $m=1,2,3,4$, are as in \eqref{funzioniabcd}. First, we clearly have
\begin{equation}
     \pi_0^\perp \cL_{\varepsilon,2}^*[1]  =  \pi_0^\perp \cL_{\varepsilon,3}^*[1] = 
     \pi_0^\perp (-\im \varepsilon {\rm b}_1) = 0 \,.
\end{equation}
Now, by Lemma \ref{lem:expansion_Laplacian_resist}, recalling that $\int_{\T} U(y)\wrt y = 0$,
we get
\[
\begin{aligned}
\pi_0^{\perp}\cL_{\varepsilon,1}^*[1] &= \pi_0^\perp \Big( - \nu  \varepsilon^2 - \im \varepsilon^{-1} \pi_0 (U''(y)) + \im \varepsilon \big( U(y) + \pi_0^\perp(\epsilon^2-\pa_y^2)^{-1}U''(y)  \big) \Big)  \\
&  = \pi_0^\perp \Big( \im \varepsilon \big( U(y) + \pi_0^\perp(\epsilon^2-\pa_y^2)^{-1}U''(y)  \big) \Big) \\
& \stackrel{\eqref{lapla.expand}}{=} \pi_0^\perp \big( \im \varepsilon^3 (\varepsilon^2-\pa_{y}^2)^{-1}U(y) \big) \,.
\end{aligned}
\]
A similar computation on the function $\cL_{\varepsilon,4}^*[1]$ implies \eqref{L0p}.
Applying again Lemma \ref{lem:expansion_Laplacian_resist} to
\eqref{L0p}, we deduce the expansion in \eqref{Lop.eps} with
\[
\tT_{0\perp}^{(5)}(\e):=
\begin{bmatrix}
			({\rm i}(\varepsilon^2-\pa_y^2)^{-1}\pa_y^{-2}U(y))^T & 0^T \\
			0^T & ({\rm i}(\varepsilon^2-\pa_y^2)^{-1}\pa_y^{-2}U(y))^T
		\end{bmatrix}\,.
\]
Finally, we prove the estimates \eqref{esti.easy}, \eqref{stimaT5nonbanale}, \eqref{Lpp.eps}. First, the estimate $\| \tL_{00}\|_{\cB(\C^2,\C^2)}\lesssim \varepsilon$ in \eqref{esti.easy} follows directly from \eqref{L00}. 
Now consider the operators $\tL_{0\perp}^{(3)}, \tT_{0\perp}^{(5)}$
in \eqref{Lop.eps}.
Observing that, for any $h\in H_0^s(\T)$,
\begin{align}
    \big| \big( (\pa_{y}^2)^{-1} U(y) \big)^T h\big| & = \big| \braket{h,(\pa_{y}^2)^{-1} U}_{L^2} \big|  = \big| \braket{(\pa_{y}^2)^{-1} h, U}_{L^2} \big| \\
   & \leq \|  (\pa_{y}^2)^{-1} h \|_{L^2} \| U \|_{L^2}  \lesssim \| h \|_{L^2} \| U \|_{C^0} \,,
\end{align}
which implies, for any $s\geq 0$,
$\| \tL_{0\perp}^{(3)}  \|_{\cB(\bH_0^s,\C^2)}\lesssim 1$.
Similarly one has
\begin{align}
    \big| \big( (\e^{2}-\pa_{y}^2)^{-1}\pa_{y}^{-2} U(y) \big)^T h\big| & = \big| \braket{h,(\e^2-\pa_{y}^2)^{-1}\pa_{y}^{-2} U}_{L^2} \big|  = 
    \big| \braket{(\e^2-\pa_{y}^2)^{-1}\pa_{y}^{2} h, U}_{L^2} \big| 
    \\
   & \leq \|  (\e^{2}-\pa_{y}^2)^{-1} \pa_{y}^{2}h \|_{L^2} \| U \|_{L^2}  \lesssim \| h \|_{L^2} \| U \|_{C^0} \,,
\end{align}
we conclude that, for any $s\geq 0$ one has
$\| \tT_{0\perp}^{(5)}  \|_{\cB(\bH_0^s,\C^2)}\lesssim 1$.
The estimates above imply
 \eqref{esti.easy} and \eqref{stimaT5nonbanale}.
Let us now consider the operator $L_{0\perp}$. The prove the estimate 
\eqref{esti.easy}, we just consider the term $L_{\perp0}^{(-1)}$.
One has
\[
\begin{aligned}
\|L_{\perp0}^{(-1)}\|_{\mathcal{B}(\C^2,{\bf H}^{s}_0)}
\lesssim \|U''\|_{s}\lesssim \|U\|_{s+2}\,.
\end{aligned}
\]
The estimates for $L_{\perp0}^{(1)}$ are similar, and hence one gets the \eqref{esti.easy}. Let us now consider the operator $\tR_{\perp\perp}$.
Let us first define
\[
\begin{aligned}
\tF_1&:=	-\nu\varepsilon^2 - \im\varepsilon\pi_0^\perp\big( U(y) + U''(y) (\varepsilon^2-\pa_{y}^2)^{-1}  \big)\,,
\\
\tF_2&:=	 -\eta\varepsilon^2-\im\varepsilon\pi_0^\perp\big(U(y) - (U''(y) + 2U'(y)\pa_{y}) (\varepsilon^2-\pa_{y}^2)^{-1} \big) \,.
\end{aligned}
\]
Given $h\in H_0^{s}$, one has
\[
\begin{aligned}
\|\tF_1[h]\|_{s}& \lesssim \e^2\|h\|_s+\e \|U\|_{C^s}\|h\|_s+
\e \|U''\|_{C^s}\|(\e^2-\pa_{y})^{-1}h\|_s \\
& \lesssim \e^2 \|h\|_s + \e \| U \|_{C^{s + 2}} \| h \|_s \lesssim \e \| h \|_s\,.
\end{aligned}
\]
Similarly, one deduces 
\[
\begin{aligned}
\|\tF_2[h]\|_{s}&\lesssim
\e^2\|h\|_s+\e \|U\|_{C^s}\|h\|_s
+
\e \|U''\|_{C^s}\|(\e^2-\pa_{y})^{-1}h\|_s 
+
\e \|U'\|_{C^s}\|(\e^2-\pa_{y})^{-1}\pa_{y}h\|_s 
\\
& \lesssim \e^2 \|h\|_s + \e \| U \|_{C^{s + 2}} \| h \|_{s} 
\lesssim \e \| h \|_s\,.
\end{aligned}
\]
where we used that the operator $(\e^2-\pa_{y})^{-1}\pa_{y}$
is bounded on $H^{s}_0$ with estimates uniform with respect to $\e$.
From the latter estimates we deduce \eqref{Lpp.eps}.
\end{proof}

\subsection{Invertibility of the unperturbed operator on the non-zero modes}\label{sect.invert.D}


In Lemma \ref{espandi.e.non.mollo}, we determined the unperturbed leading operator $\tD_{\perp\perp}$, whose properties are fundamental in the upcoming sections.
In the following lemma, we analyze its invertibility, which plays a key role in the fixed point argument for the decoupling of the zero and non-zero modes in Section \ref{sec:fixedpoint1}.
\begin{lem}\label{lem:Dperp_inv.resist}
{\bf (The diagonal operator 
$\tD_{\perp\perp}$: invertibility).} 
Let $\nu,\eta>0$.
The matrix  
$\tD_{\perp\perp}: {\bf H}^{s}_{0}\to {\bf H}^{s-2}_{0}$
defined in  \eqref{D.perperp} is invertible, with  inverse given by
\begin{align}
\tD_{\perp\perp}^{-1}& : 
{\bf H}^{s-2}_0
\to  {\bf H}_{0}^{s} \,, 
\\  
\tD_{\perp\perp}^{-1} 
&:= \big( \nu\eta \pa_{y}^4- |{\rm b}_2|^2 \pa_{y}^{2} \big)^{-1}  
\begin{bmatrix}
\eta \pa_{y}^{2} & -{\rm b}_2 \pa_{y} 
\\ 
-{\rm b}_2 \pa_{y}  & \nu \pa_{y}^{2}
\end{bmatrix}   
= \big( {\rm Id}_\perp - \tfrac{|{\rm b}_2|^2}{\nu\eta} \pa_{y}^{-2} \big)^{-1}  
\begin{bmatrix}
\frac{1}{\nu} \pa_{y}^{-2} & -\frac{{\rm b}_2}{\nu\eta} \pa_{y}^{-3} 
\\ -\frac{{\rm b}_2}{\nu\eta} \pa_{y}^{-3} 
& \frac{1}{\eta} \pa_{y}^{-2}
        \end{bmatrix}  \,,
    \end{align}
and satisfying the estimates
    \begin{equation}\label{est.D.perp.resist}
\|\tD_{\perp\perp}\|_{\cB({\bf H}_0^s,{\bf H}_0^{s-2})} \leq C_1(\nu,\eta, {\rm b}_2) \,, 
\qquad 
\|\tD_{\perp\perp}^{-1}\|_{
\cB({\bf H}_0^{s-2},{\bf H}_0^{s})} 
\leq C_2(\nu, \eta, {\rm b}_2)\,,
\end{equation}
for some positive constants 
$C_i(\nu, \eta,{\rm b}_2) > 0$, $i = 1, 2$. 
Moreover one has  
that $C_2(\eta, \nu,{\rm b}_2) \to + \infty$ 
if $\eta \to 0$ or $\nu \to 0$.   
\end{lem}

\begin{proof}
	The action of  $\tD_{\perp\perp}$ in \eqref{D.perperp} on a vector $u:\T\to\C^2$ with zero average is given by 
\[
u(y) = \sum_{j \in \Z\setminus\{0\}} \whu(j) e^{\im j y } \quad \mapsto \quad
\tD_{\perp\perp} u (y) = 
\sum_{j \in \Z \setminus \{ 0 \}} 
\tD_{\perp\perp}(j)[\widehat u(j)] e^{\im j y}
\]
where 
\begin{equation}\label{eq:D.perp.Fourier}
\tD_{\perp\perp}(j) := 
\begin{pmatrix}
-\nu j^2 & \im{\rm b}_2j 
\\
	\im{\rm b}_2j & -\eta j^2
\end{pmatrix} \,, 
\qquad j\in\Z\setminus\{0\}\,. 
	\end{equation}
Note that $\tD_{\perp\perp}(j)$ is invertible, with inverse  given by
\begin{equation}\label{eq:D.perp.-1.Fourier}
		\tD_{\perp\perp}^{-1}(j) = \Big( 1 + \frac{|{\rm b}_2|^2}{\nu\eta j^2} \Big)^{-1} \begin{pmatrix}
            -\frac{1}{\nu j^2} & \im\frac{{\rm b}_2}{\nu\eta j^3} \\
            \im\frac{{\rm b}_2}{\nu\eta j^3} & -\frac{1}{\eta j^2}\,,
        \end{pmatrix}
	\end{equation}
   and hence $\tD_{\perp\perp}^{-1}[u](y) = \sum_{j \in \Z \setminus \{ 0 \}} \tD_{\perp\perp}(j)^{- 1}[\widehat u(j) ]e^{\im j y}$. In particular, for any $j \in \Z \setminus \{ 0 \}$, one has that 
\[
\| \tD_{\perp\perp}(j) \|_{{\mathcal B}(\C^2)} \leq
C_1(\nu,\eta,{\rm b}_2)|j|^2\,, 
\qquad 
\| \tD_{\perp\perp}(j)^{-1} \|_{{\mathcal B}(\C^2)} 
\leq C_{2}(\nu,\eta,{\rm b}_2) \frac{1}{|j|^2}\,,
\]
where
\begin{equation}
    C_1(\nu,\eta,{\rm b}_2) \simeq \nu + \eta + |{\rm b}_2| \,, \quad  C_2(\nu,\eta,{\rm b}_2) \simeq  \frac{1}{\nu} + \frac{1}{\eta} + \frac{|{\rm b}_2|}{\nu \eta} \,,
\end{equation}
which implies, for $u\in {\bf H}_0^{s+2}$, that 
   $$
   \begin{aligned}
\| \tD_{\perp\perp} u \|_{s} & = 
\Big( \sum_{j \in \Z \setminus \{ 0 \}} |j|^{2 s} 
|\tD_{\perp\perp}(j)[\widehat u(j)]|^2 \Big)^{\frac12} 
\\
& \leq C_1(\nu,\eta,{\rm b}_2) 
\Big( \sum_{j \in \Z \setminus \{ 0 \}} |j|^{2 (s + 2)}
|[\widehat u(j)]|^2 \Big)^{\frac12} 
 \leq
 C_1(\nu,\eta,{\rm b}_2) \| u \|_{s+2}\,, 
\end{aligned}
   $$
   and, for $u\in {\bf H}_0^{s}$,
   $$
   \begin{aligned}
\| \tD_{\perp\perp}^{- 1} u \|_{s+2} &= 
\Big( \sum_{j \in \Z \setminus \{ 0 \}} |j|^{2 (s + 2)} |\tD_{\perp\perp}(j)^{- 1}[\widehat u(j)]|^2 \Big)^{\frac12}  
\\
& \leq C_2(\nu,\eta,{\rm b}_2 ) 
\Big( \sum_{j \in \Z \setminus \{ 0 \}} |j|^{2 s} 
|[\widehat u(j)]|^2 \Big)^{\frac12}  
  \leq C_2(\nu,\eta,{\rm b}_2) 
  \| u \|_{s}\,.
 \end{aligned}
   $$
    This concludes the proof of the estimates in \eqref{est.D.perp.resist}.
\end{proof}

\subsection{Eigenvalues of the unperturbed linearized operator on the non-zero modes}\label{sect.eigen.D}
The goal of the next lemma is to completely characterize the eigenvalues of the operator $\mathtt D_{\bot \bot}$ and provide its diagonalization, which will come in handy when studying the stability of the non-zero modes in Section \ref{sect.nonzero}.

\begin{lem}\label{lem.eigen.D.perp.resist-iii}
{\bf (The diagonal operator $\tD_{\perp\perp}$: eigenvalues).} Let $\eta, \nu > 0$. The following assertions hold:
\\[1mm]
    \noindent $(i)$
	The matrix operator $\tD_{\perp\perp}$ defined by (recall \eqref{D.perperp})
    \begin{equation}
        \tD_{\perp\perp} := \begin{bmatrix}
            \nu \partial_{y}^2 &  {\rm b}_{2} \partial_{y} \\
             {\rm b}_{2} \partial_{y} & \eta \partial_{y}^2 
        \end{bmatrix}
    \end{equation}
    reads as a block-diagonal operator
    \begin{equation}
        \tD_{\perp\perp} = {\rm diag}\big\{ \tD_{\perp\perp}(j) \,: \, j \in \Z\setminus\{0\} \big\} \,,
    \end{equation}
    where, for any $j\in\Z\setminus\{0\}$, each $2\times 2$ block $\tD_{\perp\perp}(j)$ has two eigenvalues $\lambda_{\pm}(j)\in \C$ given by
	\begin{align}\label{eigenval_Dperp.resist-iii}
		\lambda_{\pm}(j) = \lambda_{\pm} (-j)
         &  = -\frac{(\nu+\eta) j^2}{2} \pm \frac12 \sqrt{ (\nu-\eta)^2 j^4- 4 {\rm b}_{2}^2 j^2 } \,.
	\end{align}
Moreover, there exists a constant $\sigma \equiv \sigma(\nu, \eta, {\rm b}_2) > 0$ such that 
${\rm Re}(\lambda_\pm(j)) \leq - \sigma |j|^2$ for any $j \in \Z \setminus \{ 0 \}$;
\\[1mm]
\noindent
$(ii)$ For any $j \in \Z \setminus \{ 0 \}$ there exists an invertible matrix ${\mathcal C}(j) \in {\rm Mat}(2 \times 2)$ such that 
\begin{equation}\label{stimadiCJ}
\sup_{j \in \Z \setminus \{ 0 \}} \| {\mathcal C}(j)^{\pm 1} \|_{\mathcal{B}(\mathbb{C}^2)} \leq  C \equiv C(\eta, \nu, {\rm b}_2)
\end{equation}
and 
\begin{equation}\label{merocontoAlg}
{\mathcal C}(j)^{- 1} {\mathtt D}_{\perp \perp}(j) {\mathcal C}(j) = \Lambda(j) := \begin{pmatrix}
\lambda_+(j) & 0 \\
0 & \lambda_-(j)
\end{pmatrix}, \quad \forall \, j \in \Z \setminus \{ 0 \} \,.
\end{equation}
\\[1mm]
\noindent
$(iii)$ We define the map
$$
u(y) = \sum_{j \in \Z\setminus\{0\}} \whu(j) e^{\im j y }\in \C^2 \quad \mapsto \quad {\mathcal C}^{\pm 1}[u(y)] := \sum_{j \in \Z \setminus \{ 0 \}} {\mathcal C}(j)^{\pm 1}[\widehat u(j)] e^{\im  x j}\,.
$$
Then, for any $s\geq 0$, one has ${\mathcal C}^{\pm 1} \in {\mathcal B}({\bf H}^s_0)$, with estimate $\| {\mathcal C}^{\pm 1}\|_{{\mathcal B}({\bf H}^s_0)}  \leq C$. 
Moreover, it holds that
$$
{\mathcal C}^{- 1} {\mathtt D}_{\perp \perp} {\mathcal C} 
= \Lambda\,,
$$
where 
\begin{equation}\label{def Lambda}
\Lambda [u](x) = \sum_{j \in \Z \setminus \{ 0 \}} \Lambda(j)[\widehat u(j)] e^{\im j y }\,, \quad \Lambda(j) = \begin{pmatrix}
\lambda_+(j) & 0 \\
0 & \lambda_-(j)
\end{pmatrix}\,,
\end{equation}
such that, for any $s \geq 0$, $\Lambda \in 
{\mathcal B}({\bf H}^{s + 2}_0, {\bf H}^{s}_0)$ 
and is invertible, with inverse 
$\Lambda^{- 1} \in {\mathcal B}({\bf H}^{s}_0, 
{\bf H}^{s + 2}_0  )$ 
satisfying the estimate
$\| \Lambda^{- 1} \|_{
{\mathcal B}({\bf H}^s_0, 
{\bf H}^{s + 2}_0 )} \leq \sigma^{- 1}$, with $\sigma>0$ as in item $(i)$. 
\end{lem}

\begin{proof}
We start with the proof of item $(i)$.
The case ${\rm b}_2=0$ is trivial and it is enough to set $\mathcal{C}(j)={\rm id}$.
Let us analyse the case ${\rm b}_2\neq0$.
By an explicit computation one has that 
the eigenvalues of the $2 \times 2$ matrix 
${\mathtt D}_{\perp\perp}(j)$, $j \in \Z \setminus \{ 0 \}$,
have the form \eqref{eigenval_Dperp.resist-iii}.
One has to analyze the discriminant of 
the eigenvalue equation 
${\rm det}\big( {\mathtt D}_{\perp\perp}(j) 
- \lambda {\rm Id} \big) = 0$ which 
is given by 
\begin{equation}\label{deltadelta}
\Delta(j) :=  (\nu-\eta)^2 j^4 - 4 {\rm b}_{2}^2 j^2 \,.
\end{equation}
Clearly, if $\nu = \eta$, then  one has
\[
\Delta(j) =  - {\rm b}_{2}^2 j^2 < 0\,,
\]
and  the eigenvalues take the form
$$
\lambda_\pm(j) 
= - \nu j^2 \pm \tfrac{\im}{2} |\bar{\rm b}_2| |j|\,, \qquad 
j \in \Z \setminus \{ 0 \}\,,
$$
which easily implies
\[
{\rm Re}(\lambda_{\pm}(j)) = - \nu  |j|^2, \quad \forall \, j \in \Z \setminus \{ 0 \}\,,
\]
and the thesis follows.
We now consider the case $\nu \neq \eta$. Without loss of generality, we assume $\nu > \eta$, as the other case is completely analogous. 
First, we write \eqref{deltadelta} as
$$
\Delta(j) =  (\nu-\eta)^2 j^4  \Big( 1 
- \frac{ 4 {\rm b}_{2}^2 }{(\nu - \eta)^2 j^2} 
\Big)\,,
$$
and hence $\Delta(j) > 0$ if and only if 
$$
|j| > \kappa := \frac{2 |{\rm b}_2|}{|\nu - \eta|}\,.
$$
We define
\begin{equation}\label{j0.tresh}
    j_0 := {\rm min}\big\{ j \in \N : j > \kappa \big\}\,.
\end{equation}
Clearly we have $j_0 > \kappa \geq j_0 - 1$. When $|j| \leq j_0 - 1$, then $\Delta(j) \leq 0$ and the eigenvalues $\lambda_{\pm}(j)$ in \eqref{eigenval_Dperp.resist-iii} take the form
$$
\lambda_{\pm}(j) = - \frac{(\nu + \eta)j^2}{2} \pm \frac{\im}{2} \sqrt{|\Delta(j)|} 
$$
implying that 
$$
{\rm Re}(\lambda_{\pm}(j)) = - \frac{(\nu + \eta)|j|^2}{2}, \quad \forall \, 0 < |j| \leq j_0 - 1\,. 
$$
On the other hand, when
$|j| \geq j_0$,  then $\Delta(j) > 0$ and 
$$
\lambda_\pm(j) = - \frac{(\nu + \eta)|j|^2}{2} \pm \frac12 \sqrt{\Delta(j)} \in \R\,.
$$
Moreover we have the estimates
\begin{align}
 \Delta(j) & = (\nu-\eta)^2 j^4  
 \Big( 1 - \frac{\kappa^2 }{j^2} \Big) 
 \leq (\nu-\eta)^2 j^4  \,,
 \\
 \Delta(j) &= (\nu-\eta)^2 j^4  \Big( 1 
 - \frac{ \kappa^2 }{ j^2} \Big) 
 \geq (\nu-\eta)^2 j^4 (1 - \delta_0^2)\,,
\end{align}
where, recalling the definition of $j_0\in\N$ in \eqref{j0.tresh},
$$
\varpi_0 :=  \frac{ \kappa }{j_0}\in (0,1)\,.
$$
In conclusion, we have
\begin{align}
\lambda_-(j) & \leq - \frac{(\nu + \eta)|j|^2}{2} - \frac{\sqrt{\Delta(j)}}{2}  \leq - \frac{(\nu + \eta)|j|^2}{2} -  \frac{(\nu - \eta)|j|^2}{2} \sqrt{1 - \varpi_0^2}   \\
& \leq - \sqrt{1 - \varpi_0^2} \Big( \frac{(\nu + \eta)|j|^2}{2} +  \frac{(\nu - \eta)|j|^2}{2}\Big) \\
& \leq - \sqrt{1 - \varpi_0^2} \nu |j|^2\,. 
\end{align}
and 
$$
\lambda_+(j) = - \frac{(\nu + \eta)|j|^2}{2} + \frac{\sqrt{\Delta(j)}}{2}  \leq - \frac{(\nu + \eta)|j|^2}{2} +  \frac{(\nu - \eta)|j|^2}{2} = - \eta |j|^2\,. 
$$
Setting
\begin{equation}
    \sigma := \min\big\{\nu,\eta, \sqrt{1-\varpi_0^2}\big\} >0 \,,
\end{equation}
we finally deduce that ${\rm Re}(\lambda_\pm(j)) \leq - \sigma |j|^2$ for any $j \in \Z \setminus \{ 0 \}$ and  the proof of item $(i)$ is concluded.

\noindent
We now prove item $(ii)$. 
Recalling 
\eqref{eigenval_Dperp.resist-iii}, \eqref{deltadelta}, one computes explicitly the (normalized) eigenvector of $\mathtt{D}_{\perp\perp}(j)$
in \eqref{eq:D.perp.Fourier},
obtaining
\[
\mathcal{C}(j):=\begin{pmatrix} 
\dfrac{\im {\rm b}_2 j}{\sqrt{{\rm b}_2^2 j^2 + \Big( \frac{\nu - \eta}{2}j^2 + \frac{\sqrt{\Delta(j)}}{2} \Big)^2}} & 
\dfrac{\im {\rm b}_2 j}{\sqrt{{\rm b}_2^2 j^2 + \Big( \frac{\nu - \eta}{2}j^2 - \frac{\sqrt{\Delta(j)}}{2} \Big)^2}} \\
\dfrac{\frac{\nu - \eta}{2}j^2 + \frac{\sqrt{\Delta(j)}}{2}}{\sqrt{{\rm b}_2^2 j^2 + \Big( \frac{\nu - \eta}{2}j^2 + \frac{\sqrt{\Delta(j)}}{2} \Big)^2}} & 
\dfrac{\frac{\nu - \eta}{2}j^2 - \frac{\sqrt{\Delta(j)}}{2}}{\sqrt{{\rm b}_2^2 j^2 + \Big( \frac{\nu - \eta}{2}j^2 - \frac{\sqrt{\Delta(j)}}{2} \Big)^2}}
\end{pmatrix}\,.
\]
By observing that (still assuming $\nu>\eta$ without loss of generality)
\begin{align}
\sqrt{\Delta(j)}
&\stackrel{\eqref{deltadelta}}{=}
\sqrt{(\nu-\eta)^2j^4-4{\rm b}_2j^2}\sim j^2\,,
\\
\frac{\nu - \eta}{2}j^2 - \frac{\sqrt{\Delta(j)}}{2}&=
\frac{\frac{(\nu-\eta)^2 j^4}{4}-\frac{\Delta(j)}{4}}{\frac{\nu - \eta}{2}j^2 +\frac{\sqrt{\Delta(j)}}{2}}
=\frac{{\rm b}_2^2}{\nu-\eta}
+O\Big(\frac{1}{j^2}\Big)\,, \label{asintotJJ}
\\
\frac{\nu - \eta}{2}j^2 +\frac{\sqrt{\Delta(j)}}{2}&=
(\nu-\eta)j^2
-\frac{{\rm b}_2^2}{\nu-\eta}
+O\Big(\frac{1}{j^2}\Big)\,,
\end{align}
as $|j|\to+\infty$,
one deduces the estimate \eqref{stimadiCJ}
for the matrix $\mathcal{C}(j)$.
Now, it is easy to check that
\[
{\rm det}(\mathcal{C}(j))=\frac{-\im  {\rm b}_2 j \sqrt{\Delta(j)}}{
\sqrt{\Big({\rm b}_2^2 j^2+\Big(\frac{(\nu-\eta)j^2}{2}+\frac{\sqrt{\Delta(j)}}{2}\Big)^2\Big)
\Big({\rm b}_2^2 j^2+\Big(\frac{(\nu-\eta)j^2}{2}-\frac{\sqrt{\Delta(j)}}{2}\Big)^2\Big)}\,.
}
\]
Then, by an explicit computation one gets
\begin{footnotesize}
\begin{equation}
    \cC(j)^{-1}=
\frac{1}{\im {\rm b}_2 j\sqrt{\Delta}}
\begin{pmatrix}
-\Big(\frac{(\nu-\eta)j^2}{2}-\frac{\sqrt{\Delta(j)}}{2}\Big)
\sqrt{\Big(\bar{b}_2^2 j^2+\Big(\frac{(\nu-\eta)j^2}{2}+\frac{\sqrt{\Delta(j)}}{2}\Big)^2\Big)}
& \im {\rm b}_2 j \sqrt{\Big({\rm b}_2^2 j^2+\Big(\frac{(\nu-\eta)j^2}{2}+\frac{\sqrt{\Delta(j)}}{2}\Big)^2\Big)}
\\
\Big(\frac{(\nu-\eta)j^2}{2}+\frac{\sqrt{\Delta(j)}}{2}\Big)
\sqrt{\Big({\rm b}_2^2 j^2+\Big(\frac{(\nu-\eta)j^2}{2}-\frac{\sqrt{\Delta(j)}}{2}\Big)^2\Big)}
&
-\im {\rm b}_2 j \sqrt{\Big({\rm b}_2^2 j^2+\Big(\frac{(\nu-\eta)j^2}{2}-\frac{\sqrt{\Delta(j)}}{2}\Big)^2\Big)}
\end{pmatrix}\,.
\end{equation}
\end{footnotesize}
Using again \eqref{asintotJJ}, 
one deduces the estimate \eqref{stimadiCJ}
for the matrix $\mathcal{C}(j)^{-1}$.
The diagonalization \eqref{merocontoAlg}
and item $(iii)$ simply follow from a straightforward algebraic calculation, together with the estimate \eqref{stimadiCJ}.
\end{proof}

\section{Stability analysis of the zero mode}\label{sect.zeromode}
In this section,
we study the stability properties 
of the zero mode by analyzing the eigenvalues of the operator
$\tL_{\nu,\eta,\epsilon}$ in \eqref{tL.vare}.
The first step
is to conjugate the matrix $\tL_{\nu,\eta,\epsilon}$ 
with a transformation matrix (that will correspond 
to the representation of an invertible map) 
to a normal form matrix operator where the action on the zero
mode  is decoupled from the remaining modes. 
This is done in Section \ref{sec:fixedpoint1}.
Then, in Section \ref{sec:specanalzero1}, we provide a spectral analysis which provides conditions to guarantee the instability 
of the zero mode.

\subsection{Decoupling of zero and non-zero modes}\label{sec:fixedpoint1}

The main result of the section is the following proposition about the existence of a normal form transformation for the operator $\tL_{\nu,\eta,\varepsilon}$.

\begin{prop}\label{prop:fixed_point}
Let $\nu,\eta>0$, $S\geq 0$ and assume $U\in C^{S+2}(\T)$. 
There exists a constant $0<\delta_0=\delta_0(S)\ll1$ such that, 
if 
\begin{equation}\label{smallness.fp}
    \epsilon C_2(\nu,\eta,{\rm b}_2)\leq\delta_0
\end{equation}
with $C_2\equiv C_2(\nu,\eta,{\rm b}_2)>0$ as in Lemma \ref{lem:Dperp_inv.resist},
then there exist maps
\[
T_2\equiv T_2^{\nu,\eta,\epsilon}\in\cB({\bf H}_0^s,\C^2)\,,
\qquad
T_{3}\equiv T_3^{\nu,\eta,\epsilon}\in\cB(\C^2,{\bf H}_0^s)\,,
\]
satisfying
\begin{equation}\label{basicestT}
\|T_2 \|_{\cB(\bH_0^{s},\C^2)}	 
\lesssim \varepsilon^{3}\,,
\qquad 
\|T_3\|_{\cB(\C^2,\bH_0^{s})} \lesssim \varepsilon^{-1} \,,
\end{equation}
such that the following hold:
\\[1mm]
\noindent $\bullet$
 The operator 
\begin{equation}\label{eq:ansatz_conjugation}
	\tT := \begin{bmatrix}
		\rm Id_0 & T_2 \\
		T_3 & \rm Id_\perp
	\end{bmatrix},
\end{equation}
belongs to $\cB({\bf H}^{s},{\bf H}^{s})$ and is invertible, with inverse given by
\begin{equation}\label{tT-1}
		\tT^{-1} =\begin{bmatrix}
			\rm Id_0 & -T_2 \\
			-T_3 & \rm Id_\perp
		\end{bmatrix}\, \begin{bmatrix}
			({\rm Id_0} -T_2T_3)^{-1} & \bf 0^\top \\
			\bf 0 &  ({\rm Id_\perp}-T_3T_2)^{-1}
		\end{bmatrix} \,.
	\end{equation}
\noindent $\bullet$ Recalling  $\tL_{\nu,\eta,\epsilon}$ in \eqref{tL.vare},
we have that
	\begin{equation}\label{eq:conjugationTeorema}
		\tL_{\nu,\eta,\epsilon}^{(0)} := \tT\tL_{\nu,\mu,\varepsilon}\tT^{-1} = \begin{bmatrix}
			\cM_{\nu,\eta,\varepsilon}^{(0)} & 0 \\
			0 & \cN_{\nu,\eta,\varepsilon}^{(0)}
		\end{bmatrix}\, \begin{bmatrix}
			({\rm Id}_0-T_2T_3)^{-1} & \bf 0^\top \\
			\bf 0 & ({\rm Id}_\perp-T_3T_2)^{-1}
		\end{bmatrix}\,,
	\end{equation}
    where
    		\begin{align}
\cM_{\nu,\eta,\varepsilon}^{(0)} &:= \tL_{00} 
- \tL_{0\perp}T_3+T_2\tL_{\perp0}-T_2\tL_{\perp\perp}T_3\,, 
\label{emme0}
\\
\cN_{\nu,\eta,\varepsilon}^{(0)} &:= 
-T_3\tL_{00}T_2 + T_3\tL_{0\perp} 
- \tL_{\perp0}T_2 + \tL_{\perp\perp} \,. \label{enne0}
\end{align}
\noindent $\bullet$ The operator $T_3\in \cB(\C^2,\bH_0^s)$ satisfies 
\begin{equation}\label{eq:estim_order_T3_resist}
    	\| T_3  -\varepsilon^{-1} T_3^{(-1)}\|_{\cB(\C^2,\bH_0^s)} \lesssim 1 \,,
    \end{equation}
    where $T_3^{(-1)}$ is explicitly given by
    \begin{equation}
        T_3^{(-1)} := \tD_{\perp\perp}^{-1} \tL_{\perp 0}^{-1} = \tW_{\perp\perp} \begin{bmatrix}
            -\frac{\im}{\nu}U(y) & -\frac{\im{\rm b}_2}{\nu\eta}(\pa_y^{-1}U)(y) \\
            \frac{\im{\rm b}_2}{\nu\eta}(\pa_y^{-1}U)(y) & \frac{\im}{\eta}U(y)
        \end{bmatrix}, \quad \tW_{\perp\perp}:=\Big( {\rm Id}_\perp - \tfrac{{\rm b}_2^2}{\nu\eta} \pa_{y}^{-2} \Big)^{-1}\,,\label{W.pp}
    \end{equation}
    with estimate $\|T_3^{(-1)}\|_{\cB(\C^2,\bH_0^s)} \lesssim 1$.
\end{prop}
The proof of the above result  involves several arguments 
which are analyzed below.
First of all, we show that, under suitable smallness assumptions,
an operator of the form
\eqref{eq:ansatz_conjugation}
is invertible.
\begin{lem}\label{lem:T_invertible}
	Let $\tT$ be as in \eqref{eq:ansatz_conjugation} and assume that
	\begin{equation}\label{small.for.tT-1}
		\|T_2T_3\|_{\cB(\C^2,\C^2)}\,,\; \|T_3T_2\|_{\cB({\bf H}_0^{s},{\bf H}_0^{s})} \leq \frac12 \,.
	\end{equation}
	Then, the matrix $\tT$ is invertible, with inverse as in \eqref{tT-1}.
	
\end{lem}
\begin{proof}
    We observe that
    \begin{equation}\label{prove.inverse}
     \begin{bmatrix}
			\rm Id_0 & T_2 \\
			T_3 & \rm Id_\perp
		\end{bmatrix}\, \begin{bmatrix}
			\rm Id_0 & -T_2 \\
			-T_3 & \rm Id_\perp
		\end{bmatrix} = \begin{bmatrix}
			\rm Id_0 - T_2T_3 & 0 \\
			0 & \rm Id_\perp - T_3 T_2
		\end{bmatrix} = \begin{bmatrix}
			\rm Id_0 & -T_2 \\
			-T_3 & \rm Id_\perp
		\end{bmatrix}\,\begin{bmatrix}
			\rm Id_0 & T_2 \\
			T_3 & \rm Id_\perp
		\end{bmatrix}\,.
    \end{equation}
    By \eqref{small.for.tT-1}, the diagonal matrix in the middle of \eqref{prove.inverse} is invertible and we deduce \eqref{tT-1}.
\end{proof}
Once the invertibility of $\tT$ is guaranteed, we conjugate 
$\tL_{\nu,\mu,\epsilon}$  to transform it into a more ``{easy to handle}'' form.
\begin{lem}\label{lem:conjugation}
	Let $\tT$ be as in \eqref{eq:ansatz_conjugation} and assume the smallness condition
    \eqref{small.for.tT-1}.
	Then, the matrix $\tL_{\nu,\mu,\epsilon}$ in \eqref{tL.vare} is conjugated by 
    $\tT$ in \eqref{eq:ansatz_conjugation} into the matrix
	\begin{equation}\label{eq:conjugation}
		\tL_{\nu,\eta,\varepsilon}^{(0)} := \tT\tL_{\nu,\mu,\varepsilon}\tT^{-1} = \begin{bmatrix}
			\cM_{\nu,\eta,\varepsilon}^{(0)} & \cB_{\nu,\eta,\varepsilon} \\
			\cC_{\nu,\eta,\varepsilon} & \cN_{\nu,\eta,\varepsilon}^{(0)}
		\end{bmatrix}\, \begin{bmatrix}
			({\rm Id}_0-T_2T_3)^{-1} & \bf 0^\top \\
			\bf 0 & ({\rm Id}_\perp-T_3T_2)^{-1}
		\end{bmatrix}
	\end{equation}
	where the operators 
    $\cM_{\nu,\eta,\varepsilon}^{(0)}\colon \C^2\rightarrow\C^2$
    and $\cN_{\nu,\eta,\varepsilon}^{(0)}\colon H_0^s\rightarrow H_0^s$ are given in 
    \eqref{emme0}-\eqref{enne0} and where 
	\begin{equation}
	 \cB_{\nu,\eta,\varepsilon}\colon H_0^s\rightarrow\C^2\,, 
     \qquad \cC_{\nu,\eta,\varepsilon}\colon \C^2\rightarrow H_0^s, 
	\end{equation}
	are given by
    \begin{equation}\label{gerry1}
		\begin{aligned}
\cB_{\nu,\eta,\varepsilon} &:= -\tL_{00}T_2 
+ \tL_{0\perp} - T_2\tL_{\perp0}T_2 
+ T_2\tL_{\perp\perp}\,, 
\\
\cC_{\nu,\eta,\varepsilon} &:= 
T_3\tL_{00} - T_3\tL_{0\perp}T_3 
+ \tL_{\perp0} - \tL_{\perp\perp}T_3\,, 
\end{aligned}
\end{equation}
\end{lem}
\begin{proof}
    It follows by a direct computation with \eqref{tL.vare}, \eqref{eq:ansatz_conjugation} and Lemma \ref{lem:T_invertible}.
 \end{proof}
In order to conclude the proof of Proposition \ref{prop:fixed_point}
we must find operators $T_2, T_3$ that 
annihilate the entrances $\cB_{\nu,\eta,\epsilon}$ and $\cC_{\nu,\eta,\epsilon}$ in \eqref{gerry1} respectively.
This will be done via a classical Banach fixed point argument, 
which we state here for convenience.
\begin{lem}\label{lem:fixed_point}
	Let $(X,\|\,\cdot\, \|)$ be a Banach space and $Q : X\times X \to  X$ be a bilinear mapping satisfying
	\begin{equation}
		\| Q(u,v) \| \leq c\| u \|  \| v\|\,, \quad \forall \,u,v\in X \,,
	\end{equation}
	for some $c>0$. Let $u_0\in X$ be such that $\| u_0 \| \leq {\rm min}\{ \frac{1}{8c}, \frac{1}{\sqrt{24 c}}\}$. Then, there exists a solution $u\in X$ to the fixed-point problem
	\begin{equation}
		u=F(u):=u_0+Q(u,u)
	\end{equation}
	satisfying
	\begin{equation}
		\| u - u_0\| \leq 3c \| u_0\|^2.
	\end{equation}
\end{lem}
\begin{proof}
We write $u = u_0 + v$. The fixed point equation for $v$ becomes 
\begin{equation}
    v = \Xi(v)\,, \quad \Xi(v) := Q(u_0,u_0) + Q(u_0,v) + Q(v,u_0) + Q(v,v) \,.
\end{equation}
We prove that  $ \Xi : B_X(0, R) \to B_X(0, R)$
is a contraction, where
$$
B_X(0, R) := \big\{ u \in X : \| u \| \leq R \big\} \,, \quad \textnormal{with} \quad R := 3 c \| u_0 \|^2  \,.
$$
Indeed, for $v \in B_X(0, R)$, one has 
\begin{align}
\| \Xi(v) \| & \leq c \| u_0 \|^2 + 2 c \| u_0 \| \| v \| + c \| v \|^2 \leq c \| u_0 \|^2 + 2 c R \| u_0 \| + c R^2 \\
& \leq \tfrac13 R + \tfrac13 R + 3 c^2 \| u_0 \|^2 R \leq R \,,
\end{align}
by the definition of $R$ and using that $\| u_0 \| < \frac{1}{6 c}< \frac{1}{3c}$. Moreover, by the bilinearity of $Q$, for any $\| v_1 \|, \| v_2 \| \leq R$, one has that 
$$
\begin{aligned}
\| \Xi(v_1) - \Xi(v_2) & \| \leq 2 c \| u_0 \| \| v_1 - v_2 \| + (\| v_1 \| + \| v_2 \|) \| v_1 - v_2 \|  \\
& \leq (2 c \| u_0 \| + 2 R) \| v_1 - v_2 \| \leq (2 c \| u_0 \| + 6 c \| u_0 \|^2 ) \| v_1 - v_2 \| \\
&  \leq \tfrac12 \| v_1 - v_2 \| \,.
\end{aligned}
$$
since, by the hypotheses on $u_0$, we have $2 c \| u_0 \| < \frac14$ and $6 c \| u_0 \|^2 < \frac14$.
Hence $\Xi : B_X(0, R) \to B_X (0, R)$ is a contraction and the claim is proved. 
\end{proof}
\noindent
In order to apply Lemma \ref{lem:fixed_point} in the upcoming proof of Proposition \ref{prop:fixed_point},
the key point is the invertibility of the 
linear operators of the form 
\begin{align}
    \Phi(T_2) &:= -\tL_{00}T_2+T_2\tL_{\perp\perp} = -\tL_{00}T_2 + T_2\tR_{\perp\perp} + T_2\tD_{\perp\perp}, \\
    \Psi(T_3) &:= -  \tD_{\perp\perp}T_3 - \tR_{\perp\perp} T_3  +  T_3 \tL_{00} + \e^{- 1}\big( T_3^{(- 1)}  \tL_{0\perp} T_3  + T_3 \tL_{0\perp} T_3^{(-1)} \big) \\
   & \text{for a given }\quad   T_3^{(- 1)} \in {\mathcal B}(\C^2, {\bf H}^s_0), \quad \| T_3^{(- 1)} \|_{{\mathcal B}(\C^2, {\bf H}^s_0)} \lesssim 1\,. 
\end{align}
where (recall \eqref{Lpp}, \eqref{D.perperp})
\begin{equation}
	\tD_{\perp\perp} = \begin{bmatrix}
		\nu\pa_{y}^2 & {\rm b}_2\pa_{y} \\
		{\rm b}_2\pa_{y} &\eta\pa_y^2
	\end{bmatrix}\quad\textnormal{ and }\quad \|\tR_{\perp\perp}\|_{ \cB( {\bf H}_0^s)} \lesssim \varepsilon\,.
\end{equation}
These operators can be rewritten as
\begin{equation}
    \begin{aligned}\label{eq:def_linear.operators_fixed.point}
\Phi(T_2) &= \Phi_0(T_2) + \Phi_1(T_2)\,, 
\qquad \Phi_0(T_2) := T_2\tD_{\perp\perp} \\
\;\;\;\text{and}\;\;\; 
\Phi_1(T_2) & := -\tL_{00}T_2+T_2\tR_{\perp\perp}\,, 
\\
\Psi(T_3) &= \Psi_0(T_3) + \Psi_1(T_3)\,, 
\qquad 
\Psi_0(T_3):=-\tD_{\perp\perp}T_3 \;\;\; \\
\text{and}\;\;\;  
\Psi_1(T_3)& := - \tR_{\perp\perp} T_3  +  T_3 \tL_{00} + \e^{- 1}\big( T_3^{(- 1)}  \tL_{0\perp} T_3  + T_3 \tL_{0\perp} T_3^{(- 1)} \big)\,.
\end{aligned}
\end{equation}

\begin{lem}\label{lem:invert_Phi_0_Psi_0}
Let $\nu,\eta > 0$. The following hold:
\\[1mm]
\noindent $(i)$
The operator $\Phi_0$ 
defined in \eqref{eq:def_linear.operators_fixed.point} 
is invertible, with inverse given by
    \begin{equation}
\forall\,
\wtT_2\in\cB({\bf H}_0^{s-2},\C^2)\,, 
\qquad \Phi_0^{-1}(\wtT_2)
= \wtT_2 \tD_{\perp\perp}^{-1}
\in\cB({\bf H}_0^{s},\C^2)\,,
    \end{equation}
satisfying 
\begin{equation}
\|\Phi_0\|_{\cB({{\bf H}_0^{s-2}},\C^2) 
\rightarrow \cB({\bf H}_0^s,\C^2)} \leq C_1(\nu,\eta,{\rm b}_2)\,, 
\qquad 
\|\Phi_0^{-1}\|_{\cB({\bf H}_0^{s},\C^2) \rightarrow \cB({{\bf H}_0^{s - 2}},\C^2)}  \leq C_2(\nu,\eta,{\rm b}_2)
	\end{equation}
    where the constants $C_i(\nu,\eta,{\rm b}_2)>0$, $i=1,2$, are as in Lemma \ref{lem:Dperp_inv.resist}. Moreover one has that $\Phi_0^{-1}$
    is bounded form $\mathcal{B}({\bf H}_{0}^{s};\C^2)$ into itself with similar estimates as above.
\\[1mm]
\noindent $(ii)$ The operator $\Psi_0$ 
defined in \eqref{eq:def_linear.operators_fixed.point} 
is invertible, with inverse given by:
    \begin{equation}
\forall \,
\wtT_3\in\cB(\C^2,\,{\bf H}_0^{s-2})\,, 
\qquad \Psi_0^{-1}(\wtT_3)
=-\tD_{\perp\perp}^{-1}\wtT_3 
\in\cB(\C^2,{\bf H}_0^{s})\,,
    \end{equation}
satisfying
    \begin{equation}
\|\Psi_0\|_{\cB(\C^2,{\bf H}_0^{s}) 
\rightarrow \cB(\C^2, {\bf H}_0^{s-2})} 
\leq C_1(\nu,\eta,{\rm b}_2)\,, 
\qquad \|\Psi_0^{-1}\|_{\cB(\C^2,{\bf H}_0^{s-2}) 
\rightarrow \cB(\C^2, {\bf H}_0^{s})} \leq C_2(\nu,\eta,{\rm b}_2) \,,
    \end{equation}
     where the constants $C_i(\nu,\eta,{\rm b}_2)>0$, $i=1,2$, are as in Lemma \ref{lem:Dperp_inv.resist}. Moreover one has that $\Psi_0^{-1}$
    is bounded form $\mathcal{B}(\C^2;{\bf H}_{0}^{s})$ into itself with similar estimates as above.
\end{lem}
\begin{proof}
	We prove item $(i)$. By the standard algebra property of the operator norm, we compute
	\begin{align}
		\|\Phi_0\|_{\cB(\bH_0^{s-2},\C^2) \rightarrow \cB(\bH_0^{s},\C^2)} & = \sup_{T_2\neq0} \frac{\|\Phi_0(T_2)\|_{\cB(\bH_0^{s},\C^2)}}{\|T_2\|_{\cB(\bH_0^{s-2},\C^2)}} \\
		& = \sup_{T_2\neq0} \frac{\|T_2\tD_{\perp\perp}\|_{\cB(\bH_{0}^{s},\C^2)}}{\|T_2\|_{\cB(\bH_0^{s-2},\C^2)}} \\
        & \leq \|\tD_{\perp\perp}\|_{\cB(\bH_0^{s} , \bH_0^{s-2})} \leq C_1(\nu,\eta,{\rm b}_2)\,,
	\end{align}
    where the last estimate follows by Lemma \ref{lem:Dperp_inv.resist}.
	Similarly, the estimate of the inverse operator $\Phi_0^{-1}$ is given by
	\begin{equation}
		\|\Phi_0^{-1}\|_{\cB(\bH_0^{s-2},\C^2) \rightarrow \cB(\bH_0^{s},\C^2)} \leq \|\tD_{\perp\perp}^{-1}\|_{\cB(\bH_0^{s-2},\bH_0^{s})} \leq C_2(\nu,\eta,{\rm b}_2)\,.
	\end{equation}
    If instead we assume $\wtT_2\in\cB({\bf H}_0^{s},\C^2)$, we get
    \begin{align}
		\|\Phi_0^{-1}\|_{\cB(\bH_0^{s},\C^2) \rightarrow \cB(\bH_0^{s},\C^2)} & = \sup_{\wtT_2\neq0} \frac{\|\Phi_0^{-1}(\wtT_2)\|_{\cB(\bH_0^{s},\C^2)}}{\|\wtT_2\|_{\cB(\bH_0^{s},\C^2)}} \\
		& = \sup_{\wtT_2\neq0} \frac{\|\wtT_2\tD_{\perp\perp}^{-1}\|_{\cB(\bH_{0}^{s},\C^2)}}{\|\wtT_2\|_{\cB(\bH_0^{s},\C^2)}} \\
        & \leq \|\tD_{\perp\perp}^{-1}\|_{\cB(\bH_0^{s})} 
        \leq  \|\tD_{\perp\perp}^{-1}\|_{\cB(\bH_0^{s-2} , \bH_0^{s})} 
        \stackrel{\eqref{est.D.perp.resist}}{\leq} C_2(\nu,\eta,{\rm b}_2)\,,
	\end{align}
    The proof of item $(ii)$ follows by analogous arguments and is therefore omitted.
\end{proof}



We are now in position to prove the main result of this section.
\begin{proof}[{\bf Proof of Proposition \ref{prop:fixed_point}}]
We split the proof in two parts.  We start with the equation for $T_3$, as it turns out to be the most complicated one. Later, we solve the equation for $T_2$.
    \\[1mm]
\noindent $\blacktriangleright$ {\sc Fixed point for $T_3$.}
	Recalling $\cC_{\nu,\eta,\epsilon}$ in \eqref{gerry1}, we want to solve
	\begin{equation}\label{eq:T3}
		\cC_{\nu,\eta,\varepsilon} = T_3\tL_{00} - T_3\tL_{0\perp}T_3 + \tL_{\perp0} - \tL_{\perp\perp}T_3 = 0\,.
	\end{equation}
    We look for $T_3$ of the form
    \begin{equation}\label{ansazt.T3.expand}
        T_3 = \varepsilon^{-1} T_3^{(-1)} + T_3^{(0)}\,,  \quad T_3^{(0)} \equiv T_3^{(0)}(\varepsilon) \,,
    \end{equation}
    where $T_3^{(0)} \in {\mathcal B}(\C^2, {\bf H}^s_0)$ of order $O(1)$ with respect to $\e$. We proceed in the following way: first, we explicitly compute the leading order term $T_3^{(-1)}$; then, we determine $T_3^{(0)}$ by a fixed point argument.
    \\[1mm]
   \noindent Inserting the ansatz \eqref{ansazt.T3.expand} into the homological equation \eqref{eq:T3} together with \eqref{L00.esp}, \eqref{Lop.eps}, \eqref{Lp0.eps} and \eqref{Lpp.eps} in Lemma \ref{espandi.e.non.mollo}, we get
    \begin{align}
      &  \varepsilon^{-1} \big( \tL_{\perp0}^{(-1)} - \tD_{\perp\perp}T_3^{(-1)} \big) -  \tD_{\perp\perp}T_3^{(0)} + \e^{- 1} T_3^{(- 1)} \tL_{00} +  T_3^{(0)} \tL_{00} \\
      & + \e^{- 2}T_3^{(- 1)} \tL_{0\perp} T_3^{(- 1)} + \e^{- 1}T_3^{(- 1)} \tL_{0\perp} T_3^{(0)}  + \e^{- 1}T_3^{(0)} \tL_{0\perp} T_3^{(- 1)} + T_3^{(0)} \tL_{0\perp} T_3^{(0)}  \label{T3.espanxion.homol} \\
      & +\varepsilon \tL_{\perp 0}^{(1)} - \e^{- 1}\tR_{\perp\perp} T_3^{(- 1)} - \tR_{\perp\perp} T_3^{(0)}  = 0 \,.
        \end{align}
    We determine $T_3^{(-1)}$ by solving the highest order $\varepsilon^{-1}$ in \eqref{T3.espanxion.homol}, namely
    \begin{equation}\label{fix.T3-1}
       \tL_{\perp0}^{(-1)} - \tD_{\perp\perp}T_3^{(-1)}= 0 \quad \Rightarrow \quad  T_3^{(-1)}=\tD_{\perp\perp}^{-1}\tL_{\perp0}^{(-1)}.
    \end{equation}
    In particular, by Lemma \ref{lem:Dperp_inv.resist} and \eqref{Lp0.eps}, we compute
        \begin{align}
            T_3^{(-1)} &= \Big( {\rm Id}_\perp - \tfrac{{\rm b}_2^2}{\nu\eta} \pa_{y}^{-2} \Big)^{-1} \begin{bmatrix}
            \frac{1}{\nu}\pa_y^{-2} & -\frac{{\rm b}_2}{\nu\eta}\pa_y^{-3} \\
            -\frac{{\rm b}_2}{\nu\eta}\pa_y^{-3} & \frac{1}{\eta}\pa_y^{-2}
        \end{bmatrix} \begin{bmatrix}
            -\im U''(y) & 0 \\
            0 & \im U''(y)
        \end{bmatrix} \\
        &= \tW_{\perp\perp} \begin{bmatrix}
            -\frac{\im}{\nu}U(y) & -\frac{\im{\rm b}_2}{\nu\eta}(\pa_y^{-1}U)(y) \\
            \frac{\im{\rm b}_2}{\nu\eta}(\pa_y^{-1}U)(y) & \frac{\im}{\eta}U(y)
        \end{bmatrix} \,.
        \end{align}
        This proves \eqref{W.pp}. The estimate $\|T_3^{(-1)}\|_{\cB(\C^2,\bH_0^s)} \lesssim 1$ follows by the explicit definition of $T_3^{(-1)}$. We now determine $T_3^{(0)}$ with a fixed point argument. In particular,  from \eqref{T3.espanxion.homol} and \eqref{fix.T3-1}, recalling \eqref{eq:def_linear.operators_fixed.point} we find that $T_3^{(0)}$ has to solve
        \begin{equation}
            \Psi(T_3^{(0)}) + F - T_3^{(0)} \tL_{0\perp} T_3^{(0)} = 0\,,
        \end{equation}
        where
\begin{align}
& F := \e^{- 1} T_3^{(- 1)} \tL_{00} + \e^{- 2}T_3^{(- 1)} \tL_{0\perp} T_3^{(- 1)} +\varepsilon \tL_{\perp 0}^{(1)} - \e^{- 1}\tR_{\perp\perp} T_3^{(- 1)} \\
& \Psi(T_3^{(0)}) := \Psi_0(T_3^{(0)}) + \Psi_1(T_3^{(0)})\,, \quad \Psi_0(T_3^{(0)}) = -  \tD_{\perp\perp}T_3^{(0)}  \label{eq senza ordine - 1} \\   
& \Psi_1(T_3^{(0)}) =   T_3^{(0)} \tL_{00} + \e^{- 1}\big( T_3^{(- 1)} \tL_{0\perp} T_3^{(0)}  + T_3^{(0)} \tL_{0\perp} T_3^{(- 1)}\big) - \tR_{\perp\perp} T_3^{(0)}\,.
\end{align}
        By Lemma \ref{espandi.e.non.mollo}, one has that 
        \begin{equation}\label{stima F}
\begin{aligned}
\| F \|_{{\mathcal B}(\C^2, {\bf H}^s_0)} & \leq \e^{- 1}\| T_3^{(- 1)} \|_{{\mathcal B}(\C^2, {\bf H}^s_0)} \| \tL_{00} \|_{{\mathcal B}(\C^2)}+  \e^{- 2}\| T_3^{(- 1)}\|_{{\mathcal B}(\C^2, {\bf H}^s_0)} \| \tL_{0\perp}\|_{{\mathcal B}({\bf H}^s_0, \C^2)} \| T_3^{(- 1)}\|_{{\mathcal B}(\C^2, {\bf H}^s_0)}  \\
& \quad + \e \| \mathtt L_{0 \perp}^{(1)} \|_{{\mathcal B}(\C^2, {\bf H}^s_0)} + \e^{- 1}\| \tR_{\perp\perp}\|_{{\mathcal B}({\bf H}^s_0)} \| T_3^{(- 1)}\|_{{\mathcal B}(\C^2, {\bf H}^s_0)}  \\
& \lesssim \e^{- 1} \e + \e^{- 2}\e^3 + \e \lesssim 1
\end{aligned}
        \end{equation}
      and, for any $T_3^{(0)} \in {\mathcal B}(\C^2, {\bf H}^s_0)$, 
        \begin{equation}\label{stima Phi 1 nel lemma}
        \begin{aligned}
\| \Phi_1(T_3^{(0)}) \|_{{\mathcal B}(\C^2, {\bf H}^s_0)} & \leq \| T_3^{(0)} \|_{{\mathcal B}(\C^2, {\bf H}^s_0)} \| \tL_{00} \|_{{\mathcal B}(\C^2)}+ 2 \e^{- 1}\| T_3^{(- 1)}\|_{{\mathcal B}(\C^2, {\bf H}^s_0)} \| \tL_{0\perp}\|_{{\mathcal B}({\bf H}^s_0, \C^2)} \| T_3^{(0)}\|_{{\mathcal B}(\C^2, {\bf H}^s_0)}  \\
& \quad + \| \tR_{\perp\perp}\|_{{\mathcal B}({\bf H}^s_0)} \| T_3^{(0)}\|_{{\mathcal B}(\C^2, {\bf H}^s_0)} \\
& \lesssim (\e + \e^{- 1} \e^3) \| T_3^{(0)}\|_{{\mathcal B}(\C^2, {\bf H}^s_0)} \lesssim \e \| T_3^{(0)}\|_{{\mathcal B}(\C^2, {\bf H}^s_0)}\,. 
\end{aligned}
        \end{equation}
       Since $\Psi_0$ is invertible by Lemma \ref{lem:invert_Phi_0_Psi_0}, and, together with the latter estimate,
       $$
\| \Psi_0^{- 1} \Psi_1 \|_{{\mathcal B}(\C^2, {\bf H}^s) \to {\mathcal B}(\C^2, {\bf H}^s) } \leq C_0 \e \leq \tfrac12
       $$
       for some constant $C_0>0$ and for $\e \ll 1$ small enough, by Neumann series one has that $\Psi^{- 1} = ({\rm Id} + \Psi_0^{- 1} \Psi_1) \Psi_0^{- 1}$ and 
       \begin{equation}\label{stima Psi inv nel lemma}\| \Psi^{- 1} \|_{{\mathcal B}(\C^2, {\bf H}^s) \to {\mathcal B}(\C^2, {\bf H}^s)} \lesssim 1 \,. \end{equation} 
       Thus the equation \eqref{eq senza ordine - 1} in the unknown $T_3^{(0)}$ is equivalent to the fixed point equation
       \begin{equation}\label{punto fisso per T 3 (0)}
       \begin{aligned}
& T_3^{(0)} = - \Psi^{- 1}(F) + {\mathcal Q}(T_3^{(0)}, T_3^{(0)}), \\ 
&  {\mathcal Q}(G_1, G_2) :=  \Psi^{- 1} G_1 \tL_{0\perp} G_2, \quad G_1, G_2 \in {\mathcal B}(\C^2, {\bf H}^s_0)\,. 
\end{aligned}
       \end{equation}
       By the estimates \eqref{stima F}, \eqref{stima Psi inv nel lemma} and again by Lemma \ref{espandi.e.non.mollo}, 
       one has 
       \begin{equation}\label{stima cal Q Psi F}
       \begin{aligned}
       \| \Psi^{- 1}(F) \|_{{\mathcal B}(\C^2, {\bf H}^s_0)} & \leq K_0  \\
\| {\mathcal Q}(G_1, G_2) \|_{{\mathcal B}(\C^2, {\bf H}^s_0)} & \leq K_1  \e^3 \| G_1 \|_{{\mathcal B}(\C^2, {\bf H}^s_0)} \| G_2 \|_{{\mathcal B}(\C^2, {\bf H}^s_0)} 
\end{aligned}
       \end{equation}
       for some constants $K_0, K_1 > 0$. We want to apply the abstract fixed point Lemma \ref{lem:fixed_point}. First, take $c > 0$ as
    \begin{equation}
        c := \frac12 {\rm min}\Big\{ \frac{1}{8 K_0}, \frac{1}{24 K_0^2} \Big\}\,, \quad \textnormal{ so that } \quad   K_0 < \frac{1}{8 c}\,, \ \ K_0^2 < \frac{1}{2 4 c} \,.
    \end{equation}
       Then, for $0 < \e \ll 1$ small enough, we guarantee that $K_1 \e^3 < c$.  The application of Lemma \ref{lem:fixed_point} allows us to deduce the existence of a unique $T_3^{(0)} \in {\mathcal B}(\C^2, {\bf H}^s_0)$ solving \eqref{eq senza ordine - 1} with estimates $\| T_3^{(0)} \|_{{\mathcal B}(\C^2, {\bf H}^s_0)} \lesssim 1$, as claimed in \eqref{eq:estim_order_T3_resist}. The claimed bound for $T_3$ in \eqref{basicestT} then follows since $T_3 = \e^{- 1} T_3^{(- 1)} + T_3^{(0)}$. 
\\[1mm]
\noindent $\blacktriangleright$ {\sc Fixed point for $T_2$.}
	Recalling $\cB_{\nu,\eta,\epsilon}$ in \eqref{gerry1},
    we want to solve
	\begin{equation}\label{eq:T2}
\cB_{\nu,\eta,\varepsilon} = -\tL_{00}T_2 + \tL_{0\perp} - T_2\tL_{\perp0}T_2 + T_2\tL_{\perp\perp} = 0\,.
	\end{equation}
    We know by Lemma \ref{lem:invert_Phi_0_Psi_0}-$(i)$ that $\Phi_0$ is invertible. 
    Thus, we rewrite $\Phi$ 
    defined in \eqref{eq:def_linear.operators_fixed.point}
    as
    \begin{equation}
    \Phi=\Phi_0({\rm Id_0} + \Phi_0^{-1}\Phi_1)\,,
	\end{equation}
    where ${\rm Id}_0$ is the identity on $\cB({\bf H}_0^{s},\,\C^2) $ and
    \begin{align}
&\Phi_0^{-1} : \cB({\bf H}_0^{s},\,\C^2) 
\rightarrow \cB({\bf H}_0^{s},\,\C^2)\,, 
\qquad
\Phi_1 : \cB({\bf H}_0^{s },\,\C^2) 
\rightarrow \cB({\bf H}_0^{s},\, \C^2)\,, 
\\
&{\rm Id}_0 + \Phi_0^{-1}\Phi_1 
: \cB({\bf H}_0^{s},\,\C^2) 
\rightarrow \cB({\bf H}_0^{s},\,\C^2).
    \end{align}
	By using Lemma \ref{espandi.e.non.mollo} one verifies that
    $\|\Phi_1\|_{\cB({\bf H}_0^{s},\C^2) \rightarrow \cB({\bf H}_0^{s},\,\C^2)} \lesssim \varepsilon$ 
    and hence we obtain that
    \begin{equation}
        \|\Phi_0^{-1}\Phi_1\|_{\cB({\bf H}_0^{s},\, \C^2) \rightarrow \cB({\bf H}_0^{s},\, \C^2)} 
        \lesssim
        \|\Phi_0^{-1}\|_{\cB({\bf H}_0^{s},\, \C^2)\rightarrow 
        \cB({\bf H}_0^{s},\C^2)} 
        \|\Phi_1\|_{\cB({\bf H}_0^{s},\, \C^2)\rightarrow 
        \cB({\bf H}_0^{s},\C^2)} \leq C
        \varepsilon\,,
    \end{equation}
   for some constant $C>0$.  Therefore, for $\varepsilon\ll 1$ small enough such that $C \e < \frac12$, we can invert ${\rm Id}_0 + \Phi_0^{-1}\Phi_1$ using Neumann series and conclude that the whole $\Phi$ is invertible,
    with inverse given by 
    \begin{equation}\label{stimaPhiPsimeno1}
    \Phi^{-1} = ({\rm Id_0} + \Phi_0^{-1}\Phi_1)^{-1}\Phi_0^{-1}\,, \quad \| \Phi^{- 1} \|_{{\mathcal B}({\bf H}^s_0, \C^2) \to {\mathcal B}({\bf H}^s_0, \C^2)} \lesssim 1
    \end{equation}
    We now go back to \eqref{eq:T2} and we rewrite it as a fixed point equation, namely
	\begin{equation}
		T_2 = -\Phi^{-1}\tL_{0\perp} + \Phi^{-1}[T_2\tL_{\perp0}T_2] =: T_{2,0} + Q_2(T_2,T_2) \,.
	\end{equation}
  \begin{equation}
        T_2  = 
        T_{2,0} + Q_2(T_2,T_2)
    \end{equation}
    where the quadratic part and the free term are given by
	\begin{equation}
			Q_2(G_1, G_2) := \Phi^{-1}[G_1\tL_{\perp0} G_2]\,, \quad G_1, G_2 \in {\mathcal B}({\bf H}^s_0, \C^2)\,,
            \qquad 
			T_{2,0} := -\Phi^{-1}\tL_{0\perp}\,,
	\end{equation}
    with
    \begin{equation}
Q_2 : \cB({\bf H}_0^{s},\C^2) \times \cB(
{\bf H}_0^{s},\C^2) 
\rightarrow \cB({\bf H}_0^{s},\C^2)\,, 
\qquad 
        T_{2,0} \in \cB({\bf H}_0^{s},\C^2)\,.
    \end{equation}
	By \eqref{stimaPhiPsimeno1} and by \eqref{esti.easy} in Lemma \ref{espandi.e.non.mollo}, one has the estimates
    \begin{align*}
    \| T_{2, 0} \|_{\cB({\bf H}_0^s,\C^2)} & \leq K_0 \e^3 \,, \\
\|Q_2(G_1, G_2)\|_{\cB({\bf H}_0^s,\C^2)} &
\stackrel{\eqref{stimaPhiPsimeno1}}{\lesssim}
\| G_1 \tL_{\perp0} G_2 \|_{\cB({\bf H}_0^s,\C^2)} \\
& \lesssim 
\|G_1 \|_{\cB({\bf H}_0^s;\C^2)} \|G_2 \|_{\cB({\bf H}_0^s,\C^2)}
\|\tL_{\perp0} \|_{\cB(\C^2,{\bf H}_0^s)} \\
& \stackrel{\eqref{esti.easy}}{\leq}
K_1 \e^{- 1}\|G_1\|_{\cB({\bf H}_0^s,\C^2)} \|G_2\|_{\cB({\bf H}_0^s,\C^2)}\,,
    \end{align*}
    for some constants $K_0, K_1 >0$.   
	We verify the hypotheses of Lemma \ref{lem:fixed_point}. Let 
    $c := K_1 \e^{- 1}$.  We need that 
    $$
K_0 \e^3 < \frac{1}{8 c} = \frac{\e}{8 K_1}, \quad K_0^2 \e^6 < \frac{1}{24 c} = \frac{\e}{24 K_1}.
    $$
    Clearly, the latter conditions are verified for $0 <\e \ll 1$ small enough. Hence, there exists a unique fixed point $T_2 \in {\mathcal B}({\bf H}^s_0, \C^2)$ satisfying 
    $$
\| T_2 - T_{2, 0} \|_{{\mathcal B}({\bf H}^s_0, \C^2)} \leq 3 c \| T_{2, 0} \|_{{\mathcal B}({\bf H}^s_0, \C^2)}^2 \lesssim \e^5\,, 
    $$
and, by triangular inequality, $T_2$ satisfies the estimate in \eqref{basicestT}. The proof of the proposition is then concluded. 
\end{proof}

\begin{rmk}\label{rmk:order_T2.T3}
	From estimates \eqref{basicestT}, one  deduces that
	\begin{equation}
		\|T_2T_3\|_{\cB(\C^2)}\,, 
        \|T_3T_2\|_{\cB({\bf H}_0^s)} 
        \lesssim \varepsilon^2 \ll 1\,,
	\end{equation}
	so thanks to Lemma \ref{lem:T_invertible}, both $({\rm Id_0}-T_2T_3)$ and $({\rm Id_\perp}-T_3T_2)$ are invertible and thus $\tT^{-1}$ is well-defined. However, although $T_2T_3$ and $T_3T_2$ are both small (in the previous norms), the solution $T_3$ is not small as $T_3=\cO(\epsilon^{-1})$. Therefore, the ansatz \eqref{eq:ansatz_conjugation} is not close to the identity and we are actually not in a perturbative regime.
\end{rmk}

\subsection{Spectral analysis of the zero mode}\label{sec:specanalzero1}

We want to analyze the spectral properties of our operator. We start with the first entrance of \eqref{eq:conjugation},
\begin{align}\label{eq:1st_entrance_resist}
	M_{\nu,\eta}^{(0)}  = M_{\nu,\eta}^{(0)}(\epsilon) &:= \cM_{\nu,\eta,\varepsilon}^{(0)}({\rm Id}_0-T_2 T_3)^{-1} \\
    & = (\tL_{00} - \tL_{0\perp}T_3+T_2\tL_{\perp0}-T_2\tL_{\perp\perp}T_3)({\rm Id}_0-T_2 T_3)^{-1} \,.
\end{align}
In Lemma \ref{espandi.e.non.mollo}, we expanded each entrance of $\tL_{\nu,\eta,\epsilon}$ in powers of $\varepsilon$, with $\varepsilon$-independent coefficients, whereas, in Proposition \ref{prop:fixed_point}-$(ii)$, we expanded the entrance $T_3$ of the conjugation map $\tT$.  The goal now is to study the explicit $\varepsilon$-dependence of $M_{\nu,\eta}^{(0)}(\epsilon)$ and compute an explicit expansion of its eigenvalues in order to impose conditions on their real parts.


We need to compute the explicit expansion of the matrix \eqref{eq:1st_entrance_resist} up to order $\varepsilon^2$.
\begin{prop}\label{prop.M2.resist}
Under the assumptions of Proposition \ref{prop:fixed_point},
    the zeroth mode block $M_{\nu,\eta}^{(0)}$ in \eqref{eq:1st_entrance_resist} has the following explicit expansion
    \begin{equation}\label{M.expansion}
        M_{\nu,\eta}^{(0)}(\epsilon) =  \tM_{\nu,\eta}^{(2)} + \tM_{\nu,\eta}^{(3)}(\varepsilon)\,,
    \end{equation}
    where
    \begin{equation}\label{eq:def_M2.zeroth.block.resist}
        \tM_{\nu,\eta}^{(2)} := \begin{bmatrix}
            -\epsilon^2\nu + \frac{\epsilon^2}{\nu}\la \tW_{\perp\perp}\pa_y^{-1}U,\pa_y^{-1}U\ra & \im\epsilon{\rm b}_1 \\
            \im\epsilon{\rm b}_1 & -\epsilon^2\eta - \frac{\epsilon^2}{\eta}\la \tW_{\perp\perp}\pa_y^{-1}U,\pa_y^{-1}U\ra
        \end{bmatrix}\,,
    \end{equation}
    with $\tW_{\perp\perp}$ is as in \eqref{W.pp}, and $\tM_{\nu,\eta}^{(3)}(\varepsilon)$ satisfies $\| \tM_{\nu,\eta}^{(3)}(\varepsilon)\|_{\cB(\C^2)} \lesssim \varepsilon^3$.
\end{prop}
\begin{proof}
We split the proof in steps.
\\[1mm]
\noindent $\blacktriangleright$ {\sc Step 1.} We first want to rewrite \eqref{eq:1st_entrance_resist} with a more compact formula. Using that $T_3$ solves the homological equation \eqref{eq:T3}, we deduce that
\begin{equation}\label{Lpp.T3}
    \tL_{\perp\perp} T_3 = T_3 \tL_{00} - T_3 \tL_{0\perp} T_3 + \tL_{\perp 0} \,.
\end{equation}
Inserting \eqref{Lpp.T3} into \eqref{eq:1st_entrance_resist}, we compute
\begin{align}
    M_{\nu,\eta}^{(0)}(\varepsilon) & = (\tL_{00} - \tL_{0\perp}T_3+T_2\tL_{\perp0}-T_2\tL_{\perp\perp}T_3)({\rm Id}_0-T_2 T_3)^{-1} \\
    & = \big(\tL_{00} - \tL_{0\perp}T_3+T_2\tL_{\perp0} -T_2T_3 \tL_{00} +T_2 T_3 \tL_{0\perp} T_3 -T_2\tL_{\perp 0}  \big)({\rm Id}_0-T_2 T_3)^{-1} \\
    & = ({\rm Id}_0-T_2 T_3) \big(\tL_{00} - \tL_{0\perp}T_3   \big)({\rm Id}_0-T_2 T_3)^{-1} \,.
\end{align}
Now, recalling \eqref{L00.esp}, \eqref{Lop.eps}, \eqref{eq:estim_order_T3_resist} and Remark \ref{rmk:order_T2.T3}, we deduce that $M_{\nu,\eta}^{(0)}$ satisfies the expansion in \eqref{M.expansion}, where $\tM_{\nu,\eta}^{(2)}$ is given by
\begin{equation}\label{M2.to.write}
    \tM_{\nu,\eta}^{(2)} := \epsilon\tL_{00}^{(1)} + \epsilon^2\big(\tL_{00}^{(2)}- \tL_{0\perp}^{(3)}T_3^{(-1)}\big) \,,
\end{equation}
and $\tM_{\nu,\eta}^{(3)}(\varepsilon)$ is some matrix collecting all the lower terms, satisfying the estimate $\| \tM_{\nu,\eta}^{(3)}(\varepsilon)\|_{\cB(\C^2)} \lesssim \varepsilon^3$.
\\[1mm]
\noindent $\blacktriangleright$ {\sc Step 2.} We are left to explicitly write the matrix $\tM_{\nu,\eta}^{(2)}$ in  \eqref{M2.to.write}.
We first note that both $\pa_y^{-k}$ and $\tW_{\perp\perp}$ are well-defined Fourier multipliers on $H_0^s$, so they commute with each other. Observe also that the operator $\tW_{\perp,\perp}$ is self-adjoint. Therefore,
by \eqref{Lop.eps} and Lemma \ref{prop:fixed_point}-$(iii)$, we compute
        \begin{align}
            -\tL_{0\perp}^{(3)}T_3^{(-1)} &= \begin{bmatrix}
                (\im\pa_y^{-2}U(y))^T & 0^T \\
                0^T & (\im\pa_y^{-2}U(y))^T
            \end{bmatrix} \tW_{\perp\perp} \begin{bmatrix}
                -\frac{\im}{\nu}U(y) & -\frac{\im{\rm b}_2}{\nu\eta}(\pa_y^{-1}U)(y) \\
                \frac{\im{\rm b}_2}{\nu\eta}(\pa_y^{-1}U)(y) & \frac{\im}{\eta}U(y)
            \end{bmatrix} \\
            &= \begin{bmatrix}
                \la -\frac{\im}{\nu}\tW_{\perp\perp}U,\im\pa_y^{-2}U \ra & \la \frac{-\im{\rm b}_2}{\nu\eta}\tW_{\perp\perp}\pa_y^{-1}U, \im\pa_y^{-2}U \ra \\
                \la \frac{\im{\rm b}_2}{\nu\eta}\tW_{\perp\perp}\pa_y^{-1}U, \im\pa_y^{-2}U \ra & \la \frac{\im}{\eta}\tW_{\perp\perp}U,\im\pa_y^{-2}U \ra
            \end{bmatrix} \\
            &= \begin{bmatrix}
                \frac{1}{\nu}\la \tW_{\perp\perp}\pa_y^{-1}U,\pa_y^{-1}U \ra & 0 \\
                0 & -\frac{1}{\eta}\la \tW_{\perp\perp}\pa_y^{-1}U,\pa_y^{-1}U \ra
            \end{bmatrix} \,. \label{compu1}
        \end{align}
    Inserting \eqref{compu1} into \eqref{M2.to.write} and recalling \eqref{L00.esp}, we conclude the claimed expansion of $M_{\nu,\eta}^{(0)}$ in \eqref{M.expansion}, \eqref{eq:def_M2.zeroth.block.resist}.
\end{proof}
\begin{rmk}\label{rmk.W.pp.pos}
    At this point, it is useful to point out that the operator $\tW_{\perp\perp}$ in \eqref{W.pp} is positive definite. Indeed, for any $f\in L_0^2(\T)$, $f\neq 0$,
    \begin{equation}
        \la \tW_{\perp\perp}f,f \ra \simeq \sum_{j\neq0} \bigg(1 + \frac{|{\rm b}_2|^2}{\nu\eta j^2} \bigg)^{-1} |\whf(j)|^2>0.
    \end{equation}
    In particular, we have that $\la \tW_{\perp\perp}\pa_y^{-1}U,\pa_y^{-1}U \ra>0$.
\end{rmk}

We are finally in position to state and prove the main proposition that ensures the presence of an unstable eigenvalue associated with the zeroth mode block $M_{\nu,\eta}^{(0)}$ in \eqref{eq:1st_entrance_resist}.
\begin{prop}\label{prop:positive.zeromode}
    Let $\eta, \nu >0$. There exists $\varepsilon_0\equiv \varepsilon_0(S,\nu,\eta,\bb) \in (0, 1) $ small enough such that, for any $\varepsilon \in (0,\varepsilon_0)$,  the following hold:
    \begin{enumerate}
        \item[(i)] When ${\rm b}_1\neq0$, if $\eta>\nu$ and $\la \tW_{\perp\perp}\pa_y^{-1}U,\pa_y^{-1}U\ra>\frac{(\nu+\eta)\nu\eta}{\eta-\nu}$, then the $2\times2$ matrix $M_{\nu,\eta}^{(0)}$ has two positive eigenvalues $\mu_\pm \equiv \mu_\pm^{\nu, \eta, \e}$ with positive real part having the asymptotic expansion 
        \begin{equation}\label{explago1}
{\rm Re}(\mu_\pm) =  \e^2 {\mathtt c} + O(\e^3)\,,
\qquad {\rm Im}(\mu_\pm)=O(\e)\,,
\end{equation}
where
$$
\mathtt{c}:=-(\nu+\eta) + \Big(\frac{1}{\nu}-\frac{1}{\eta}\Big)\la \tW_{\perp\perp}\pa_y^{-1}U,\pa_y^{-1}U\ra \,;
$$
        \item[(ii)] When ${\rm b}_1=0$, if $\la \tW_{\perp\perp}\pa_y^{-1}U,\pa_y^{-1}U\ra>\nu^2$, then the $2\times2$ matrix $M_{\nu,\eta}^{(0)}$ has exactly one 
        eigenvalue $\mu_+ \equiv \mu_{+}^{\nu,\eta,\e}$ with positive real part, whereas the second $\mu_- \equiv \mu_-^{\nu, \eta, \e}$
        one has negative real part. Moreover one has the
      the asymptotic expansion 
              $$
{\rm Re}(\mu_+) =   \e^2 {\mathtt c}_1 + O(\e^3)\,, \quad {\rm Re}(\mu_-) =  - \e^2 {\mathtt c}_2 + O(\e^3)\,,
\quad {\rm Im}(\mu_\pm)=O(\e^3)\,,
$$
where
\begin{equation}
    \mathtt{c}_1:=\nu^{-1}\big( -\nu^2 + \la \tW_{\perp\perp}\pa_y^{-1}U,\pa_y^{-1}U\ra \big) > 0\,, \quad  \mathtt c_2 := \eta^{- 1} \big( \eta^2 + \la \tW_{\perp\perp}\pa_y^{-1}U,\pa_y^{-1}U\ra \big)  > 0\,;
\end{equation}
\item[(iii)] When ${\rm b}_1\neq0$, if $\nu \geq \eta$, the eigenvalues $\mu_{\pm}\equiv \mu_{\pm}^{\nu,\eta,\varepsilon}$ have both negative real part. More precisely 
$$
{\rm Re}(\mu_{\pm }) = - \e^2 \mathtt c + O(\e^3), \quad {\rm Im}(\mu_{\pm }) = O(\e)\,. 
$$
for some constant $\mathtt c \geq \frac{\nu + \eta}{4}$. 
    \end{enumerate}
    
\end{prop}

\begin{proof}
    By \eqref{M.expansion}, \eqref{eq:def_M2.zeroth.block.resist} in Proposition \ref{prop.M2.resist}, we write the entries of the matrix $M_{\nu,\eta}^{(0)}(\varepsilon)$ as
    \begin{equation}\label{M0.entries}
        M_{\nu,\eta}^{(0)}(\varepsilon) := \begin{bmatrix}
            -\varepsilon^2\nu + \frac{\varepsilon^2}{\nu}\la \tW_{\perp\perp}\pa_y^{-1}U,\pa_y^{-1}U\ra + O(\varepsilon^3) & \im\varepsilon{\rm b}_1 + O(\varepsilon^3) \\
            \im\varepsilon{\rm b}_1+ O(\varepsilon^3) & -\varepsilon^2\eta - \frac{\varepsilon^2}{\eta}\la \tW_{\perp\perp}\pa_y^{-1}U,\pa_y^{-1}U\ra + O(\varepsilon^3)
        \end{bmatrix}\,.
    \end{equation}
    The eigenvalues of $M_{\nu,\eta}^{(0)}(\varepsilon)$ are given by the following expression
    \begin{equation}\label{mu.eigen.M2}
        \mu_\pm = \frac{{\rm Tr}(M_{\nu,\eta}^{(0)}(\varepsilon))}{2} \pm \sqrt{\bigg(\frac{{\rm Tr}(M_{\nu,\eta}^{(0)}(\varepsilon))}{2}\bigg)^2 - {\rm det}(M_{\nu,\eta}^{(0)}(\varepsilon))} \,=: \frac{{\rm Tr}(M_{\nu,\eta}^{(0)}(\varepsilon))}{2} \pm \sqrt{\Delta_{\nu,\eta,\varepsilon}}\,.
    \end{equation}
    We have that
    \begin{align}\label{traccia.M0}
        {\rm Tr}(M_{\nu,\eta}^{(0)}(\varepsilon)) & = \varepsilon^2\bigg(-(\nu+\eta) + \Big(\frac{1}{\nu}-\frac{1}{\eta}\Big)\la \tW_{\perp\perp}\pa_y^{-1}U,\pa_y^{-1}U\ra \bigg) + O(\varepsilon^3) \,.
    \end{align}
    To compute $\Delta_{\nu,\eta,\varepsilon}$, we observe that, for a generic $2\times 2$ matrix $M=\big( \begin{smallmatrix}
        M_1 & M_2\\
        M_3 & M_4
    \end{smallmatrix}\big)$, one clearly has
    \begin{equation}
        {\rm Tr}(M) = M_{1}+M_{4} \quad\text{and}\quad \det(M)=M_{1}M_{4}-M_{2}M_{3}\,,
    \end{equation}
    which implies that
    \begin{equation}\label{conto.matrix}
        \bigg(\frac{{\rm Tr}(M)}{2}\bigg)^2-{\rm det}(M)= \frac{1}{4}(M_{1}+M_{4})^2-(M_{1}M_{4}-M_{2}M_{3}) = \frac{1}{4}(M_{1}-M_{4})^2+M_{2}M_{3}.
    \end{equation}
    Applying \eqref{conto.matrix} to $M_{\nu,\eta}^{(0)}(\varepsilon)$ in \eqref{M0.entries}, we get that
    \begin{align}
        \Delta_{\nu,\eta,\varepsilon}& = \bigg(\frac{{\rm Tr}(M_{\nu,\eta}^{(0)}(\varepsilon))}{2}\bigg)^2 - {\rm det}(M_{\nu,\eta}^{(0)}(\varepsilon))  \label{Delta.M2} \\
        & = 
        \begin{cases}
            - \varepsilon^2{\rm b}_1^2 +\frac{\varepsilon^4}{4}\Big((\eta-\nu) + \Big(\frac{1}{\nu}+\frac{1}{\eta}\Big)\la \tW_{\perp\perp}\pa_y^{-1}U,\pa_y^{-1}U\ra\Big)^2  + O(\varepsilon^4) \,, & \textnormal{if } \ \ {\rm b}_1\neq 0 \,, \\
            \frac{\varepsilon^4}{4}\Big((\eta-\nu) + \Big(\frac{1}{\nu}+\frac{1}{\eta}\Big)\la \tW_{\perp\perp}\pa_y^{-1}U,\pa_y^{-1}U\ra\Big)^2 + O(\varepsilon^5) \,, & \textnormal{if } \ \ {\rm b}_1 = 0 \,.
        \end{cases}
    \end{align}
    We distinguish two cases:
    \\[1mm]
    \noindent $\bullet$
    When ${\rm b}_1\neq0$, the discriminant becomes $\Delta_{\nu,\eta,\varepsilon}<0$ in the regime of $\varepsilon\ll 1$ small enough. Therefore,  $\sqrt{\Delta_{\nu,\eta,\varepsilon}}$ does not contribute  to the real part of the eigenvalues $\mu_{\pm}$ in \eqref{mu.eigen.M2}. In this case, by \eqref{traccia.M0}, their real parts are simply given by
    \begin{equation}
      {\rm Re}(\mu_{\pm})  = \frac{{\rm Tr}(M_{\nu,\eta}^{(0)})(\varepsilon)}{2} = \frac{\varepsilon^2}{2}\bigg(-(\nu+\eta) + \Big(\frac{1}{\nu}-\frac{1}{\eta}\Big)\la \tW_{\perp\perp}\pa_y^{-1}U,\pa_y^{-1}U\ra \bigg) +O(\varepsilon^3) \,,
    \end{equation}
    which are both  positive for $\varepsilon \ll 1$ small enough as long as $\eta>\nu>0$ and $\la \tW_{\perp,\perp}\pa_y^{-1}U,\pa_y^{-1}U\ra>\frac{(\nu+\eta)\nu\eta}{\eta-\nu}>0$, as stated in item $(i)$.  Instead, assuming $\nu \geq \eta$, one has
    \begin{equation}
      {\rm Re}(\mu_{\pm})  = \frac{{\rm Tr}(M_{\nu,\eta}^{(0)})(\varepsilon)}{2} = \frac{\varepsilon^2}{2}\bigg(-(\nu+\eta) + \Big(\frac{1}{\nu}-\frac{1}{\eta}\Big)\la \tW_{\perp\perp}\pa_y^{-1}U,\pa_y^{-1}U\ra \bigg) + O(\varepsilon^3) \leq - \e^2 \frac{\nu + \eta}{4} \,,
    \end{equation}
    with $\varepsilon \ll 1$ small enough, which proves item $(iii)$;
\\[1mm]
\noindent $\bullet$
    When ${\rm b}_1=0$, the discriminant $\Delta_{\nu,\eta,\varepsilon}$ in \eqref{Delta.M2} is a perfect square at its leading order $\varepsilon^4$, and, by a Taylor expansion in the regime $\varepsilon\ll 1$ small enough, we deduce
    \begin{align}
        \sqrt{\Delta_{\nu,\eta,\varepsilon}} 
        & = \pm \frac{\varepsilon^2}{2}\bigg( (\eta-\nu) + \Big(\frac{1}{\nu}+\frac{1}{\eta}\Big)\la \tW_{\perp,\perp}\pa_y^{-1}U,\pa_y^{-1}U\ra \bigg) \sqrt{1 + O(\varepsilon)} \\
        & = \pm \frac{\varepsilon^2}{2}\bigg( (\eta-\nu) + \Big(\frac{1}{\nu}+\frac{1}{\eta}\Big)\la \tW_{\perp,\perp}\pa_y^{-1}U,\pa_y^{-1}U\ra \bigg) + O(\varepsilon^3) \,.
    \end{align}
    In this case,  the eigenvalues $\mu_{\pm}$ in \eqref{mu.eigen.M2} are real at their leading order $\varepsilon^2$ and are given by
    \begin{equation}
        \mu_\pm = \frac{\varepsilon^2}{2}\bigg( -(\nu+\eta) + \Big(\frac{1}{\nu}-\frac{1}{\eta}\Big)\la \tW_{\perp\perp}\pa_y^{-1}U,\pa_y^{-1}U\ra \bigg) \pm \frac{\varepsilon^2}{2}\bigg( (\eta-\nu) + \Big(\frac{1}{\nu}+\frac{1}{\eta}\Big)\la \tW_{\perp,\perp}\pa_y^{-1}U,\pa_y^{-1}U\ra \bigg) + O(\varepsilon^3) \,,
    \end{equation}
    that is
     \begin{equation}
        \mu_+ = \frac{\varepsilon^2}{\nu} \big(-\nu^2 + \la \tW_{\perp\perp}\pa_y^{-1}U,\pa_y^{-1}U\ra\big) + O(\varepsilon^3)\,, \quad \mu_- = \frac{\varepsilon^2}{\eta} \big(-\eta^2 - \la \tW_{\perp\perp}\pa_y^{-1}U,\pa_y^{-1}U\ra\big) + O(\varepsilon^3) \,.
    \end{equation}
    Recalling that $\tW_{\perp\perp}$ is definite positive as shown in Remark \ref{rmk.W.pp.pos}, in the regime of $\varepsilon\ll 1$ small enough,  the real part of the eigenvalue $\mu_-$ is always negative, whereas ${\rm Re}(\mu_{+})$ turns out to be positive as long as $\la \tW_{\perp\perp}\pa_y^{-1}U,\pa_y^{-1}U\ra> \nu^2>0$, as stated in item $(ii)$. 
    This concludes the proof.
\end{proof}

We conclude this section by studying the dynamics generated by the matrix $M_{\nu,\eta}^{(0)}(\varepsilon)$, that will be needed in Section \ref{sect.conclusion} to prove Theorem \ref{thm:dinamico}. To this end, we first study the invertible matrix that diagonalize $M_{\nu,\eta}^{(0)}(\varepsilon)$. In the following, $\|\cdot\|$ denotes the standard operator norm of a $2\times 2$ matrix.
\begin{lem}\label{rmk:prop49}
 There exists an invertible matrix $P$ such that 
\begin{equation}\label{M0.diagola}
    P^{- 1} M_{\nu,\eta}^{(0)} P = \begin{pmatrix}
\mu_+ & 0 \\
0 & \mu_-
\end{pmatrix}
\end{equation}
and the following estimates hold: if ${\rm b}_1 \neq0$, then $\| P \|_{{\mathcal B}(\C^2)} \lesssim \e$ and $\| P^{- 1} \|_{{\mathcal B}(\C^2)} \lesssim \e^{- 1}$; if ${\rm b}_1 = 0$, then $\| P \|_{{\mathcal B}(\C^2)}$, $ \| P^{- 1} \|_{{\mathcal B}(\C^2)} \lesssim 1$. 
\end{lem}
\begin{proof}
Consider the case ${\rm b}_1\neq0$.
Using \eqref{M0.entries} and formul\ae\, \eqref{traccia.M0}, \eqref{mu.eigen.M2}, \eqref{Delta.M2} from the proof of Proposition \ref{prop:positive.zeromode},
we remark that 
\begin{equation}\label{foglio1}
|\mu_{+}-\mu_{-}| = 
|\sqrt{\Delta_{\nu,\eta}}|\gtrsim \e |{\rm b}_1| -O(\e^3) \gtrsim \varepsilon |{\rm b}_1| \,,
\qquad 
|\mu_{+} \mu_{-}|\simeq \e^2\,, 
\end{equation}
for $\e$ small enough.
By an explicit computation (recalling again \eqref{M0.entries}), one has that 
\[
M_{\nu,\eta}^{(0)}(\e)=
\begin{pmatrix}
a & b \\ c & d
\end{pmatrix}\,,\qquad |a|,|d|\simeq \e^2\,,\qquad |b|,|c| \simeq \e\,
\]
is diagonalized as in \eqref{M0.diagola} through an  invertible matrix $P$ of the form
\begin{equation}\label{foglio2}
P:=\begin{pmatrix}
b & b
\\
\mu_{+}-a & \mu_{-}-a
\end{pmatrix}\,,\qquad P^{-1}=\frac{1}{{\rm det}(P)}
\begin{pmatrix}
\mu_{-}-a & -b \\
-(\mu_{+}-a) & b
\end{pmatrix}
\end{equation}
In view of \eqref{foglio1}, \eqref{foglio2} one has that
\[
|{\rm det}(P)|=|b(\mu_{+}-\mu_{-})|\gtrsim \e^2\,.
\]
Therefore one can check that 
 $\|P\|_{{\mathcal B}(\C^2)}\lesssim \e$ and $\|P^{-1}\|_{{\mathcal B}(\C^2)}\lesssim \e^{-1}$.
In the case ${\rm b}_1=0$, by \eqref{M0.entries} and recalling the definition of $\mathtt c_1, \mathtt c_2 > 0$ in Proposition \ref{prop:positive.zeromode}-$(ii)$, we write 
$$
M_{\nu, \eta}^{(0)} (\varepsilon) = \e^2 A, \quad A :=  \begin{pmatrix}
a & b \\
c & d
\end{pmatrix} \,,
$$
with $a = \mathtt c_1 + O(\e)$, $d = - \mathtt c_2 + O(\e)$, $b, c = O(\e)$. Then it is enough to diagonalize the matrix $A$. One has that $P, P^{- 1}$ are given by 
$$
P = \begin{pmatrix}
\mu_+ - d & b\\
c & \mu_- - a
\end{pmatrix}, \quad P^{- 1} = \frac{1}{{\rm det}(P)} \begin{pmatrix}
\mu_- - a & - b \\
- c & \mu_+ - d
\end{pmatrix}.
$$
Note that 
$$
\mu_+ - d = \mathtt c_1 + \mathtt c_2 + O(\e) \,, \quad \mu_- - a = - (\mathtt c_1 + \mathtt c_2) + O(\e) \,.
$$
Hence, we have 
$$
|{\rm det}(P)| = (\mathtt c_1 + \mathtt c_2)^2 - O(\e) \gtrsim 1
$$
for $0< \e \ll 1$ small enough. We deduce that $\| P \|_{{\mathcal B}(\C^2)}, \| P^{- 1} \|_{{\mathcal B}(\C^2)} \lesssim 1$. This concludes the proof of the lemma.
\end{proof}
In the next lemma we shall describe precisely the dynamics of the linear ODE 
\begin{equation}\label{cauchy prob 2 per 2}
\begin{cases}
\partial_t y(t) = M_{\nu, \eta}^{(0)} y(t) \\
y(0) = y_0 \in \C^2\,. 
\end{cases}
\end{equation} 
\begin{lem}\label{dinamica sistemino due per due}
Let $\eta, \nu >0$. There exists $\varepsilon_0\equiv \varepsilon_0(S,\nu,\eta,\bb) \in (0, 1) $ small enough such that, for any $\varepsilon \in (0,\varepsilon_0)$,  the following hold:
\begin{enumerate}
        \item[(i)] When ${\rm b}_1\neq0$, if $\eta>\nu$ and $\la \tW_{\perp\perp}\pa_y^{-1}U,\pa_y^{-1}U\ra>\frac{(\nu+\eta)\nu\eta}{\eta-\nu}$, then the solutions of the Cauchy problem \eqref{cauchy prob 2 per 2} satisfies 
        $$
|y(t)| \gtrsim e^{\e^2 (\mathtt c /2) t} |y_0|\,, \quad \forall \, y_0 \in \C^2 \quad  \forall \, t \geq 0\,; 
        $$
        \item[(ii)] When ${\rm b}_1=0$, if $\la \tW_{\perp\perp}\pa_y^{-1}U,\pa_y^{-1}U\ra>\nu^2$, then there are two one-dimensional subspace $E_+, E_-$ of $\C^2$ such that $\C^2 = E_+ \oplus E_-$ and the solutions of the Cachy problem \eqref{cauchy prob 2 per 2} satisfy
        $$
        \begin{aligned}
& |y(t)| \gtrsim e^{\e^2 (\mathtt c_1/2) t} |y_0| \,,  \quad \forall \,y_0 \in E_+ \quad \forall \, t \geq 0 \,, \\
& |y(t)| \lesssim e^{- \e^2 (\mathtt c_2/2) t} |y_0| \,, \quad \forall \, y_0 \in E_- \quad \forall \, t \geq 0\,;
        \end{aligned}
        $$
\item[(iii)] When ${\rm b}_1\neq0$, if $\nu \geq \eta$,  then the solutions of the Cauchy problem \eqref{cauchy prob 2 per 2} satisfies 
        $$
|y(t)| \lesssim e^{- \e^2 (\mathtt c /2) t} |y_0|\,, \quad \forall \, y_0 \in \C^2 \quad  \forall \, t \geq 0\,. 
        $$
    \end{enumerate}
\end{lem}
\begin{proof}
We apply Proposition \ref{prop:positive.zeromode} and Lemma \ref{rmk:prop49}. Under the change of coordinates $y(t) = P z(t)$, the Cauchy problem \eqref{cauchy prob 2 per 2} transforms into 
\begin{equation}\label{sistema diagonale 2 per 2 con P}
\begin{cases}
\partial_t z_+(t) = \mu_+ z_+(t)\\
\partial_t z_-(t) = \mu_- z_-(t) \\
z_+(0) = z_{+, 0} \\
z_-(0) = z_{-, 0}
\end{cases}
\end{equation}
where $z = (z_+, z_-) \in \C^2$, $z_0 = (z_{0, +}, z_{0, -}) \in \C^2$, $z_0 := P^{- 1} y_0$. The solution of  \eqref{sistema diagonale 2 per 2 con P} is given by 
\begin{equation}\label{sol sistema diagonale 2 per 2 con P}
z_\pm(t) = e^{\mu_{\pm t}}z_{\pm, 0} \quad t \geq 0.
\end{equation}
Based on Proposition \ref{prop:positive.zeromode}, we analyze the different cases:
\\[1mm]
\noindent $\blacktriangleright$
{\bf The case ${\rm b}_1 \neq 0$, $\eta>\nu$ and $\la \tW_{\perp\perp}\pa_y^{-1}U,\pa_y^{-1}U\ra>\frac{(\nu+\eta)\nu\eta}{\eta-\nu}$.} 
\\[1mm]
\noindent
In this case $\| P \|_{{\mathcal B}(\C^2)} \lesssim \e$, $\| P^{- 1} \|_{{\mathcal B}(\C^2)} \lesssim \e^{- 1}$ and ${\rm Re}(\mu_{\pm}) = \mathtt c \e^2 + O(\e^3) \geq (\mathtt c/2) \e^2$ for $0 < \e \ll 1$ small enough. Hence, by \eqref{sol sistema diagonale 2 per 2 con P},
$$
|z(t)| \geq e^{(\mathtt c/2)\e^2 t} |z_0|\, \quad \forall \, t \geq 0
$$
Moreover, by the estimates on $P^{\pm 1}$, one has that 
\begin{align}
& |z(t)| = |P^{- 1} y(t)| \lesssim \e^{- 1} |y(t)|\,, \qquad  |y_0| = |P z_0| \lesssim \e |z_0|\,. 
\end{align}
Therefore, for any $t \geq 0$, 
$$
\begin{aligned}
e^{(\mathtt c/2)\e^2 t} |y_0| & \lesssim \e e^{(\mathtt c/2)\e^2 t} |z_0| \lesssim \e |z(t)| \\
& \lesssim \e \e^{- 1} |y(t)| \lesssim |y(t)|
\end{aligned}
$$
which proves the item $(i)$.
\\[1mm]
\noindent $\blacktriangleright$
{\bf The case ${\rm b}_1 \neq 0$, $\nu \geq \eta$.} 
\\[1mm]
\noindent
Also in this case $\| P \|_{{\mathcal B}(\C^2)} \lesssim \e$, $\| P^{- 1} \|_{{\mathcal B}(\C^2)} \lesssim \e^{- 1}$ and, for some $\mathtt c > 0$, ${\rm Re}(\mu_{\pm}) = - \mathtt c \e^2 + O(\e^3) \leq - (\mathtt c/2) \e^2$ for $0 < \e \ll 1$ small enough. Hence, by \eqref{sol sistema diagonale 2 per 2 con P},
$$
|z(t)| \leq e^{- (\mathtt c/2)\e^2 t} |z_0|\, \quad \forall t \geq 0
$$
Moreover, by the estimates on $P^{\pm 1}$, one has that 
\begin{align}
& |y(t) | = |P z(t)| \lesssim \e |z(t)|\,, \quad  |z_0| = |P^{- 1} y_0| \lesssim \e^{- 1} |y_0|\,. 
\end{align}
Therefore, for any $t \geq 0$, 
$$
\begin{aligned}
|y(t)| & \lesssim \e |z(t)| \lesssim \e e^{- (\mathtt c/2)\e^2 t} |z_0| \lesssim \e e^{- (\mathtt c/2)\e^2 t} \e^{- 1} |y_0| \lesssim e^{- (\mathtt c/2)\e^2 t} |y_0|
\end{aligned}
$$
which proves the item $(iii)$.
\\[1mm]
\noindent $\blacktriangleright$
{\bf The case ${\rm b}_1 = 0$, $\la \tW_{\perp\perp}\pa_y^{-1}U,\pa_y^{-1}U\ra>\nu^2$.}
\\[1mm ]
\noindent
In this case $\| P^{\pm 1} \|_{{\mathcal B}(\C^2)} \lesssim 1$ and there are $\mathtt c_1, \mathtt c_2 > 0$ such that ${\rm Re}(\mu_{+}) =  \mathtt c_1 \e^2 + O(\e^3) \geq (\mathtt c_1/2 ) \e^2$ and ${\rm Re}(\mu_{-}) =  - \mathtt c_2 \e^2 + O(\e^3) \leq - (\mathtt c_2/2)  \e^2$ for $0 < \e \ll 1$ small enough. Hence by \eqref{sol sistema diagonale 2 per 2 con P}
$$
|z_+(t)| \geq e^{ (\mathtt c_1/2)\e^2 t} |z_{0, +}|, \quad |z_-(t)| \leq e^{- (\mathtt c_2/2)\e^2 t} |z_{0, -}|\,, \quad \forall \, t \geq 0\,.
$$
Then define 
$$
\begin{aligned}
E_+ & := \Big\{ P \big( \begin{smallmatrix}
    z_+ \\ 0
\end{smallmatrix} \big) : z_+ \in \C \Big\}\,, \quad E_-  := \Big\{ P\big( \begin{smallmatrix}
    0 \\ z_- 
\end{smallmatrix} \big) : z_- \in \C \Big\}\,.
\end{aligned}
$$
Since $P$ is an isomorphism, it holds $\C^2 = E_+ \oplus E_-$.
If $y_0 \in E_+$, then $y_0 = P \big( \begin{smallmatrix}
 z_{0, +}  \\ 0
\end{smallmatrix} \big)  $ for some $z_{0,+}\in \C$, and we have
$$
y(t) = P\big( \begin{smallmatrix}
  e^{\mu_+ t} z_{0, +}  \\ 0
\end{smallmatrix} \big) = e^{\mu_+ t} y_0 \,,
$$
implying that 
$$
|y(t)| \geq e^{(\mathtt c_1/2)\e^2 t} |y_0| ,, \quad \forall \, t \geq 0.
$$
On the other hand, if 
$y_0 \in E_-$, then $y_0 = \big( \begin{smallmatrix}
  0 \\ z_{0, -} 
\end{smallmatrix} \big) $  for some $z_{0,-}\in\C$, and we have
$$
y(t) = P \big( \begin{smallmatrix}
  0 \\ e^{\mu_- t} z_{0, -} 
\end{smallmatrix} \big)  = e^{\mu_- t} y_0 \,, 
$$
implying that 
$$
|y(t)| \leq e^{- (\mathtt c_2/2)\e^2 t} |y_0| \,,  \quad \forall \, t \geq 0.
$$
The item $(ii)$ has been proved. This also concludes the proof of the lemma.
\end{proof}

\section{Stability analysis of the non-zero mode}\label{sect.nonzero}

In Section \ref{sect.zeromode}, we conjugated the matrix representation $\tL_{\nu,\eta,\varepsilon}$ in \eqref{tL.vare} of the operator $\cL_{\nu,\eta,\varepsilon}$ into the reduced matrix operator  $\tL_{\nu,\eta,\varepsilon}^{(0)}$ in \eqref{eq:conjugation}, where the action over the zero mode is decoupled from the one over the non-zero mode. This is the content of Proposition \ref{prop:fixed_point}. For convenience, in the following lemma, we further study the normal form  $\tL_{\nu,\eta,\varepsilon}^{(0)}$ that we obtained.
\begin{lem}\label{lem.cQ0.resist}
Let $U\in C^{S+2}$, $S\geq 0$.
    The matrix $\tL_{\nu,\eta,\varepsilon}^{(0)}$ in \eqref{eq:conjugation} is given by
    \begin{equation}\label{matricina}
        \tL_{\nu,\eta,\varepsilon}^{(0)} =  \begin{bmatrix}
            M_{\nu,\eta}^{(0)}(\varepsilon) & \boldsymbol{0}^\top \\ \boldsymbol{0} & \cL_{\nu,\eta,\varepsilon}^{(0)}
        \end{bmatrix}  \,,
    \end{equation}
    where $M_{\nu,\eta}^{(0)}(\varepsilon)$ is as in \eqref{eq:1st_entrance_resist} and the operator
    \begin{equation}
      \cL_{\varepsilon}^{(0)} :=  \cL_{\nu,\eta,\varepsilon}^{(0)} := \cN_{\nu,\eta,\varepsilon}^{(0)} ({\rm Id}_\perp - T_3T_2)^{-1} = \big( \tL_{\perp\perp} - \tL_{\perp 0} T_2 + T_3 \tL_{0\perp} - T_3 \tL_{00} T_2 \big) ({\rm Id}_\perp - T_3T_2)^{-1} 
    \end{equation}
    can be rewritten as
    	\begin{equation}\label{cL.eps.0}
		\cL_{\varepsilon}^{(0)} = \tD_{\perp\perp} + \cQ_{\varepsilon}^{(0)} \,,
	\end{equation}
    with $\tD_{\perp\perp}$ as in \eqref{D.perperp} and $\cQ_{\varepsilon}^{(0)} \in \cB(\bH_0^s)$ satisfying, for $0\leq s\leq S$,
    \begin{equation}\label{cQ.eps.est}
        \| \cQ_{\varepsilon}^{(0)} \|_{\cB(\bH_0^s)} \lesssim \varepsilon\,.
    \end{equation}
\end{lem}
\begin{proof}
    From the equation solved by $T_3$ in \eqref{eq:T3}, we have
    \begin{equation}
        \tL_{\perp0} = \tL_{\perp\perp}T_3 + T_3\tL_{0\perp}T_3 - T_3\tL_{00}.
    \end{equation}
   We insert it into \eqref{gerry1} to compute
        \begin{align}
            \cN_{\nu,\eta,\varepsilon}^{(0)} &= -T_3\tL_{00}T_2 + T_3\tL_{0\perp} - \tL_{\perp0}T_2 + \tL_{\perp\perp} \\
            &= -T_3\tL_{00}T_2 + T_3\tL_{0\perp} - (\tL_{\perp\perp}T_3 + T_3\tL_{0\perp}T_3 - T_3\tL_{00})T_2 + \tL_{\perp\perp} \\
            &= ( \tL_{\perp\perp} + T_3 \tL_{0\perp})({\rm Id}_\perp-T_3T_2) \,,
        \end{align}
    from which we obtain, recalling \eqref{Lpp},
    \begin{align}
        \cL_{\varepsilon}^{(0)} & =  ( \tL_{\perp\perp} + T_3 \tL_{0\perp})({\rm Id}_\perp-T_3T_2)  ({\rm Id}_\perp-T_3T_2) ^{-1} \\
        & =  \tL_{\perp\perp} + T_3 \tL_{0\perp}   = \tD_{\perp\perp} + \cQ_{\varepsilon}^{(0)} \,, \quad \cQ_{\varepsilon}^{(0)} := \tR_{\perp\perp} + T_3 \tL_{0\perp}  \,,\label{cQ.eps}
    \end{align}
    which proves \eqref{cL.eps.0}. Finally, we estimate $\cQ_{\varepsilon}^{(0)}$ in \eqref{cQ.eps}, recalling \eqref{esti.easy}, \eqref{Lpp.eps}, \eqref{eq:estim_order_T3_resist} and Lemma \ref{espandi.e.non.mollo}, by
    \begin{equation}
        \| \cQ_{\varepsilon}^{(0)} \|_{\cB(\bH_0^s)} \leq \| \tR_{\perp\perp} \|_{\cB(\bH_0^s)} + \| T_3 \|_{\cB(\C^2,\bH_0^s)} \| \tL_{0\perp}\|_{\cB(\bH_0^s,\C^2)} \lesssim \varepsilon + \varepsilon^{-1} \varepsilon^3 \lesssim \varepsilon \,,
    \end{equation}
    which proves the estimate \eqref{cQ.eps.est}. 
\end{proof}

    

The goal of this section is to study the spectrum of ${\mathcal L}_\e^{(0)}$, and in particular to show that it exhibits stability. 
To this end, it is useful to diagonalize the leading operator $\tD_{\perp\perp}$ using Lemma \ref{lem.eigen.D.perp.resist-iii}: indeed, we proved that there exists a bounded and invertible map ${\mathcal C} \in  {\mathcal B}(\bH_0^s)$ with inverse ${\mathcal C}^{- 1} \in {\mathcal B}(\bH_0^s)$ such that 
$$
{\mathcal C}^{- 1} {\mathtt D}_{\perp \perp} {\mathcal C} = \Lambda \,, 
$$
where $\Lambda $ is defined in \eqref{def Lambda} of Lemma \ref{lem.eigen.D.perp.resist-iii}-$(iii)$. Therefore, one has that 
\begin{equation}\label{def mathcal L epsilon 1}
{\mathcal P}_\e : = {\mathcal C}^{- 1} {\mathcal L}_\e^{(0)} {\mathcal C} = \Lambda + {\mathcal E}_\e , \quad {\mathcal E}_\e := {\mathcal C}^{- 1} {\mathcal Q}_\e^{(0)} {\mathcal C}\,.
\end{equation}
Clearly by Lemma \ref{lem.cQ0.resist} and since $\| {\mathcal C}^{\pm 1} \|_{{\mathcal B}(\bH_0^s)} \lesssim 1$, one has that 
\begin{equation}\label{stima mathcal Q epsilon (1)}
\| {\mathcal E}_\e \|_{{\mathcal B}(\bH_0^s)} \lesssim \e, \quad \forall \, 0\leq s \leq S \,,
\end{equation}
provided that $U\in C^{S+2}(\T)$.
In the following lemma, we show that spectrum of $\cP_{\varepsilon}$ in \eqref{def mathcal L epsilon 1} is pure point and stable, namely all its eigenvalues have strictly negative real part.

\begin{lem}\label{spettro mathcal L epsilon (1)}
Let $U\in C^{S+2}$, $S\geq 0$. There exists $\varepsilon_0\equiv \varepsilon_0(S,\nu,\eta,{\rm b}_2) \ll 1$ small enough such that,  for $\varepsilon \in (0,\varepsilon_0)$ and for any $s \in [0,S]$, the following hold:
\\[1mm]
\noindent $(i)$
The spectrum $\sigma({\mathcal P}_\e) = (\lambda_n)_{n \geq 0}$ of the operator ${\mathcal P}_\e : \bH_0^{s + 2}  \to \bH_0^{s}$ is discrete and $\lambda_n \to + \infty$ as $n \to + \infty$;
\\[1mm]
\noindent $(ii)$
For any $n \geq 0$, one has that 
${\rm Re}(\lambda_n) \leq - \frac{\sigma}{2}$ where $\sigma> 0$ is the constant appearing in Lemma \ref{lem.eigen.D.perp.resist-iii}-$(i)$. 
\end{lem}

\begin{proof}
We start with the proof of item $(i)$.
We shall prove first that ${\mathcal P}_\e :\bH_0^{s + 2} \to \bH_0^s$ is invertible. We write 
\begin{equation}\label{cP.e1}
    {\mathcal P}_\e = \Lambda \big( {\rm Id}_\bot + \Lambda^{- 1} {\mathcal E}_\e \big)\,.
\end{equation}
By \eqref{stima mathcal Q epsilon (1)} and Lemma \ref{lem.eigen.D.perp.resist-iii}-$(iii)$, we have that
$$
\| \Lambda^{- 1} {\mathcal E}_\e \|_{{\mathcal B}(\bH_0^s)} \leq \| \Lambda^{- 1} \|_{{\mathcal B}(\bH_0^s, \bH_0^{s + 2})} \|{\mathcal E}_\e \|_{\cB(\bH_0^s)} \lesssim \e \sigma^{- 1} \,.
$$
Hence, for $C(S)\e \sigma^{- 1} \ll 1$ small enough, by a Neumann series argument we obtain that $ {\rm Id}_0 + \Lambda^{- 1} {\mathcal E}_\e  : \bH_0^s \to \bH_0^s$ is invertible, with  
$$
\big\| \big( {\rm Id}_\bot + \Lambda^{- 1} {\mathcal E}_\e \big)^{- 1}\big\|_{{\mathcal B}(\bH_0^s)} \leq 2\,. 
$$
Therefore, we deduce that $\cP_\varepsilon$ in \eqref{cP.e1} is invertible as well, with
$$
{\mathcal P}_\e^{- 1} = \big( {\rm Id}_\bot + \Lambda^{- 1} {\mathcal E}_\e \big)^{- 1} \Lambda^{- 1}:\bH_0^{s} \to \bH_0^{s+2} \,,
$$
satisfying
$$
\| {\mathcal P}_\e^{- 1} \|_{{\mathcal B}(\bH_0^s, \bH_0^{s + 2})} \leq \big\| \big( {\rm Id}_\bot + \Lambda^{- 1} {\mathcal E}_\e \big)^{- 1} \big\|_{{\mathcal B}(\bH_0^{s + 2})} \| \Lambda^{- 1} \|_{{\mathcal B}(\bH_0^s, \bH_0^{s + 2})} \leq 2 \sigma^{- 1} \,.
$$
Since $\bH_0^{s + 2}$ is compactly embedded in $\bH^s_0$, we deduce that ${\mathcal P}_\e^{- 1} : \bH^s_0 \to \bH^s_0$ is compact, implying that the spectrum of ${\mathcal P}_\e$ consists of a sequence of eigenvalues with finite multiplicity $(\lambda_n)_{n \geq 0}$, with $\lambda_n \to + \infty$ as $n \to + \infty$.  

We now prove item $(ii)$.
For any $n\geq 0$, we take an eigenvalue $\lambda_n \in \sigma(\cP_\varepsilon)$ with a corresponding eigenfunction $\varphi_n = (\varphi_{n, +}, \varphi_{n, -}) \in \bH^{s + 2}_0$ that is normalized in $\bL^2$, that is $\| \varphi_n \|_{\bL^2} = 1$. One has that 
\begin{align}
& \langle \Lambda \varphi_n\,,\, \varphi_n \rangle_{\bL^2} + \langle {\mathcal E}_\e \varphi_n\,,\, \varphi_n \rangle_{\bL^2} = \lambda_n \quad \text{and} \quad  \overline{\langle \Lambda \varphi_n\,,\, \varphi_n \rangle_{\bL^2}} + \overline{\langle {\mathcal E}_\e \varphi_n\,,\, \varphi_n \rangle_{\bL^2}} = \overline \lambda_n
\end{align}
implying that 
\begin{equation}\label{Re.lambda}
    {\rm Re}(\lambda_n) = \frac{\lambda_n + \overline \lambda_n}{2} = \dfrac{\langle \Lambda \varphi_n\,,\, \varphi_n \rangle_{\bL^2} + \overline{\langle \Lambda \varphi_n\,,\, \varphi_n \rangle_{\bL^2}} + \langle {\mathcal E}_\e \varphi_n\,,\, \varphi_n \rangle_{\bL^2} + \overline{\langle {\mathcal E}_\e \varphi_n\,,\, \varphi_n \rangle_{\bL^2}}}{2}\,. 
\end{equation}
Recalling Lemma \ref{lem.eigen.D.perp.resist-iii}-$(iii)$, we have that
$$
\Lambda \varphi_n = \sum_{j \in \Z \setminus \{ 0 \}} \Lambda(j)[\widehat \varphi_n(j)] e^{\im j y} =  \begin{pmatrix}\sum_{j \in \Z \setminus \{ 0 \}} \lambda_+(j)\widehat \varphi_{n, +}(j) e^{\im j y}   \\ \sum_{j \in \Z \setminus \{ 0 \}} \lambda_-(j)\widehat \varphi_{n, -}(j) e^{\im j y} 
\end{pmatrix}
$$
which implies that
$$
\begin{aligned}
& \langle \Lambda \varphi_n\,,\, \varphi_n \rangle_{\bL^2} = \sum_{j \in \Z \setminus \{ 0 \}} \lambda_+(j)|\widehat \varphi_{n, +}(j)|^2  + \sum_{j \in \Z \setminus \{ 0 \}} \lambda_-(j)|\widehat \varphi_{n, -}(j)|^2   \,, \\
& \overline{\langle \Lambda \varphi_n\,,\, \varphi_n \rangle_{\bL^2}} = \sum_{j \in \Z \setminus \{ 0 \}} \overline{\lambda_+(j)}|\widehat \varphi_{n, +}(j)|^2  + \sum_{j \in \Z \setminus \{ 0 \}} \overline{\lambda_-(j)}|\widehat \varphi_{n, -}(j)|^2 \,.
\end{aligned}
$$
Consequently, using by Lemma \ref{lem.eigen.D.perp.resist-iii}-$(i)$ and the normalization  $\| \varphi_n \|_{\bL^2} = 1$, we get 
\begin{align}
\frac{\langle \Lambda \varphi_n\,,\, \varphi_n \rangle_{\bL^2} + \overline{\langle \Lambda \varphi_n\,,\, \varphi_n \rangle_{\bL^2}} }{2} & = \sum_{j \in \Z \setminus \{ 0 \}} {\rm Re}(\lambda_+(j))|\widehat \varphi_{n, +}(j)|^2  + \sum_{j \in \Z \setminus \{ 0 \}} {\rm Re}(\lambda_-(j))|\widehat \varphi_{n, -}(j)|^2 \\
& \leq - \sigma \sum_{j \in \Z^2 \setminus \{ 0 \}} |j|^2 \big( |\widehat \varphi_{n, +}(j)|^2 + |\widehat \varphi_{n, -}(j)|^2 \big)  \label{stima Re lambdan primo pezzo} \\
& \leq - \sigma \| \varphi_n \|_{\bH_0^1}^2 \leq - \sigma \| \varphi_n \|_{\bL^2}^2 \leq - \sigma\,. 
\end{align}
Moreover, by the Cauchy Schwartz inequality and using that $\| {\mathcal E}_\e \|_{{\mathcal B}({\mathcal L}^2_0)} \lesssim \e$ as in \eqref{stima mathcal Q epsilon (1)}, one gets that
\begin{equation}\label{stima Re lambda n sec pezzo}
\begin{aligned}
\frac{\langle {\mathcal E}_\e \varphi_n\,,\, \varphi_n \rangle_{\bL^2} + \overline{\langle {\mathcal E}_\e \varphi_n\,,\, \varphi_n \rangle_{\bL^2}}}{2} &  \leq |\langle {\mathcal E}_\e \varphi_n\,,\, \varphi_n \rangle_{\bL^2} | \\
& \leq \| {\mathcal E}_\e \varphi_n \|_{\bL^2} \| \varphi_n \|_{\bL^2} \leq C \e \| \varphi_n \|_{\bL^2}^2 \leq C \varepsilon
\end{aligned}
\end{equation}
for some constant $C > 0$. 
Finally, collecting \eqref{stima Re lambdan primo pezzo}, \eqref{stima Re lambda n sec pezzo} into \eqref{Re.lambda}, we obtain that
$$
{\rm Re}(\lambda_n) \leq - \sigma + C \e \leq - \frac{\sigma}{2}
$$
by taking $0 < \e \ll 1$ small enough. The proof of the lemma is then concluded. 
\end{proof}

We now consider the Cauchy problem 
\begin{equation}\label{cauchy stability}
\begin{cases}
\partial_t u = {\mathcal L}_\e^{(0)} u \\
u(0, y) = \varphi(y) \in \bH_0^s \,.
\end{cases}
\end{equation}
In the following proposition, we show that we have not only the spectral stability shown in Lemma \ref{spettro mathcal L epsilon (1)}-$(ii)$, but also  the stability of the evolution under the linear operator $\cL_{\varepsilon}^{(0)}$ in \eqref{cL.eps.0} of Lemma \ref{lem.cQ0.resist}.
\begin{lem}\label{stability dynamics}
Let $U\in C^{S+2}$, $S\geq 0$. There exists $\e_0 \equiv \e_0(S, \nu, \eta, {\rm b}_2) \in (0, 1)$ such that, for any $\e \in (0, \e_0)$, $s \in [0,S]$ and for any $\varphi \in \bH^s_0$, there exists a unique solution $u \in {\mathcal C}^0([0, + \infty), \bH^s_0)$ of \eqref{cauchy stability} satisfying 
$$
\| u(t) \|_s \lesssim e^{- \frac{\sigma}{2} t} \| \varphi \|_{s}\,, \quad \forall \, t \geq 0\,. 
$$
\end{lem}
\begin{proof}
Under the change of coordinates $u(t) = {\mathcal C} w(t)$ (see Lemma \ref{lem.eigen.D.perp.resist-iii}-$(ii)$), one gets that \eqref{cauchy stability} transforms into 
\begin{equation}\label{cauchy stability2}
\begin{cases}
\partial_t w = \Lambda w + {\mathcal E}_\e w \\
u(0, y) = w_0 (y) := \cC^{-1} \vf(y) \,.
\end{cases}
\end{equation}
Note that $\| u (t) \|_s \simeq \| w(t) \|_s$ by the continuity properties of ${\mathcal C}, {\mathcal C}^{- 1}$. Let $e^{\Lambda t}$, $t \geq 0$, be the propagator associated with the linear PDE $\partial_t w = \Lambda w$. For $w = (w_+, w_-) \in \bH^s_0$, this takes the form 
$$
e^{\Lambda t}w = \begin{pmatrix}
\sum_{j \in \Z \setminus \{ 0 \}} e^{\lambda_+(j) t} \widehat w_+(j) e^{\im j y } \\
\sum_{j \in \Z \setminus \{ 0 \}} e^{\lambda_-(j) t} \widehat w_-(j) e^{\im j y } \\
\end{pmatrix}\,.
$$
Since  ${\rm Re}(\lambda_\pm(j)) \leq - \sigma j^2 \leq - \sigma$ for any $j \in \Z \setminus \{ 0 \}$, one has that $|e^{\lambda_\pm(j) t}| \leq e^{- \sigma t}$ for any $t \geq 0$. Therefore 
\begin{equation}\label{stima propagatore lineare e lambda t}
\begin{aligned}
\| e^{\Lambda t} w \|_{s} & \leq \bigg( \sum_{j \in \Z \setminus \{ 0 \}} \langle j \rangle^{2 s} \big(  e^{2 \lambda_+(j) t} |\widehat w_+(j)|^2 + e^{2 \lambda_-(j) t} |\widehat w_-(j)|^2 \big)  \bigg)^{\frac12} \\
& \leq e^{- \sigma t} \Big( \sum_{j \in \Z \setminus \{ 0 \}} \langle j \rangle^{2 s}\big( |\widehat w_+(j)|^2 +  |\widehat w_-(j)|^2 \big)  \Big)^{\frac12} \leq e^{- \sigma t} \| w \|_s\,. 
\end{aligned}
\end{equation}
The Cauchy problem \eqref{cauchy stability2} is equivalent to the fixed point equation
\begin{equation}\label{fixed point modi stabili}
w = \Phi(w), \quad \Phi(w(t)) := e^{\Lambda t} w_0 + \int_0^t e^{\Lambda(t - \tau)}{\mathcal E}_\e[w(\tau)] \wrt \tau, \quad t \geq 0
\end{equation}
We define 
$$
{\mathcal A}_s := \Big\{ w \in {\mathcal C}^0([0, + \infty), \bH^s_0) : \| w \|_{{\mathcal A}^s} := \sup_{t \geq 0} e^{\frac{\sigma}{2} t} \| w(t) \|_s < + \infty \Big\}\,.
$$
Clearly ${\mathcal A}_s$ is a Banach space. We also define the closed ball
$$
{\mathcal B}_s(R) := \big\{ w \in {\mathcal A}_s : \| u \|_{{\mathcal A}_s} \leq R \big\}\,. 
$$
We now claim that, for any $0<\varepsilon \ll 1$ small enough, the map
\begin{equation}\label{claim.contr}
    \Phi : {\mathcal B}_s(2 \| w_0\|_s) \to {\mathcal B}_s(2 \| w_0\|_s) \quad \text{is a contraction.}
\end{equation}
We split the proof of this claim into steps.
\\[1mm]
\noindent {\bf Step 1.} First, we show that $\| e^{t \Lambda} w_0 \|_{{\mathcal A}^s} \leq \| w_0 \|_s$. Indeed, by \eqref{stima propagatore lineare e lambda t}, one has that
$$
\| e^{\Lambda t} w_0 \|_s \leq e^{- \sigma t} \| w_0 \|_s \,, \quad  \forall\, t \geq 0 \,,
$$
implying that 
\begin{equation}\label{stima lin prop As}
\| e^{t \Lambda} w_0 \|_{{\mathcal A}^s} = \sup_{t \geq 0} e^{\frac{\sigma }{2} t} \| e^{\Lambda t} w_0 \|_s \leq \| w_0 \|_s \,.
\end{equation}
\noindent
{\bf Step 2.} We define the linear operator ${\mathcal F}$ as 
\begin{equation}\label{mappa F}
    {\mathcal F}[w] := \int_0^t e^{\Lambda (t - \tau)} {\mathcal E}_\e [w(\tau)] \wrt \tau, \quad w \in {\mathcal C}^0([0, + \infty), \bH^s_0)\,. 
\end{equation}
We now want to show that 
\begin{equation}\label{prop cal F fixed point}
{\mathcal F} \in {\mathcal B}({\mathcal A}_s) \quad \text{and}  \quad \| {\mathcal F} \|_{\cB({\mathcal A}_s)} \leq C(S) \e\,, 
\end{equation}
for some $C(S)>0$ large and 
for $0 < \e \ll 1$ small enough. Let $w \in {\mathcal A}_s$, then, for any $t \geq 0$, one has 
\begin{align}
\Big\| \int_0^t e^{\Lambda(t - \tau)}{\mathcal E}_\e [w(\tau)] \wrt \tau \Big\|_s & \leq \int_0^t \Big\| e^{\Lambda(t - \tau)}{\mathcal E}_\e [w(\tau)] \Big\|_s \wrt \tau \\
& \stackrel{\eqref{stima propagatore lineare e lambda t}}{\leq} \int_0^t e^{- \sigma(t - \tau)} \|{\mathcal E}_\e [w(\tau)] \|_s \wrt \tau \\
& \stackrel{\eqref{stima mathcal Q epsilon (1)}}{\lesssim}  \e \int_0^t e^{- \sigma(t - \tau)} e^{- \frac{\sigma}{2} \tau}  (e^{\frac{\sigma}{2} \tau}\| w(\tau) \|_s) \wrt \tau \\
& \lesssim \e e^{- \sigma t} \int_0^t e^{\frac{\sigma}{2} \tau} \wrt \tau \| w \|_{{\mathcal A}_s} \\
& \lesssim \e e^{- \sigma t} \big( e^{\frac{\sigma}{2} t} - 1 \big) \| w \|_{{\mathcal A}^s}  \lesssim \e e^{- \frac{\sigma}{2} t} \| w \|_{{\mathcal A}_s} \,.
\end{align}
We conclude that
$$
\| {\mathcal F}[w] \|_{{\mathcal A}_s} = \sup_{t \geq 0} e^{\frac{\sigma}{2} t} \Big\|\int_0^t e^{\Lambda(t - \tau)}{\mathcal E}_\e [w(\tau)] \wrt \tau \Big\|_s \leq C(S) \e \| w \|_{{\mathcal A}_s} \,.
$$
{\bf Step 3. Conclusion.} Recalling \eqref{fixed point modi stabili} and \eqref{mappa F}, we write
$$
\Phi(w) := e^{\Lambda t} w_0 + {\mathcal F}(w)\,.
$$
If $w \in {\mathcal B}_s(2 \| w_0 \|_s)$, then, by \eqref{stima lin prop As} and \eqref{prop cal F fixed point},
\begin{equation}\label{contra.claim1}
    \| \Phi(w) \|_s \leq \| w_0 \|_s + C(S) \e \| w \|_{{\mathcal A}_s} \leq \| w_0 \|_s + 2 C(S) \e  \| w_0 \|_s \leq 2 \| w_0 \|_s
\end{equation}
for $0 < \e \ll 1$ small enough.
Moreover, if $w_1, w_2 \in {\mathcal B}_s(2 \| w_0 \|_s)$, by \eqref{prop cal F fixed point}, one has that 
\begin{equation}\label{contra.claim2}
    \| \Phi(w_1) - \Phi(w_2) \|_{{\mathcal A}_s} = \| {\mathcal F}(w_1) - {\mathcal F}(w_2) \|_{{\mathcal A}_s} \leq C(S) \e \| w_1 - w_2 \|_{{\mathcal A}_s} \leq \tfrac12 \| w_1 - w_2 \|_{{\mathcal A}_s}
\end{equation}
for $0 < \e \ll 1$ small enough. The estimates \eqref{contra.claim1} and \eqref{contra.claim2} prove that $\Phi$ is a contraction, as claimed in \eqref{claim.contr}. Therefore, by the contraction mapping theorem, we deduce that there exists a unique $w \in {\mathcal B}_s(2 \| w_0 \|_s)$ such that $w = \Phi(w)$. Clearly $\| w(t) \|_s \leq 2 e^{- \frac{\sigma}{2} t} \| w_0 \|_s $.  Finally, in terms of $u(t)= \cC w(t)$ solution of \eqref{cauchy stability}, we get that
$$
\| u(t) \|_s = \| {\mathcal C} w(t) \|_s \lesssim \| w(t) \|_s \lesssim e^{- \frac{\sigma}{2} t} \| w_0 \|_s \lesssim e^{- \frac{\sigma}{2} t} \| {\mathcal C}^{- 1} \varphi \|_s \lesssim e^{- \frac{\sigma}{2} t} \|  \varphi \|_s
$$
using that ${\mathcal C} : \bH^s_0 \to \bH^s_0$ is a linear isomorphism. The proof of the lemma is then concluded. 
\end{proof}

\section{Proof of the main results}\label{sect.conclusion}

We are finally in position to prove Theorem \ref{theo.princ} and Theorem \ref{thm:dinamico}. 
By Proposition \ref{prop:fixed_point}, recalling \eqref{matricina} in Lemma \ref{lem.cQ0.resist}, the Cauchy problem 
\begin{equation}\label{change.vTu}
        \begin{cases}
            \pa_{t} u(t) = \cL_{\nu,\eta,\varepsilon} u(t) \,, \\
            u(0) = \vf \in H^s \times H^s \,.
        \end{cases}
    \end{equation}
transforms, under the change of coordinates
\begin{equation}
   \begin{pmatrix}
        v_0(t) \\ v_\perp(t)
    \end{pmatrix} := \tT \begin{pmatrix}
        u_0(t) \\ u_\perp(t)
    \end{pmatrix}, \quad  \textnormal{with } \ \ \  u(t) = u_0(t) + u_\perp(t) \,, \quad u_0=\Pi_0 u \,, \quad u_\perp = \Pi_0^\perp u \,, 
\end{equation}
with $\tT$ as in \eqref{eq:ansatz_conjugation}, into the following two Cauchy problems
\begin{equation}\label{two.Cauchy}
    \begin{cases}
            \pa_{t} v_0(t) = M_{\nu,\eta}^{(0)} v_0(t) \,, \\
            v_0(0) = \phi_0  \in \C^2 \,, 
        \end{cases} \quad \quad \begin{cases}
            \pa_{t} v_\perp(t) = \cL_{\nu,\eta,\varepsilon}^{(0)} v_\perp(t) \,, \\
            v_{\perp}(0) = \phi_\perp \in \bH_0^s \,,
        \end{cases}
\end{equation}
where we also have the splitting of the initial datum
\begin{align}\label{initial.data.split}
    & \begin{pmatrix}
        \phi_0 \\ \phi_\perp
    \end{pmatrix} := \tT \begin{pmatrix}
        \vf_0 \\ \vf_\perp
    \end{pmatrix} \quad \Leftrightarrow \quad   \begin{cases}
        \phi_0 = \vf_0 + T_2 \vf_\perp \,, \\
        \phi_\perp = T_3 \vf_0 + \vf_\perp \,,
    \end{cases} \\
    & \quad \textnormal{with } \ \ \ \vf = \vf_0 + \vf_\perp  \,, \quad \vf_0=\Pi_0 \vf \,, \quad \vf_\perp = \Pi_0^\perp \vf \,.
\end{align}

\begin{proof}[Proof of Theorem \ref{theo.princ}]
    By Lemma  \ref{spettro mathcal L epsilon (1)} and inverting the conjugation in \eqref{def mathcal L epsilon 1},
one has that the spectrum of the operator $\cL_{\nu,\eta,\varepsilon}^{(0)} $ is pure point and each eigenvalue has negative real part.
By Proposition \ref{prop:positive.zeromode} and the invertibility of the map $\tT$ in Proposition \ref{prop:fixed_point}, we deduce the claims in Theorem \ref{theo.princ}.
\end{proof}

\begin{proof}[Proof of Theorem \ref{thm:dinamico}]
    We prove the claims in the three different cases:
    \\[1mm]
    \noindent $\blacktriangleright$ {\sc Case ${\rm b}_1\neq 0$, $\eta > \nu$ and $\la \tW_{\perp\perp}\pa_y^{-1}U,\pa_y^{-1}U\ra>\frac{(\nu+\eta)\nu\eta}{\eta-\nu}$.} 
    We define the sets
    \begin{align}
        \cU_\varepsilon & := \{ \vf=\vf_0 + \vf_\perp \in H^s \times H^s \, : \, T_3\vf_0 +  \vf_\perp =  0 \in \bH_0^s \}\,. \label{U.unstab.neq0} \\
        \cS_\varepsilon & := \{ \vf=\vf_0 + \vf_\perp \in H^s\times H^s \, : \,\vf_0 + T_2 \vf_\perp = 0 \in \C^2 \} \,.  \label{S.stab.neq0} 
    \end{align}
Together with the invertibility of the map $\tT$ in Proposition \ref{prop:fixed_point}-$(ii)$, we deduce that
\begin{equation}
    H^s \times H^s = \tT^{-1}(H^s \times H^s) = \tT^{-1} \big( \C^2 \oplus \bH_0^s \big) = \cU_{\varepsilon} \oplus \cS_{\varepsilon} \,,
\end{equation}
which shows \eqref{stable.unstable.dec}. We now prove the estimates on the evolution of the initial data:
\begin{itemize}
\item Let us assume that $\vf =\vf_0 + \vf_\perp \in \cU_{\varepsilon}$. In view of \eqref{U.unstab.neq0} and \eqref{initial.data.split},
the new initial conditions $(\phi_0,\phi_{\perp})$ satisfy the following: first we have
\begin{equation}\label{mandala11}
\begin{aligned}
\phi_{\perp}\stackrel{\eqref{initial.data.split}}{=}T_3 \varphi_0+\vf_\perp \stackrel{\eqref{U.unstab.neq0}}{=}0\,,
\end{aligned}
\end{equation}
which also implies 
\begin{equation}\label{mandala12}
\vf_\perp=-T_3 \varphi_0\,.
\end{equation}
Secondly, we have
\begin{equation}\label{mandala13}
\phi_0\stackrel{\eqref{initial.data.split}}{=}
\vf_0+T_2\vf_{\perp}\stackrel{\eqref{mandala12}}{=}
({\rm Id}_0 - T_2T_3) \vf_0 \in \C^2\,.
\end{equation}
Notice that, in view of Proposition \ref{prop:fixed_point}, Remark 
\ref{rmk:order_T2.T3} and \eqref{mandala13} we also have
\begin{equation}\label{mandala14}
|\phi_0|\simeq |\vf_0|\,.
\end{equation}
Moreover, the two Cauchy problem in 
\eqref{two.Cauchy}, with initial conditions
\eqref{mandala11}-\eqref{mandala13} respectively,
have solutions
\begin{align}
v_0(t)&={\rm exp}(M_{\nu,\eta}^{(0)}t)\phi_0\in \C^2\,,
\;\;\; \forall \, t\geq0\,,\label{sol.v0}
\\
v_{\perp}(t)&\equiv0\,,\;\;\;\quad\qquad\qquad   
\qquad \;\;\forall\, t\geq0\,,\label{sol.vperp}
\end{align}
Now, recalling  Proposition \ref{prop:positive.zeromode}-$(i)$, 
we have that the $2\times 2$ matrix $M_{\nu,\eta}^{(0)}$ in \eqref{eq:1st_entrance_resist} 
has two distinct eigenvalues 
$\mu_{+},\mu_{-}\in\C \setminus\{0\} $ 
satisfying ${\rm Re}(\mu_{\pm}) = \e^2 \mathtt c + O(\e^3) \geq \e^2\mathtt c /2$ for $0 < \e \ll 1$ small enough. 
Then, by applying Lemma \ref{dinamica sistemino due per due}-$(i)$, one has that the solution \eqref{sol.v0} 
satisfies the estimates
\begin{equation}\label{stima.v0}
    | v_0(t) | \gtrsim  e^{ \e^2 (\mathtt c /2) t} |\phi_0| , \quad \forall \, t\geq 0 \,, \ \ \phi_{0} \in \C^2 \,,
\end{equation}
for some constant $C_1>0$. 
 We now come back to the original variables to estimate 
 from below
 the norm of 
 the solution $u(t)$
 of \eqref{change.vTu}.
By  \eqref{eq:ansatz_conjugation} in Proposition \ref{prop:fixed_point}
we recall that
\[
 \begin{pmatrix}
        v_0(t) \\ v_\perp(t)
    \end{pmatrix} := \tT \begin{pmatrix}
        u_0(t) \\ u_\perp(t)
    \end{pmatrix}
    \stackrel{\eqref{eq:ansatz_conjugation}}{=}
    \begin{pmatrix}
    u_0(t)+T_2 u_{\perp}(t)
    \\
    T_3u_0(t)+u_{\perp}(t)
    \end{pmatrix}\,,
\]
which, by estimates \eqref{basicestT}, implies 
\begin{align*}
|v_0(t)|&\lesssim |u_0(t)|+\e^3\|u_{\perp}(t)\|_s\,,
\\
\|v_{\perp}(t)\|_{s}&\lesssim \e^{-1}|u_0(t)|+\|u_\perp(t)\|{_s}\,.
\end{align*}
 In conclusion we get
\begin{align}
    \| u(t)\|_{s} 
    & \simeq |u_0(t)| + \| u_{\perp}(t) \|_{s} 
    \gtrsim |u_0(t)| +\e^3\|u_\perp(t)\|_s
    \gtrsim |v_0(t)|
    \\
     & \stackrel{\eqref{stima.v0}}{\gtrsim} 
   e^{\e^2 (\mathtt c /2))t} |\phi_0| 
   \stackrel{\eqref{mandala14}}{\gtrsim}
e^{\e^2 (\mathtt c /2))t} |\vf_0|\,.
\end{align}
This  proves \eqref{unstable.est.thm} with 
$\mu = \mathtt c /2 > 0$ (see \eqref{explago1}).


\item 
We now  assume that $\vf =\vf_0 + \vf_\perp \in \cS_{\varepsilon}$. In view of \eqref{S.stab.neq0} and \eqref{initial.data.split},
the new initial conditions $(\phi_0,\phi_{\perp})$ satisfy the following: first we have
\begin{equation}\label{mandala11Bis}
\begin{aligned}
\phi_{0}\stackrel{\eqref{initial.data.split}}{=}\varphi_0+T_2\vf_\perp \stackrel{\eqref{S.stab.neq0}}{=}0\,,
\end{aligned}
\end{equation}
which also implies 
\begin{equation}\label{mandala12Bis}
\vf_0=-T_2 \varphi_\perp\,.
\end{equation}
Secondly, we have
\begin{equation}\label{mandala13Bis}
\phi_\perp\stackrel{\eqref{initial.data.split}}{=}
T_3\vf_0+\vf_{\perp}\stackrel{\eqref{mandala12Bis}}{=}
({\rm Id}_0 - T_3T_2) \vf_\perp \in {\bf H}_0^s\,.
\end{equation}
Notice that,
in view of Proposition \ref{prop:fixed_point}, Remark 
\ref{rmk:order_T2.T3} and \eqref{mandala13Bis} we also have
\begin{equation}\label{mandala14Bis}
\|\phi_\perp\|_{s}\simeq \|\vf_\perp\|_s\,.
\end{equation}

Therefore, the two Cauchy problem in 
\eqref{two.Cauchy}, with initial conditions
\eqref{mandala11Bis}-\eqref{mandala13Bis} respectively, have solutions,
using Lemma \ref{stability dynamics}, satisfying 
\begin{align}
v_{0}(t)&\equiv0\,,\qquad \forall\, t\geq 0\,,\label{stima.v0eq0}
\\
    \| v_\perp (t) \|_{s} &\leq C e^{-\frac{\sigma}{2}t} 
    \| \phi_{\perp} \|_{s} \,, 
    \quad  \forall \, t \geq 0 \,, \ \ 
    \phi_\perp \in \bH_0^s\,,\label{stima.vperp}
\end{align}
for some constant $C>0$.  
 We now come back to the original variables to estimate 
 from above
 the norm of 
 the solution $u(t)$
 of \eqref{change.vTu}.
 Recalling \eqref{tT-1} in Proposition \ref{prop:fixed_point}
 we deduce 
\begin{equation}\label{mandala15}
\begin{aligned}
 \begin{pmatrix}
         u_0(t) \\ u_\perp(t)
     \end{pmatrix} & =\tT^{-1}
     \begin{pmatrix}
    v_0(t) \\ v_\perp(t)
     \end{pmatrix} \stackrel{\eqref{tT-1}}{=}
     \begin{pmatrix}
 ({\rm Id_0} -T_2T_3)^{-1}v_0-T_2({\rm Id_\perp} -T_3T_2)^{-1}v_{\perp}
 \\
 ({\rm Id_\perp} -T_3T_2)^{-1}v_{\perp}-T_3({\rm Id_0} -T_2T_3)^{-1}v_0
     \end{pmatrix} \\
     & = \begin{pmatrix}
-T_2({\rm Id_\perp} -T_3T_2)^{-1}v_{\perp}
 \\
 ({\rm Id_\perp} -T_3T_2)^{-1}v_{\perp}
     \end{pmatrix} \,.
     \end{aligned}
 \end{equation}
 By estimates of  Remark \ref{rmk:order_T2.T3} (see also Proposition \ref{prop:fixed_point}), we get
\begin{equation}\label{stimaaltoU}
\begin{aligned}
|u_0(t)|&\lesssim |v_0(t)|+\e^3\|v_{\perp}(t)\|_s \lesssim \e^3\|v_{\perp}(t)\|_s \,,
\\
\|u_{\perp}(t)\|_s&\lesssim \e^{-1}|v_0(t)|+\|v_{\perp}(t)\|_s  \stackrel{\eqref{stima.v0eq0}}{\lesssim} \|v_{\perp}(t)\|_s\,,
\end{aligned}
\end{equation}
for any $t\geq0$.
Hence, we conclude that
\begin{align}
\| u(t)\|_{s} & \simeq 
|u_0(t)| + \| u_{\perp}(t) \|_{s} 
\stackrel{\eqref{stimaaltoU}}{\lesssim}
\e^{-1}|v_0(t)| + \| v_{\perp}(t) \|_{s} 
\\& 
\stackrel{\eqref{stima.v0eq0}}{=} 0 
+ \| v_{\perp}(t) \|_{s}  
\stackrel{\eqref{stima.vperp}}{\lesssim} e^{-\frac{\sigma}{2}t} \|\phi_\perp \|_{s} \\
     & \stackrel{\eqref{mandala14Bis}}{\lesssim} e^{-\frac{\sigma}{2}t} 
     \|\vf_\perp \|_{s} 
     \lesssim 
     e^{-\frac{\sigma}{2}t} \|\vf \|_{s} \,,
\end{align}
which proves \eqref{stable.est.thm} with $\mu_0 = \sigma/2$.
\end{itemize}

\bigskip

\noindent $\blacktriangleright$ {\sc Case ${\rm b}_1\neq 0$, $\nu \geq \eta$.} 
We recall \eqref{change.vTu}-\eqref{initial.data.split}. One has that the systems \eqref{two.Cauchy} have solutions $v(t) = v_0(t) + v_\bot (t)$ with 
$$
v_0(t) = {\rm exp}(M_{\nu, \eta}^{(0)} t)\phi_0 
$$
and by Lemma \ref{stability dynamics}
\begin{equation}\label{stima v bot caso b1 neq 0 stable}
\| v_\bot(t) \|_{s} \lesssim e^{- \frac{\sigma}{2} t} \| \phi_\bot \|_s, \quad \forall t \geq 0\,. 
\end{equation}
By Proposition \ref{prop:positive.zeromode}-$(iii)$ ${\rm Re}(\mu_{\pm}) = - \e^2 \mathtt c + O(\e^3) \leq - \e^2 \frac{\mathtt c}{2}$ for $\e \ll 1$ small enough (for some constant $\mathtt c > 0$). By Lemma \ref{dinamica sistemino due per due}-$(iii)$, one has that
\begin{equation}\label{stima v0 caso b1 neq stabile}
|v_0(t)| \lesssim e^{- \frac{\mathtt c}{2}\e^2 t} |\phi_0| , \quad \forall t \geq 0\,. 
\end{equation}
 Recalling Proposition \ref{prop:fixed_point} 
 we deduce 
\begin{equation}\label{mandala152}
\begin{aligned}
 \begin{pmatrix}
         u_0(t) \\ u_\perp(t)
     \end{pmatrix} & =\tT^{-1}
     \begin{pmatrix}
    v_0(t) \\ v_\perp(t)
     \end{pmatrix} \stackrel{\eqref{tT-1}}{=}
     \begin{pmatrix}
 ({\rm Id_0} -T_2T_3)^{-1}v_0-T_2({\rm Id_\perp} -T_3T_2)^{-1}v_{\perp}
 \\
 ({\rm Id_\perp} -T_3T_2)^{-1}v_{\perp}-T_3({\rm Id_0} -T_2T_3)^{-1}v_0
     \end{pmatrix}, \\
     \begin{pmatrix}
        \phi_0 \\ \phi_\perp
    \end{pmatrix} & := \tT \begin{pmatrix}
        \vf_0 \\ \vf_\perp
    \end{pmatrix} \quad \Leftrightarrow \quad   \begin{cases}
        \phi_0 = \vf_0 + T_2 \vf_\perp \,, \\
        \phi_\perp = T_3 \vf_0 + \vf_\perp \,,
    \end{cases} 
     \end{aligned}
 \end{equation}
 with estimates 
 \begin{equation}\label{stime caso stabile b1 neq 0}
\begin{aligned}
|u_0(t)| & \lesssim |v_0(t)|+ \e^3 \| v_\bot(t) \|_s ,\qquad \| u_\bot(t) \|_s  \lesssim \| v_\bot(t) \|_s + \e^{- 1} |v_0(t)|, \\
|\phi_0| & \lesssim |\vf_0| +  \e^2 \| \vf_\bot \|_s , \qquad \| \phi_\bot \|_s  \lesssim \e^{- 1} |\vf_0| + \| \vf_\bot \|_s\,. 
\end{aligned}
 \end{equation}
Hence for $0 < \e \ll 1 $ small  enough (use that $e^{- (\sigma /2) t} \leq e^{- \e^2(\mathtt c/2) t}$), collecting \eqref{stima v bot caso b1 neq 0 stable}, \eqref{stima v0 caso b1 neq stabile}, \eqref{stime caso stabile b1 neq 0} one gets 
\begin{equation}\label{stima u0 caso b1 stabile}
\begin{aligned}
|u_0(t)| &  \lesssim |v_0(t)|+ \e^3 \| v_\bot(t) \|_s \lesssim e^{- \frac{\mathtt c}{2}\e^2 t} |\phi_0| + \e^3 e^{- (\sigma /2) t} \| \phi_\bot \|_s \\
& \lesssim e^{- \frac{\mathtt c}{2}\e^2 t} (|\vf_0| +  \e^2 \| \vf_\bot \|_s) + \e^3 e^{- (\sigma /2) t} (\e^{- 1} |\vf_0| + \| \vf_\bot \|_s) \\
& \lesssim e^{- \frac{\mathtt c}{2}\e^2 t} (|\vf_0| + \| \vf_\bot \|_s) \lesssim e^{- \frac{\mathtt c}{2}\e^2 t} \| \vf \|_s
\end{aligned}
\end{equation}
and 
\begin{equation}\label{stima u bot caso b1 stabile}
\begin{aligned}
\|u_\bot(t)\|_s &  \lesssim \| v_\bot(t) \|_s + \e^{- 1} |v_0(t)| \lesssim   e^{- (\sigma /2) t} \| \phi_\bot \|_s + \e^{- 1} e^{- \frac{\mathtt c}{2}\e^2 t} |\phi_0| \\
& \lesssim e^{- (\sigma /2) t} \big( \e^{- 1} |\vf_0| + \| \vf_\bot \|_s \big) + \e^{- 1} e^{- \frac{\mathtt c}{2}\e^2 t} \big( |\vf_0| +  \e^2 \| \vf_\bot \|_s \big) \\
& \lesssim \e^{- 1} e^{- \frac{\mathtt c}{2}\e^2 t} |\vf_0| + e^{- \frac{\mathtt c}{2}\e^2 t} \| \vf_\bot \|_s \lesssim \e^{-1} e^{- \frac{\mathtt c}{2}\e^2 t} \| \vf\|_s  \,.
\end{aligned}
\end{equation}
This concludes the proof of the case ${\rm b}_1 \neq 0$ since \eqref{stima b1 neq 0 caso stabile} follows from the latter two estimates \eqref{stima u0 caso b1 stabile}, \eqref{stima u bot caso b1 stabile} by setting $\mu_3 = \mathtt c /2$. 
\\[1mm]
\noindent $\blacktriangleright$ {\sc Case ${\rm b}_1= 0$.} The proof in this case runs with similar arguments with some modifications. We insert the details for completeness. 
Under the assumptions of item $(ii)$ in Theorem \ref{theo.princ}, recalling also Proposition \ref{prop:positive.zeromode}, we have that the $2\times 2$ matrix $M_{\nu,\eta}^{(0)}$ in \eqref{eq:1st_entrance_resist} has two distinct eigenvalues $\mu_{+},\mu_{-}\in\C \setminus\{0\} $ such that ${\rm Re}(\mu_+) = \mathtt c_1\e^2 + O(\e^3)$, ${\rm Re}(\mu_-) = - \mathtt c_2 \e^2 + O(\e^3)$ for some constants $\mathtt c_1, \mathtt c_2 > 0$. By Lemma \ref{dinamica sistemino due per due}-$(ii)$, the system for $v_0$ in \eqref{two.Cauchy} satisfies that there are two one dimensional subspaces $E_{\pm} \subset \C^2$ such that $\C^2 = E_+ \oplus E_-$ and such that
\begin{equation}\label{v0 phi E pm nel lem}
\begin{aligned}
& |v_0 (t)| \gtrsim e^{\e^2 (\mathtt c_1/2) t} |\phi_0| \quad \forall \,\phi_0 \in E_+ \quad \forall \, t \geq 0, \\
& |v_0(t)| \lesssim e^{- \e^2 (\mathtt c_2/2) t} |\phi_0| \quad \forall \, \phi_0 \in E_- \quad \forall \, t \geq 0.
\end{aligned}
\end{equation}

We now define the sets
    \begin{align}
        \cU_\varepsilon & := \{ \vf=\vf_0 + \vf_\perp \in H^s \times H^s \, : \,  \vf_0 + T_2 \vf_\perp  \in E_+  \,, \ \   T_3\vf_0 +  \vf_\perp =  0    \} \,, \label{U.unstab.0a} 
        \\
        \cS_\varepsilon & := \{ \vf=\vf_0 + \vf_\perp \in H^s\times H^s \, : \,   \vf_0 + T_2 \vf_\perp \in  E_-  \} \,.  \label{S.stab.0a} 
    \end{align}
    \begin{itemize}
\item Let us assume that $\vf =\vf_0 + \vf_\perp \in \cU_{\varepsilon}$. In view of \eqref{initial.data.split} and \eqref{v0 phi E pm nel lem}
the new initial conditions $(\phi_0,\phi_{\perp})$ satisfy the following: first we have
\begin{equation}\label{mandala11a}
\begin{aligned}
\phi_{\perp}\stackrel{\eqref{initial.data.split}}{=}T_3 \varphi_0+\vf_\perp \stackrel{\eqref{U.unstab.neq0}}{=}0, \quad \phi_0 \in E_+\,,
\end{aligned}
\end{equation}
which also implies 
\begin{equation}\label{mandala12a}
\vf_\perp=-T_3 \varphi_0\,.
\end{equation}
Secondly, we have
\begin{equation}\label{mandala13a}
\phi_0\stackrel{\eqref{initial.data.split}}{=}
\vf_0+T_2\vf_{\perp}\stackrel{\eqref{mandala12a}}{=}
({\rm Id}_0 - T_2T_3) \vf_0 \in E_+\,.
\end{equation}
Notice that, in view of Proposition \ref{prop:fixed_point}, Remark 
\ref{rmk:order_T2.T3} and \eqref{mandala13a} we also have
\begin{equation}\label{mandala14a}
|\phi_0|\simeq |\vf_0|\,.
\end{equation}
and by \eqref{v0 phi E pm nel lem}
$$
|v_0(t)| \gtrsim e^{\e^2(\mathtt c_1/2) t} |\phi_0| \quad \forall t \geq 0\,. 
$$

 We now come back to the original variables to estimate 
 from below
 the norm of 
 the solution $u(t)$
 of \eqref{change.vTu}.
By  \eqref{eq:ansatz_conjugation} in Proposition \ref{prop:fixed_point}
we recall that
\[
 \begin{pmatrix}
        v_0(t) \\ v_\perp(t)
    \end{pmatrix} := \tT \begin{pmatrix}
        u_0(t) \\ u_\perp(t)
    \end{pmatrix}
    \stackrel{\eqref{eq:ansatz_conjugation}}{=}
    \begin{pmatrix}
    u_0(t)+T_2 u_{\perp}(t)
    \\
    T_3u_0(t)+u_{\perp}(t)
    \end{pmatrix}\,,
\]
which, by estimates \eqref{basicestT}, implies 
\begin{align*}
|v_0(t)|&\lesssim |u_0(t)|+\e^3\|u_{\perp}(t)\|_s\,,
\\
\|v_{\perp}(t)\|_{s}&\lesssim \e^{-1}|u_0(t)|+\|u_\perp(t)\|{_s}\,.
\end{align*}
 In conclusion we get
\begin{align}
    \| u(t)\|_{s} 
    & \simeq |u_0(t)| + \| u_{\perp}(t) \|_{s} 
    \gtrsim |u_0(t)| +\e^3\|u_\perp(t)\|_s
    \gtrsim |v_0(t)|
    \\
     & \stackrel{\eqref{stima.v0}}{\gtrsim} 
   e^{\e^2 (\mathtt c_1 /2))t} |\phi_0| 
   \stackrel{\eqref{mandala14}}{\gtrsim}
e^{\e^2 (\mathtt c_1 /2))t} |\vf_0|\,.
\end{align}
This  proves \eqref{unstable.est.thmcaso2} with 
$\mu_1 = \mathtt c_1 /2 $ (see \eqref{explago1}).


\item 
We now  assume that $\vf =\vf_0 + \vf_\perp \in \cS_{\varepsilon}$, recall \eqref{S.stab.0a}. In view of \eqref{S.stab.neq0} and \eqref{initial.data.split},
the new initial conditions $(\phi_0,\phi_{\perp})$ satisfy the following: first we have
\begin{equation}\label{mandala11TRIS}
\begin{aligned}
\phi_{0}\stackrel{\eqref{initial.data.split}}{=}\varphi_0+T_2\vf_\perp \in E_- . 
\end{aligned}
\end{equation}
and hence in view of \eqref{v0 phi E pm nel lem} and Lemma \ref{stability dynamics} one has that
\begin{equation}\label{stima v ultimo cao}
|v_0(t)| \lesssim e^{- (\mathtt c_2 /2)\e^2 t} |\phi_0|\,, \quad \| v_\bot (t) \|_s \lesssim e^{- (\sigma/2) t} \| \phi_\bot \|_s, \quad \forall t \geq 0\,. 
\end{equation}

 Recalling Proposition \ref{prop:fixed_point} 
 we deduce 
\begin{equation}\label{mandala152A}
\begin{aligned}
 \begin{pmatrix}
         u_0(t) \\ u_\perp(t)
     \end{pmatrix} & =\tT^{-1}
     \begin{pmatrix}
    v_0(t) \\ v_\perp(t)
     \end{pmatrix} \stackrel{\eqref{tT-1}}{=}
     \begin{pmatrix}
 ({\rm Id_0} -T_2T_3)^{-1}v_0-T_2({\rm Id_\perp} -T_3T_2)^{-1}v_{\perp}
 \\
 ({\rm Id_\perp} -T_3T_2)^{-1}v_{\perp}-T_3({\rm Id_0} -T_2T_3)^{-1}v_0
     \end{pmatrix}, \\
     \begin{pmatrix}
        \phi_0 \\ \phi_\perp
    \end{pmatrix} & := \tT \begin{pmatrix}
        \vf_0 \\ \vf_\perp
    \end{pmatrix} \quad \Leftrightarrow \quad   \begin{cases}
        \phi_0 = \vf_0 + T_2 \vf_\perp \,, \\
        \phi_\perp = T_3 \vf_0 + \vf_\perp \,,
    \end{cases} 
     \end{aligned}
 \end{equation}
 with estimates 
 \begin{equation}\label{stime caso stabile b1 neq 0A}
\begin{aligned}
|u_0(t)| & \lesssim |v_0(t)|+ \e^3 \| v_\bot(t) \|_s ,\qquad \| u_\bot(t) \|_s  \lesssim \| v_\bot(t) \|_s + \e^{- 1} |v_0(t)|, \\
|\phi_0| & \lesssim |\vf_0| +  \e^3 \| \vf_\bot \|_s , \qquad \| \phi_\bot \|_s  \lesssim \e^{- 1} |\vf_0| + \| \vf_\bot \|_s\,. 
\end{aligned}
 \end{equation}
 Hence, collecting \eqref{stima v ultimo cao}, \eqref{mandala152A}, \eqref{stime caso stabile b1 neq 0A} (using that $e^{- (\mathtt c_2/2)\e^2 t} \leq e^{- (\sigma /2) t}$ for $0 < \e \ll 1$ small enough) one gets 
\begin{equation}\label{stima u0 caso b1 stabileA}
\begin{aligned}
|u_0(t)| &  \lesssim |v_0(t)|+ \e^3 \| v_\bot(t) \|_s \lesssim e^{- (\mathtt c_2/2)\e^2 t} |\phi_0| + \e^3 e^{- (\sigma /2) t} \| \phi_\bot \|_s \\
& \lesssim e^{- (\mathtt c_2/2)\e^2 t} (|\vf_0| +  \e^2 \| \vf_\bot \|_s) + \e^3 e^{- (\sigma /2) t} (\e^{- 1} |\vf_0| + \| \vf_\bot \|_s) \\
& \lesssim e^{- (\mathtt c_2/2)\e^2 t} (|\vf_0| + \| \vf_\bot \|_s) \lesssim e^{- (\mathtt c_2/2)\e^2 t} \| \vf \|_s
\end{aligned}
\end{equation}
and 
\begin{equation}\label{stima u bot caso b1 stabileA}
\begin{aligned}
\|u_\bot(t)\|_s &  \lesssim \| v_\bot(t) \|_s + \e^{- 1} |v_0(t)| \lesssim   e^{- (\sigma /2) t} \| \phi_\bot \|_s + \e^{- 1} e^{- (\mathtt c_2/2)\e^2 t} |\phi_0| \\
& \lesssim e^{- (\sigma /2) t} \big( \e^{- 1} |\vf_0| + \| \vf_\bot \|_s \big) + \e^{- 1} e^{- (\mathtt c_2/2)\e^2 t} \big( |\vf_0| +  \e^2 \| \vf_\bot \|_s \big) \\
& \lesssim \e^{- 1} e^{- (\mathtt c_2/2)\e^2 t} |\vf_0| + e^{- (\mathtt c_2/2)\e^2 t} \| \vf_\bot \|_s \lesssim \e^{-1} e^{- (\mathtt c_2/2)\e^2 t} \| \vf \|_s \,.
\end{aligned}
\end{equation}
This concludes the proof of the case ${\rm b}_1 = 0$ since \eqref{stable.est.thmcaso2}  follows from the latter two estimates \eqref{stima u0 caso b1 stabileA}, \eqref{stima u bot caso b1 stabileA}.
\end{itemize}
The proof of Theorem \ref{thm:dinamico} is then concluded.
\end{proof}

\section*{ Acknowledgements.} 
R. Montalto  
is supported by the ERC STARTING GRANT 2021 ``Hamiltonian Dynamics,  Normal Forms and Water Waves'' (HamDyWWa), Project Number: 101039762. 
Views and opinions expressed are however those of the authors only and do not necessarily reflect those of the European Union or the European Research Council. Neither the European Union nor the granting authority can be held responsible for them. 
R. Feola is supported  by ``GNAMPA - INdAM'', CUP E53C25002010001, and  ``GNAMPA - INdAM'', CUP E5324001950001. 
L. Franzoi is supported by ``GNAMPA - INdAM'', CUP E53C25002010001, and has also been supported by the ERC STARTING GRANT 2021 ``Hamiltonian Dynamics,  Normal Forms and Water Waves'' (HamDyWWa), Project Number: 101039762, until January 2026. 
C. Pe\~na is supported by the project PID2021-123034NBI00 funded by MICIU/AEI/10.13039/501100011033 and FEDER/EU, and has also partially been supported by the  ERC STARTING GRANT 2021 ``Hamiltonian Dynamics, 
Normal Forms and Water Waves'' (HamDyWWa), Project Number: 101039762.

\section*{Statements and declarations}

\medskip

\noindent
The authors state that there is no conflict of interest and certify that they  have no affiliations or involvement with any organization or entity with any
financial interest, or non-financial interest in the subject matter or materials discussed in this manuscript.

Moreover, data sharing is not applicable to this article as no datasets were generated or analyzed during the
current study.

\begin{footnotesize}
	
\end{footnotesize}

\end{document}